\documentclass[10pt,a4paper]{article}

\usepackage{amsmath,amsthm,amsfonts,amssymb}
\usepackage{accents}
\usepackage{cite}
\usepackage[usenames,dvipsnames,svgnames,table]{xcolor}
\usepackage{pgf,tikz}\usetikzlibrary{matrix, calc, arrows}
\usepackage[bookmarks=true, bookmarksopen=true,bookmarksopenlevel=5]{hyperref} 
\definecolor{myblue}{rgb}{0,0,0.6}     
\hypersetup{pdftitle={Space-time discontinuous Galerkin approximation of acoustic waves with point singularities},  
     colorlinks=true, linkcolor=myblue,  citecolor=myblue, filecolor=myblue,   urlcolor=myblue,  }
\usepackage{enumitem}
\usepackage{subcaption}
\usepackage{mathtools}
\usepackage[ruled]{algorithm2e}
\usepackage{soul}
\makeatletter
\renewcommand{\@algocf@capt@plain}{above}
\makeatother
\allowdisplaybreaks[4]

\newcommand*{\der}[2]{\frac{\partial #1}{\partial #2}}
\newcommand{\di}{\,\mathrm{d}}
\newcommand{\iin}{\;\text{in}\;}
\newcommand{\oon}{\;\text{on}\;}
\newcommand{\deO}{{\partial\Omega}}
\newcommand*{\conj}[1]{\overline{#1}}
\newcommand*{\N}[1]{\left\|#1\right\|}
\newcommand*{\abs}[1]{\left|#1\right|}
\newcommand*{\jmp}[1]{[\![#1]\!]}
\newcommand*{\mvl}[1]{\{\!\!\{#1\}\!\!\}}
\newcommand{\dive} {\mathop{\rm div}\nolimits}
\def\div{\mathop{\rm div}\nolimits}
\DeclareMathOperator{\diam}{diam}
\DeclareMathOperator{\dist}{dist}

\newcommand{\uu}[1]{\hbox{\boldmath$#1$}}
\newcommand{\Uu}[1]{{\mathbf{#1}}}  
\newcommand{\IF}{\mathbb{F}}
\newcommand{\IN}{\mathbb{N}}\newcommand{\IP}{\mathbb{P}}
\newcommand{\IR}{\mathbb{R}}
\newcommand{\bc}{{\Uu c}}\newcommand{\bn}{{\Uu n}}
\newcommand{\bp}{{\Uu p}}\newcommand{\bx}{{\Uu x}}
\newcommand{\bz}{{\Uu z}}
\newcommand{\bL}{{\Uu L}}\newcommand{\bN}{{\Uu N}}\newcommand{\bV}{{\Uu V}}
\newcommand{\bdelta}{{\uu\delta}}
\newcommand{\bsigma}{{\boldsymbol \sigma}}\newcommand{\btau}{{\boldsymbol\tau}}	
\newcommand{\bzero}{\Uu{0}}
\newcommand{\udelta}{{\mbox{\scriptsize $\uu{\delta}$}}}  
\newcommand{\calA}{{\mathcal A}}
\newcommand{\calE}{{\mathcal E}}\newcommand{\calF}{{\mathcal F}}
\newcommand{\calK}{{\mathcal K}}
\newcommand{\calS}{{\mathcal S}}\newcommand{\calT}{{\mathcal T}}
\newcommand{\malpha}{{\boldsymbol{\alpha}}}
\newcommand{\Hdiv}{H(\div;\Omega)}
\newcommand{\Hodiv}{H_0(\div;\Omega)}
\newcommand{\tta}{{\tt a}}
\newcommand{\ttb}{{\tt b}}
\newcommand{\OO}{{(\Omega)}}

\newcommand{\half}{\frac{1}{2}}
\newcommand{\LtO}{L^2(\Omega)}
\newcommand{\Fh}{\calF_h}
\newcommand{\Th}{{(\calT_h)}}

\newcommand{\deK}{{\partial K}}
\newcommand{\GD}{{\Gamma_{\mathrm D}}}
\newcommand{\GN}{{\Gamma_{\mathrm N}}}
\newcommand{\FT}{{\Fh^T}}
\newcommand{\FO}{{\Fh^0}}
\newcommand{\FD}{{\Fh^{\mathrm D}}}
\newcommand{\FN}{{\Fh^{\mathrm N}}}
\newcommand{\Fspa}{{\Fh^{\mathrm{space}}}}
\newcommand{\Ftime}{{\Fh^{\mathrm{time}}}}
\newcommand{\hp}{_h}
\newcommand\Vhp{v\hp}
\newcommand\Shp{\bsigma\hp}
\newcommand\hVhp{\widehat v\hp}
\newcommand\hShp{\widehat \bsigma\hp}
\newcommand*{\Norm}[1]{\left\|#1\right\|}
\newcommand{\DG}{_{\mathrm{DG}}}
\newcommand{\DGp}{_{\mathrm{DG^+}}}

\newcommand{\DGQn}{_{\mathrm{DG}(Q_n)}}
\newcommand{\DGQnp}{_{\mathrm{DG}(Q_n)^+}}
\newcommand{\bVp}{{\bV_\bp}}
\newcommand{\cM}{{\mathcal M}}
\newcommand{\vs}{{(v,\bsigma)}}
\newcommand{\vsh}{{(v\hp,\bsigma\hp)}}
\newcommand{\wt}{{(w,\btau)}}
\newcommand{\wth}{{(w\hp,\btau\hp)}}
\newcommand{\Ftn}{{\calF_h^{t_n}}}

\newcommand{\IFtime}{\IF^{\mathrm{time}}}
\newcommand{\IFD}{\IF^{\mathrm D}}
\newcommand{\IFN}{\IF^{\mathrm N}}
\newcommand{\gD}{{g_{\mathrm D}}}
\newcommand{\gN}{{g_{\mathrm N}}}
\newcommand{\wK}{{\widehat K}}
\newcommand{\wI}{{\widehat I}}
\newcommand{\wPi}{{\widehat\Pi}}
\newcommand{\wbx}{{\widehat \bx}}

\newcommand{\bll}{\mathbf{l}}

\newcommand{\mhalf}{{-\half}}
\newcommand{\hsp}{{h_\bx}}
\newcommand{\calTspace}{\calT_\hsp^\bx}
\newcommand{\hspn}{{h_{\bx,n}}}
\newcommand{\calTspacen}{\calT_{h_{\bx,n}}^\bx}
\newcommand{\htime}{{h_t}}
\newcommand{\calTtime}{\calT_\htime^t}
\newcommand{\calTspaceinit}{\calT_{0}^\bx}
\newcommand{\calTspacemarked}{\calT^\bx}
\newcommand{\whvarphi}{\widehat{\varphi}}
\newcommand{\whpsi}{\widehat{\psi}}
\newcommand{\JJ}{{J}}

\newtheorem{theorem}{Theorem}[section]
\newtheorem{lemma}[theorem]{Lemma}
\newtheorem{defin}[theorem]{Definition}
\newtheorem{proposition}[theorem]{Proposition}
\newtheorem{cor}[theorem]{Corollary}
\newtheorem{remark}[theorem]{Remark}

\begin{document}
\title{Space--time discontinuous Galerkin approximation\\
        of acoustic waves with point singularities} 
\author{Pratyuksh Bansal\thanks{SAM ETH Z\"urich,
R\"amistrasse 101, CH-8092 Z\"urich
    (pratyuksh.bansal@sam.math.ethz.ch)},
  Andrea Moiola\thanks{Department of Mathematics, University of Pavia,
    Via Ferrata 5, I-27100 Pavia
    (andrea.moiola@unipv.it)},
Ilaria Perugia\thanks{Faculty of Mathematics, University of Vienna,
Oskar-Morgenstern-Platz 1, A-1090 Vienna
  (ilaria.perugia@univie.ac.at)},
  Christoph Schwab\thanks{SAM ETH Z\"urich
R\"amistrasse 101, CH-8092 Z\"urich,
   (christoph.schwab@sam.math.ethz.ch)}}
\maketitle
\begin{abstract}
We develop a convergence theory of space--time discretizations for the linear,
2nd-order wave equation in 
polygonal domains $\Omega \subset {\mathbb R}^2$,
possibly occupied by piecewise homogeneous media with 
different propagation speeds.
Building on an unconditionally stable space--time DG formulation developed in~\cite{MoPe18}, we 
(a) prove optimal convergence rates for the space--time scheme with 
local isotropic corner mesh refinement on the spatial domain, and
(b) demonstrate numerically optimal convergence rates
of a suitable \emph{sparse} space--time version of the DG scheme.
The latter scheme is based on the so-called \emph{combination formula}, 
in conjunction with a family of anisotropic space--time DG-discretizations.
It results in optimal-order convergent schemes, also in domains with corners,
with a number of degrees of freedom that scales essentially like the DG solution
of one stationary elliptic problem in $\Omega$ on the finest spatial grid.
Numerical experiments for both smooth and singular solutions
support convergence rate optimality on spatially refined meshes of
the full and sparse space--time DG schemes.
\end{abstract}

\begin{center}
\textbf{Dedicated to the memory of John W.~Barrett}
\end{center}

\section{Introduction}
\label{sec:Intro}
In recent years, substantial interest and progress has been made on 
so-called \emph{space--time discretizations of evolution equations}.
This development has been driven by several factors. We mention only
the need to compute the evolution of the solution over a
finite time interval $(0,T)$ rather than
a sequence of spatial solutions resulting in the solution only at the 
final time horizon $T$. 
This issue arises for example in optimal control of
parabolic and hyperbolic PDEs, where the feedback depends on the solution
of an adjoint PDE which is driven by a functional of the forward solution.
Similar issues and needs arise in space--time optimization of such PDEs,
and in particular in space--time adaptivity.
Here, traditional time-marching schemes encounter several difficulties
which can be easily overcome by combined space--time, ``monolithic'' discretizations.
This is purchased, however, with a number of new issues: 
increased memory requirement for the simultaneous storage of the entire
solution history over $(0,T)$, efficient solvers, etc. 
In the present paper, we address a class of space--time discretizations of
the linear, acoustic wave equation in two space dimensions, in polygonal
domains which are occupied by possibly heterogeneous media with 
corresponding variable wave speeds. 
Acoustic (and other types of) waves exhibit diffraction at conical singularities
such as corners or multi-material interfaces (e.g.\ \cite{KEL62,LBA02,Borovikov}).
It has been clarified in the past decades
(e.g.\ \cite{KokPlam2004,KNP_Dir99,KoPlaMixed} and the references there)
that these singularities are, in a sense, 
the hyperbolic counterpart of elliptic conical singularities, and
that the solutions of linear, second order hyperbolic equations
admit regularity results in scales of corner-weighted spaces.
Based on the regularity results in \cite{KokPlam2004,KNP_Dir99,KoPlaMixed,LuTu15},
it was shown by some of the authors of the present manuscript 
in \cite{MuScSc18} that high-order, time-marching discretizations of linear,
acoustic wave equations with FEM discretizations
in the spatial domain can recover  the maximal convergence rate 
afforded by the elemental polynomial degree, even in the presence of
conical singularities.

\subsection{Previous results}
\label{sec:PrevWrk}
The necessity of some form of stabilization has been recognized early on in the 
development of space--time discretizations.
We refer to the classical stabilized spacetime FEM \cite{HughHulb88,HughHulb90},
where least squares, mesh-dependent stabilization has been introduced.
For general first-order systems, Monk--Richter \cite{MoRi05} introduced a 
DG space--time scheme based on numerical fluxes which upwind in the transport direction.
More recently, Trefftz DG methods, i.e.\ schemes employing locally exact solutions of the wave equation,
have been considered by Moiola--Perugia \cite{MoPe18}, 
following the DG framework of \cite{MoRi05}, and 
Banjai--Georgoulis--Lijoka \cite{BGL2016}, in an interior penalty (IP) setting.
Most of these error analyses have been developed under the assumption that the exact solution is sufficiently smooth, as a function of space and time.
In polyg\-o\-nal spatial domains or in domains composed of different homogeneous materials, high regularity of solutions in standard Sobolev scales is known to fail,
due to diffractive solution components with point singularities.
We refer to results obtained in recent years on the regularity 
of solutions, also in nonsmooth domains, which are phrased in terms
of \emph{corner-weighted Sobolev spaces of Kondrat'ev type},
see \cite{KokPlam2004} and the references there, and 
the more recent Luong--Tung \cite{LuTu15}.
The solutions of linear, second-order wave equations 
in polygons has been shown in these
references to consist of a smooth part plus a finite, 
time-dependent linear combination of singular ``corner'' solutions,
with coefficients which depend regularly on the time-variable.
The canonical corner solutions 
correspond to certain solutions of 
certain conical diffraction problems in infinite wedges
given by tangent cones to the polygon with vertices in the corners.

A-priori estimates of solutions in corner-weighted spaces were
used in M\"uller--Schwab \cite{MuSc15}, 
and in the PhD thesis \cite{MullerPhD}, 
to develop convergence rate bounds for time-marching 
spatial DG-FEM approximations of
time-domain wave propagation in polygonal domains under 
realistic regularity assumptions on the solution.

Concurrent (full-tensor, in our terminology) 
space--time discretizations for linear wave equations 
were already proposed by \cite{HughHulb88,HughHulb90}.
Error bounds for smooth solutions were obtained and numerical examples
in one spatial dimension were reported.
For related recent numerical analysis results of space--time discretization
schemes for evolution equations, we mention
the work Cangiani--Dong--Georgoulis \cite{CDG17} where space--time DG schemes for 
linear, parabolic PDEs were investigated on prismatic meshes.
Maximal regularity of the solution in the spatial and the temporal variable 
was assumed in the analysis in \cite{CDG17}.
Steinbach--Zank \cite{SteinbachZank2019} proposed a stabilization for a
conforming space--time finite element method on Cartesian meshes, 
in order to overcome time-step restrictions in meshes that are locally refined in space.
Ernesti--Wieners \cite{WienersxtDGWave2019} consider a (full-tensor) \emph{Petrov-Galerkin (DPG) discretization} of linear acoustic wave equations, 
and prove a-priori error bounds assuming maximal solution regularity
 in the spatial and temporal variable {in standard Sobolev spaces, i.e.\
precluding point singularities admitted in the present work.}
The work \cite{dohr2018parallel} focuses again on space--time discretization 
of the heat equation, and addresses the solver complexity.
A space--time DG discretization for the parabolic Navier-Stokes equation was proposed in \cite{VanderVxtNSE2013}.

For the linear, acoustic wave equation, 
the CFL constraint encountered in explicit time-stepping
is exacerbated by local spatial mesh-refinement near corners.
One remedy with a space--time discretization flavour is here 
the so-called \emph{local time-stepping}. 
We refer to Diaz--Grote \cite{DiazGrote2015} for details
and to \cite{PetSched} for a method-of-lines based, 
related approach to circumvent the CFL-constraint in 
explicit time-marching.

The recent work on explicit, marching-type space--time schemes 
by Gopalakrishnan--Sch\"oberl--Winter\-steiger \cite{GSW16} 
is rooted in Falk--Richter \cite{FalkRichter1999}.
These schemes are based on so-called \emph{tent-pitched space--time meshes}, 
where the PDE can be explicitly evolved in the causal direction, i.e.
from the ``bottom'' to the ``top'' of the space--time cylinder element by element, thereby avoiding global CFL constaints. 
This evolution is performed, after mapping each element into a 
space--time cylinder, by applying a Runge-Kutta or a Taylor time-stepping; 
see~\cite{GSW16} and~\cite{GHSW19}, respectively. 
An alternative is to combine tent-pitching with Trefftz basis  functions~\cite{StockerSchoeberl}.

\subsection{Contributions}
\label{sec:Contr}

In the present paper, 
we obtain the following novel contributions.
Firstly, we extend the consistency analysis of the space--time DG discretization
proposed in \cite{MoPe18} to
general (namely non-Trefftz) discrete spaces and to the (realistic) setting 
of solutions which exhibit spatial point singularities situated at corners of 
the spatial domain $\Omega$ or, for transmission problems, 
at multimaterial interface points.
These points generate, in two spatial dimensions, diffraction terms which propagate radially with the acoustic speed of sound afforded by the medium.
In this case, convergence rate bounds which are based on maximal spatial 
regularity of solutions on quasi-uniform meshes of the spatial domain
are moot, as higher order spatial regularity can 
only be expressed in scales of \emph{corner-weighted Sobolev spaces} of Kondrat'ev type. 
Our present error analysis and convergence rate bounds (see Proposition~\ref{prop:rough_errorBounds_lrefMeshes})
also cover \emph{interior singularities} of solutions as arise, 
for example, due to multi-material interfaces in $\Omega$,
based on corresponding regularity in weighted Sobolev scales; see~\cite{MuSc15,MuScSc18}
and the references there.

Furthermore, in this (realistic, i.e.\ including corner-singularities) 
setting, we propose a \emph{novel, sparse space--time DG discretization}.
It allows to build an approximate solution of the evolution equation
in error vs.\ work which corresponds, asymptotically, 
to that of one elliptic solve at the highest spatial resolution in~$\Omega$, 
with the corresponding number of degrees of freedom.
This is comparable with the error vs.\ work offered, for example, by
integral equation based methods combined with convolution quadrature 
(see, e.g., \cite{BanjLubMel2011} and the references there), 
without access to explicit fundamental solutions.

The \emph{consistency error bounds} in mesh-dependent norm
we provide for the full space--time DG scheme holds true without any
time-step size constraint, also in the presence of 
corner singularities and spatially refined meshes. 
This indicates that its sparse--tensor version may feature superior error vs.\ work 
performance, as observed subsequently in detailed numerical experiments.
Unlike integral-equation-based methods, which require an explicit fundamental solution, the presently proposed results and space--time DG discretization does \emph{not} mandate homogeneous media: 
the presently considered \emph{DG formulation} can accomodate also more general, inhomogeneous but (piecewise) sufficiently regular coefficients (in particular the wave speed).
Here we focus on the piecewise-constant materials case and leave the extension of the analysis to smooth coefficients to future work.

Some further comments on the novelty and the generality of our results are in order.
Part of the abstract analysis of the DG scheme follows that of \cite{MoRi05,MoPe18}, 
but it differs in that the techniques used in these references 
do not allow the use of tensor-product discrete spaces, see Remark~\ref{rem:MonkRichter} for details.
The presently developed error analysis is a high-order ``$h$-convergence'' analysis; 
it does not establish ``$p$-version'' convergence for increasing polynomial degrees on a fixed space--time mesh.
Finally, we only consider two-dimensional spatial domains
as the regularity and corner singularity theory for the wave equation in polygons, in weighted spaces of finite order is nowadays quite complete, 
as described above. 
The corresponding theory for solutions in polyhedra (with diffraction from both, vertex and edge singularities) will mandate
\emph{anisotropic mesh refinement in the vicinity of edges} which, in turn, 
requires significant technical modifications in all parts of the present DG error analysis and whose development is beyond the scope of the present paper.

\subsection{Outline}
\label{sec:Outline}
The structure of this paper is as follows. 
In \S\ref{sec:IBVP} we state initial boundary value problems for the acoustic wave equation
for a homogeneous medium in both first- and second-order formulations; 
we also provide sufficient conditions for their well-posedness. 
\S\ref{s:Regularity} reviews relevant regularity results on the exact solution
from \cite{MuSc15,MullerPhD}, based on \cite{KokPlam2004,LuTu15}.
\S\ref{s:Mesh} prepares the (somewhat involved) notation for the ensuing
development of the space--time DG discretization. \S \ref{sec:InvIneq}
in particular has some polynomial inverse inequalities.
\S \ref{s:DG} then develops the space--time DG formulation,
with corresponding existence and uniqueness results of the discrete solution,
and abstract, ``quasi-optimality-like'' a-priori error bounds.
\S \ref{s:ErrorBounds} then provides the convergence rate bounds for the full-tensor space--time DG scheme,
subject to either smooth solutions or solutions belonging to corner-weighted spaces.
\S \ref{sec:RefMesXDom} discusses the generation and properties 
of families of corner-refined triangulations 
of the spatial domain and their complexity, leading to convergence rate bounds for the space--time DG scheme posed on such meshes.
\S \ref{s:implement_numexp} describes the outcomes of several numerical examples 
involving smooth and singular solutions in homogeneous and heterogeneous media, quasi-uniform and locally refined meshes.
\S \ref{sec:SprseXT} finally introduces the sparse space--time DG discretization, 
compares its computational complexity against the full-tensor approach,
and provides numerical results for the sparse version of the scheme applied to the same examples of the previous section.

\section{Initial boundary value problem}
\label{sec:IBVP}
In this section, we introduce the model problem given by linear, acoustic
wave propagation in a bounded, polygonal domain $\Omega$.
We discuss its well-posedness.

\subsection{Problem statement}
\label{sec:PrbStat}
We consider an initial boundary value problem (IBVP) for the linear, acoustic wave-equation posed on a space--time domain $Q=\Omega\times I$, 
where $\Omega\subset\IR^2$ is an open, bounded, Lipschitz polygon with 
straight sides and with outward unit normal $\bn_\Omega^x$, and 
where $I=(0,T)$, $T>0$, denotes a time-interval with finite time horizon.

The boundary of $\Omega$ is divided in two parts, with mutually
disjoint interiors, denoted $\GD$ and $\GN$ corresponding to
Dirichlet and Neumann boundary conditions, respectively; 
one of them may be empty.
The first-order acoustic wave IBVP reads as
\begin{align}\label{eq:IBVP}
\left\{\begin{aligned}
&\nabla v+\der{\bsigma}t = \bzero &&\iin Q,\\
&\nabla\cdot\bsigma+c^{-2}\der{v}t = f &&\iin Q,\\
&v(\cdot,0)=v_0, \quad \bsigma(\cdot,0)=\bsigma_0 &&\oon \Omega,\\
&v=\gD, &&\oon \GD\times [0,T],\\
&\bsigma\cdot\bn_\Omega^x=\gN, &&\oon \GN\times [0,T].
\end{aligned}
\right.
\end{align}
Here $f,v_0,\bsigma_0,\gD,\gN$ are the problem data;
$c(\bx)\ge c_0>0$ is the  wave speed, which is assumed to be piecewise constant 
on a fixed, finite polygonal partition $\{ \Omega_i \}$ of $\Omega$ 
which is independent of $t$.
We denote by $\calS:=\{\bc_i,i=1\ldots,M\}$ the set of all vertices of the 
polygons $\Omega_j\subset \Omega$ on which $c$ is constant.
We also assume that $\overline{\Omega}_j\cap \partial \Omega$ is 
contained in either $\overline{\Gamma}_D$ or in $\overline{\Gamma}_N$.
This IBVP is a special case of \cite[(1)]{MoPe18} and \cite[\S6]{MoRi05}, where also impedence boundary conditions were allowed.

The corresponding IBVP for the second-order scalar wave equation is:
\begin{align}\label{eq:IBVP_u}
\left\{\begin{aligned}
&-\Delta u+c^{-2}\der{{^2}u}{t^2}=f&&\iin Q,\\
&\der u t(\cdot,0)=v_0, \quad u(\cdot,0)=u_0 &&\oon \Omega,\\
&\der ut=\gD, &&\oon \GD\times [0,T],\\
&-\bn_\Omega^x\cdot\nabla u=\gN, &&\oon \GN\times [0,T].
\end{aligned}
\right.
\end{align}
Formally, given a smooth solution $u$ of \eqref{eq:IBVP_u}, 
its first-order derivatives $(v,\bsigma)=(\der ut,-\nabla u)$ constitute 
a solution of \eqref{eq:IBVP} with $\bsigma_0=-\nabla u_0$.
Vice versa, if $(v,\bsigma)$ is a smooth solution of \eqref{eq:IBVP} with $\bsigma_0=-\nabla u_0$, 
then $u(\cdot,t)=u_0+\int_0^t v(\cdot,s)\di s$ is solution of \eqref{eq:IBVP_u}.

The IBVP \eqref{eq:IBVP_u} can be written with the Dirichlet condition in the more 
commonly encountered form $u=G_{\mathrm D}$ on $\GD\times[0,T]$ simply by taking 
$G_{\mathrm D}(\bx,t):=u_0(\bx)+\int_0^t\gD(\bx,s)\di s$ for $(\bx,t)\in \GD\times[0,T]$, for, e.g., 
continuous $u_0$ and $\gD$.
An IBVP with boundary condition $u=G_{\mathrm D}$ on $\GD\times[0,T]$ satisfying $G_{\mathrm D}(\bx,0)=u_0(\bx)$ can be written as \eqref{eq:IBVP_u} simply by taking $\gD=\der{G_{\mathrm D}}t$.

\subsection{Variational solutions}
In order to give conditions ensuring the well-posedness of the first-order IBVP \eqref{eq:IBVP}, we firstly consider the analogous problem for the second-order wave equation \eqref{eq:IBVP_u}.
In this section we only consider the homogeneous Dirichlet case, namely with $\GN=\emptyset$ and $\gD=0$.

A classical result on variational methods for evolution equations, see
\cite[pp.~581--582, Chapt.~XVIII, \S6.1]{DL5}, states that, if
$$
f\in L^2(Q), \quad
u_0\in H^1_0\OO,\quad
v_0\in L^2\OO,\quad
\GN=\emptyset,\quad
\gD=0,
$$
then \eqref{eq:IBVP_u} admits a unique variational solution
\begin{equation}\label{eq:WPu}
u\in C^0\big([0,T];H^1_0\OO\big)\cap C^1\big([0,T];L^2\OO\big)\cap H^2\big(0,T;H^{-1}\OO\big),
\end{equation}
where we use the Bochner-space notation as in, e.g., \cite{DL5}.
The fact that $\der{{^2}u}{t^2}\in L^2(0,T; H^{-1}\OO)$ 
follows from the wave equation, the space regularity of $u$ and $f$, 
and the mapping $\Delta:H^1_0\OO\to H^{-1}\OO$.

We state some simple properties of negative-regularity Sobolev spaces of vector fields.
We define $\Hodiv^*$ as the dual space of the classical space 
$\Hodiv=\{\btau\in \LtO^2,\nabla\cdot\btau\in\LtO, \btau\cdot\bn^x_\Omega=0 \iin H^\mhalf(\deO)\}$.
Then, for all $w\in\LtO$, we have that $\nabla w\in \Hodiv^*$: this is because the duality
$\langle \nabla w,\bz\rangle_{H_0(\div;\OO)^*\times H_0(\div;\OO)}
  =  -\int_\Omega w \nabla\cdot\bz\di\bx$ for $\bz\in H_0(\div;\OO)$ 
is an extension of the $(L^2\OO)^2$ scalar product
(equivalently, $\nabla: L^2\OO\to\Hodiv^*$ as it is the adjoint of $-\nabla\cdot: \Hodiv\to L^2\OO$).
Moreover, the embedding $H_0(\div;\OO)^*\subset (H^{-1}\OO)^2$ holds,
since the inclusion $(H^1_0\OO)^2\subset H_0(\div;\OO)$ is dense and continuous and $H^{-1}\OO=H^1_0\OO^*$.

We now make use of the second-order equation in the context of the first-order system \eqref{eq:IBVP}.
If
\begin{align}\label{eq:WPassumpt}
f\in L^2(Q), \quad
\bsigma_0\in L^2\OO^2,\quad
v_0\in L^2\OO,\quad
\GN=\emptyset,\quad
\gD=0,
\end{align}
we define $u_0\in H^1_0\OO$ to be the solution of the Poisson equation 
$-\Delta u_0=\nabla\cdot\bsigma_0\in H^{-1}\OO$ in
$\Omega$.
Then we denote by $u$ the unique variational solution of \eqref{eq:IBVP_u} with such initial condition $u_0$, and define
\begin{align}\label{eq:WPregul}
\vs:=&\Big(\der ut,\; -\nabla u+\nabla u_0+\bsigma_0\Big)\\&
\in \Big(C^0\big([0,T],L^2\OO\big)\cap H^1\big(0,T; H^{-1}\OO\big)\Big)
\times\Big(C^0\big([0,T],(L^2\OO)^2\big)\cap C^1\big([0,T]; H_0(\dive;\Omega)^*\big)\Big).
\nonumber
\end{align}
Then, $\vs$ is solution of the first-order IBVP~\eqref{eq:IBVP}.
Note that, in the definition of $\bsigma$, the time-independent term $\nabla u_0+\bsigma_0$ is needed in order to
deal with initial conditions $\bsigma_0$ that are not gradients, see
\cite[Rem.~1]{MoPe18}.

Thus, for all data satisfying \eqref{eq:WPassumpt}, we have defined a solution field $\vs$.
In which sense does $\vs$ solve the IBVP~\eqref{eq:IBVP}?
If it is smooth, e.g.\ $\vs\in (C^1(Q)\cap C^0(\conj Q))\times (C^1(Q)\cap C^0(\conj Q))^2$, then~\eqref{eq:IBVP} holds pointwise in the classical sense.
More generally, we deduce from \eqref{eq:WPregul} that the two PDEs
in~\eqref{eq:IBVP} 
hold in the following spaces: 
$$
\nabla v+\der{\bsigma}t = \bzero \;\iin C^0\big([0,T]; H_0(\dive;\OO)^*\big),
\qquad
\nabla\cdot\bsigma+c^{-2}\der{v}t = f \;\iin L^2\big(0,T; H^{-1}\OO\big)
$$
and the initial conditions $v=v_0$, $\bsigma=\bsigma_0$ hold in
$L^2\OO$.
The reason why the scalar equation is valid in $L^2\big(0,T;H^{-1}\OO\big)$
only and not in $C^0\big([0,T];H^{-1}\OO\big)$ is that $f\in L^2(Q)$.
The spaces in~\eqref{eq:WPregul} do not allow to take traces of $\vs$ on $\deO\times(0,T)$, 
so the boundary condition $v=\gD=0$ on $\deO\times[0,T]$ is satisfied
only weakly.
If $\vs$ is a smooth solution, by testing the PDEs of \eqref{eq:IBVP} 
against a test field $\wt$ and integrating by parts, 
we obtain 
\begin{align}\nonumber
\calA\big(\vs,\wt\big)
&=-\int_Q \bigg(v\Big(\nabla\cdot\btau+c^{-2}\der wt\Big)
  + \bsigma\cdot\Big(\nabla w+\der\btau t\Big)\bigg)\di V
+\int_{\Omega\times\{T\}}(\bsigma\cdot\btau+c^{-2}vw)\di\bx 
\\
\label{eq:WPA}
&=\int_Q fw \di V+\int_{\Omega\times\{0\}}(\bsigma_0\cdot\btau+c^{-2}v_0w)\di\bx
\\
\forall\wt\in \Big(&C^0\big([0,T];H^1_0\OO\big)\cap H^1\big(0,T;\LtO\big)\Big)
\times \Big(C^0\big([0,T];\Hdiv\big)\cap C^1\big([0,T];(\LtO)^2\big)\Big).
\nonumber
\end{align}
Here and in the following, our notation for differentials within integrals is as follows: we use $\!\di\bx=\!\di x_1 \di x_2$ within integrals over (two-dimensional) spatial regions, $\!\di V=\!\di\bx\di t$ within volume integrals over (three-dimensional) space--time domains, and $\!\di S$ within surface integrals over general two-dimensional surfaces in space--time.

Note that, of the two boundary terms expected from integration by parts, 
$\int_{\deO\times(0,T)}\bsigma \cdot\bn_\Omega^x w\di S$ is not present in $\calA(\cdot,\cdot)$ as 
$w\in C^0\big([0,T];H^1_0\OO\big)$, while the absence of the term
$\int_{\deO\times(0,T)}v\btau\cdot\bn_\Omega^x\di S$ 
corresponds to weakly imposing the homogeneous Dirichlet boundary condition.
By density\footnote{We sketch here the density argument.
Given $\vs$ a variational solution as in \eqref{eq:WPregul}, there is $u$, solution of the second order problem \eqref{eq:IBVP_u}, with regularity \eqref{eq:WPu}.
By density of smooth functions in Sobolev spaces, there exists a sequence $u_j\in C^\infty(\conj Q)$ such that $u_j\to u$ in $C^0\big([0,T];H^1_0\OO\big)\cap C^1\big([0,T];L^2\OO\big)\cap H^2\big(0,T;H^{-1}\OO\big)$ norm.
Defining $(v_j,\bsigma_j):=(\der{u_j}t,-\nabla u_j)$, we have $(v_j,\bsigma_j)\to\vs$ in the norm of \eqref{eq:WPregul}, and 
$\nabla v_j+\der{\bsigma_j}t = \bzero$, $\nabla\cdot\bsigma_j+c^{-2}\der{v_j}t = f_j$, where $f_j\to f$ in $L^2\big(0,T; H^{-1}\OO\big)$.
Moreover $(v_{0,j},\bsigma_{0,j}):=(v_j,\bsigma_j)|_{t=0}\to\vs$ in $L^2\OO^{1+2}$ because of the definition of the space in \eqref{eq:WPregul}.
From the continuity of the bilinear form $\calA$ in the topology in which we have convergence, we deduce that, for all $\wt$ as in \eqref{eq:WPA},
$$
\calA(\vs,\wt)
=\lim_{j\to\infty}\calA((v_j,\bsigma_j),\wt)
=\lim_{j\to\infty} \int_Q f_jw+\int_{\Omega\times\{0\}}(\bsigma_{0,j}\cdot\btau+c^{-2}v_{0,j}w)
=\int_Q fw+\int_{\Omega\times\{0\}}(\bsigma_{0}\cdot\btau+c^{-2}v_{0}w).
$$
}, the variational identity \eqref{eq:WPA} is satisfied also by rougher $\vs$ with regularity \eqref{eq:WPregul}:
this makes precise the sense in which
the boundary conditions are weakly satisfied by $\vs$.

In order to show uniqueness of the variational solution
  of~\eqref{eq:IBVP}, let $\vs$ be a solution of~\eqref{eq:WPA},
with $f=0$, $v_0=0$, and $\bsigma_0=\bzero$,
that belongs to the space indicated in~\eqref{eq:WPregul}.
Define
\[
u(\cdot,t):=\int_0^tv(\cdot,s)\di s,
\]
which, from the regularity~\eqref{eq:WPregul} of $\vs$, satisfies
$u\in C^1\big([0,T],L^2\OO\big)\cap H^2\big(0,T; H^{-1}\OO\big)$.
Note that the continuity in time of $v$ implies that $u(\cdot,0)=0$ in $L^2\OO$.
As discussed above, in this setting, $\vs$ satisfies
\begin{equation}\label{eq:eqhomog}
\nabla v+\der{\bsigma}t = \bzero \;\iin C^0\big([0,T]; H_0(\dive;\OO)^*\big),
\qquad
\nabla\cdot\bsigma+c^{-2}\der{v}t = 0 \;\iin C^0\big(0,T; H^{-1}\OO\big).
\end{equation}

By integrating the first equation in~\eqref{eq:eqhomog} in time, taking into
account that $\bsigma_0=\bzero$, we obtain
\[
\bzero=\int_0^t\nabla v(\cdot,s)\di s +\int_0^t
\der{\bsigma}t(\cdot,s) \di s 
=\nabla \int_0^t v(\cdot,s)\di s+\bsigma(\cdot,t)
=\nabla u(\cdot,t)+\bsigma(\cdot,t),
\]
namely,
\[
\nabla u(\cdot,t)=-\bsigma(\cdot,t) \;\iin
C^0\big([0,T],(L^2\OO)^2\big)\cap C^1\big([0,T]; H_0(\dive;\Omega)^*\big),
\]
thus $u$ also belongs to $C^0\big([0,T],H^1\OO\big)$.
  
The fact that $u$ has zero Dirichlet trace on $\partial\Omega$
follows from the homogeneous weak Dirichlet boundary condition on $v$, 
namely
$\int_{\deO\times(0,T)}v\btau\cdot\bn_\Omega^x\di S=0$ for all
$\btau\in C^0\big([0,T];\Hdiv\big)\cap C^1\big([0,T];(\LtO)^2\big)$.
In fact, taking $\btau$ independent of time, replacing $v$ by $\der{u}t$, and
integrating by parts in time immediately give that the trace of
$u(\cdot,T)$ on $\partial\Omega$ is zero in the sense of
$H^{1/2}(\partial\Omega)$, thus $u(\cdot,T)\in H^1_0\OO$.
Therefore, by integrating by parts in time, taking into account that
the initial and final traces of $u$ on $\partial\Omega$ vanish,
we have
\[
  0=\int_{\deO\times(0,T)}v\btau\cdot\bn_\Omega^x\di S
  =\int_{\deO\times(0,T)}\der{u}t\btau\cdot\bn_\Omega^x\di S
  =-\int_{\deO\times(0,T)}u\der{\btau}t\cdot\bn_\Omega^x\di S
\]
for all $\der{\btau}t\cdot\bn_\Omega\in
C^0([0,T];H^{-1/2}(\partial\Omega))$,
from which we deduce that the trace of $u$ on
$\partial\Omega\times(0,T)$
is zero in $C^0([0,T];H^{1/2}(\partial\Omega))$.
Consequently, $u$ also belongs to
$C^0\big([0,T],H^1_0\OO\big)$, and we conclude that $u$ has the
regularity~\eqref{eq:WPu}.

Finally, inserting $v=\der{u}t$ and $\bsigma=-\nabla u$ into the second equation
in~\eqref{eq:eqhomog}
gives
\(
-\Delta u+c^{-2}\der{{^2}u}{t^2}=0 \;\iin C^0\big(0,T; H^{-1}\OO\big),
\)
together with $u(\cdot,0)=0$ and $\der{u}t(\cdot,0)=0$.

We have proven that $u$ is actually  variational solution of~\eqref{eq:IBVP_u}.
Therefore, by invoking uniqueness of the variational solution of~\eqref{eq:IBVP_u}, we deduce uniqueness of the variational solution of~\eqref{eq:IBVP}. 
We summarize the results above in the following proposition.

\begin{proposition}
Under assumptions~\eqref{eq:WPassumpt}, there exists a unique
variational solution $\vs$, in the sense of \eqref{eq:WPA}, of the IBVP~\eqref{eq:IBVP}, with regularity
as in~\eqref{eq:WPregul}.
\end{proposition}

Alternatively, when $\GN=\emptyset$, $f=0$, $\gD=0$, $v_0\in H^1_0\OO$, $\bsigma_0\in \Hdiv$, the well-posedness of IBVP \eqref{eq:IBVP} could be derived from the Lumer--Phillips theorem as in \cite[pp.~238--239]{HPSTW15}, without using the second-order problem.

\section{Solution regularity}\label{s:Regularity}
For the convergence rate analysis of high-order, full-tensor and sparse-tensor 
DG discretizations in the presence of geometric singularities due to 
corners and multimaterial interfaces, we require regularity results in 
corner-weighted spaces in $\Omega$, which we now recapitulate, based on
\cite{LuTu15,MuSc15,MuSc16,MullerPhD,MuScSc18}.

We collect the weight exponents $\delta_i\in[0,1)$ assigned to each $\bc_i\in\calS$, in $\bdelta=\{\delta_i\}_{i=1}^M\in [0,1)^M$ (see Remarks \ref{rmk:deltai=0}, \ref{rmk:deltai>0}) and define the weight function $\Phi_\udelta(\bx):=\prod_{i=1}^M|\bx-\bc_i|^{\delta_i}$. 
We express the regularity of the solution in the spatial domain $\Omega$ 
in terms of weighted spaces $H^{k,\ell}_\udelta\OO$,
which are defined as the completions of 
$C^\infty(\overline{\Omega})$ with respect to the weighted Sobolev 
norms $\N{\cdot}_{H^{k,\ell}_\udelta(\Omega)}$ which, for $1\le \ell \le k\in \IN$, 
are given by
\begin{align*}
\N{u}_{H^{k,\ell}_\udelta(\Omega)}^2 
:=\N{u}_{H^{\ell-1}(\Omega)}^2 +\abs{u}_{H^{k,\ell}_\udelta(\Omega)}^2, 
\qquad
\abs{u}_{H^{k,\ell}_\udelta(\Omega)}^2
:=\sum_{m=\ell}^k \int_\Omega \bigg(\Phi_{\udelta+m-\ell}^2
\hspace{-2mm}\sum_{\substack{\malpha\in\IN_0^2\\\alpha_1+\alpha_2=m}}
|D^\malpha u|^2\bigg)\di\bx.
\end{align*}
We will use these spaces only for $\ell=1$ and $\ell=2$; 
in particular, we will mostly use the following norms and seminorms
\begin{align*}
\N{u}_{H^{1,1}_\udelta(\Omega)}^2 
:= &\N{u}_{L^2(\Omega)}^2+\abs{u}_{H^{1,1}_\udelta(\Omega)}^2,
\quad 
&&\abs{u}_{H^{1,1}_\udelta(\Omega)}^2 
:= 
\N{\Phi_\udelta\nabla u}_{L^2(\Omega)^2}^2,
\\
\N{u}_{H^{2,2}_\udelta(\Omega)}^2
:=&
\N{u}_{L^2(\Omega)}^2+\N{\nabla u}_{L^2(\Omega)^2}^2+\abs{u}_{H^{2,2}_\udelta(\Omega)}^2,
\quad 
&&\abs{u}_{H^{2,2}_\udelta(\Omega)}^2:= 
\N{\Phi_\udelta D^2 u}_{L^2(\Omega)^{2\times2}}^2,
\end{align*}
where $D^2 u$ denotes the Hessian of $u$.

\begin{remark}\label{rmk:deltai=0}
For $\delta_i=0$, $1\le i\le M$, we have $\Phi_\udelta \equiv1$. 
In this case, some of the weighted seminorms reduce to standard ones, i.e.\ $\abs{\circ}_{H^{k,k}_\udelta(\Omega)} = \abs{\circ}_{H^k(\Omega)}$ for $\bdelta=(0,\ldots,0)$ and $k\in \IN$, and $\N{u}_{H^{2,2}_\udelta(\Omega)} = \N{u}_{H^2(\Omega)}$.
\end{remark}

\begin{remark}	\label{rmk:deltai>0}
For a general polygon $\Omega$, admissible values of the parameters $\delta_i\in [0,1)$ depend on the coefficients of the elliptic spatial operator, on the boundary conditions at the sides of $\Omega$ meeting at corner $\bc_i\in\calS$ and on the interior opening angle $\omega_i\in (0,2\pi)$ at $\bc_i$. 
For a homogeneous, isotropic material at a corner $\bc_i$, with either homogeneous Dirichlet or homogeneous Neumann BCs at either side meeting at $\bc_i$, we can take $\delta_i > 1-\pi/\omega_i$. 
At convex corners $\omega_i < \pi$, so that $\delta_i=0$ is admissible.
At corners $\bc_i$ where BCs change type $\delta_i > 1-\pi/(2\omega_i)$. 
Similar conditions are valid for transmission problems with multi-material interface points $\bc \in \Omega$.
At such points, however, $\delta_i$ may take values close to $1$.
We refer to \cite[Section~3.2]{GaspozMorin} for an example; a numerical example is provided in Section \ref{ss:test_1.4} ahead.
\end{remark}

The regularity of the solution $\vs$ of the IBVP \eqref{eq:IBVP} 
follows from the regularity of the solution $u$ of the second-order IBVP \eqref{eq:IBVP_u}, 
by taking time and space derivatives.
For example, if 
$$
v_0,u_0\in C^\infty_0(\Omega),\quad \bsigma_0=-\nabla u_0, \quad f\in C^\infty_0(Q), \quad \gD=\gN=0,
$$
then, by \cite[Cor.~2.6.6]{MullerPhD} (see also \cite[Prop.~2.2]{MuScSc18} for the case with constant $c$), 
there exists $\bdelta\in[0,1)^M$ such that, for all $k_t,k_x\in\IN$, it holds that
\begin{equation}\label{eq:vsRegularity2}
(v,\bsigma)=\Big(\der ut,-\nabla u\Big)\in 
C^{k_t-1}\big([0,T]; H^{k_x+1,2}_\udelta\OO\big)\times 
C^{k_t}\big([0,T]; H^{k_x,1}_\udelta\OO^2\big).
\end{equation} 

\section{Space--time DG discretization: Notation}
\label{s:Mesh}
We prepare notation and conventions which shall be used in the ensuing space--time DG discretization.

\subsection{Temporal, spatial, and space--time meshes}
\label{sec:SapTmpMsh}
Let us introduce a partition $\calTtime$ of the time domain
$(0,T)$ into $N\in\IN$ intervals $I_n$, $1\le n\le N$, with
\[
0=:t_0<t_1<\ldots, t_N:=T,\quad
I_n:=(t_{n-1},t_n),\quad h_n:=|I_n|, \quad \htime:=\max_{1\le n\le N}h_n,
\]
and introduce the following notation for the time slabs and the
partial cylinders, respectively:
$$
D_n:=\Omega\times I_n,\quad
Q_n:=\Omega\times(0,t_n),\quad \text{for }n=1,\ldots,N.
$$

For each $1\le n\le N$, we introduce a polygonal finite element
mesh $\calTspacen=\{K_\bx\}$ of the spatial domain~$\Omega$,
possibly with hanging nodes,
with 
\[
\hspn:=\displaystyle{\max_{K_\bx\in \calTspacen}}h_{K_\bx},\qquad 
h_{K_\bx}:=\diam K_\bx,\qquad \hsp:=\max_{1\le n\le N}h_{\bx,n}.
\]

For each spatial mesh
$\calTspacen$, we assume
{\em i)} shape-regularity, 
{\em ii)} non-degeneracy of faces, namely all element- and
  face-sizes are locally comparable,  
{\em iii)} alignment with the fixed partition $\{\Omega_i\}$ on
which the wave speed $c$ is piecewise constant, and
{\em iv)} that for each $K_\bx\in \calTspacen$, $\partial K_\bx$ contains
at most one element of~$\calS$, i.e.\ one vertex of the partition $\{\Omega_i\}$.

We partition the space--time domain $Q=\Omega\times (0,T)$ with a finite
element mesh $\calT_h$ given by 
\[
\calT_h:=\calT_h(Q):=\{K=K_\bx\times I_n:\ K_\bx\in \calTspacen,\ 1\le n\le N\}.
\]
Note that $\calT_h$ is a tensor product mesh whenever 
$\calTspacen=\calTspace$ for all $1\le n\le N$,
for a given spatial mesh $\calTspace$, that is $\calT_h=\calTspace\times \calTtime$.

As all spatial meshes $\calTspacen$, $1\le n\le N$, are aligned with the 
partition $\{\Omega_i\}$, then the wave speed $c$ is constant in each
element $K$ of the space--time mesh $\calT_h$, and we set $c_K:=c|_K$.
Moreover, if $K=K_\bx\times I_n$ is such that $\partial K_\bx\cap\calS=\{\bc_i\}$ is
non-empty, we denote by $\delta_K=\delta_{\bc_i}\in[0,1)$ the exponent of the space in~\eqref{eq:vsRegularity2}.

For any $1\le n\le N$, we define the time--truncated mesh
\begin{align*}
\calT_h(Q_n):=&\{K\in\calT_h, \;K\subset Q_n\}.
\end{align*}

We partition $\calT_h$ into elements abutting at a corner in $\calS$ and the remaining elements, i.e.\
\begin{align*}
\calT_h^\angle:=&\{K = K_\bx \times I \in \calT_h, \overline{K_\bx}\cap \calS \; \text{nonempty} \},\qquad
\calT_h^\odot:=\calT_h\setminus\calT_h^\angle.
\end{align*}
\subsection{Mesh faces}
\label{sec:Faces}
Each internal face $F=\partial K_1\cap \partial K_2$, 
for $K_1,K_2\in\calT_h$, 
with positive $2$-dimensional measure, 
is a subset of a hyperplane:
$$F\subset\Pi_F:=\big\{(\bx,t)\in \IR^{2+1}:\; \bx\cdot\bn^x_F+t\, n^t_F=a_F\big\},$$ 
where $(\bn^x_F,n^t_F)$ is a unit vector in $\IR^{2+1}$ and $a_F\in\IR$. 
We assume that all internal faces $F$ are either ``space-like'', 
i.e. with $\bn^x_F=\bzero$, or ``time-like'', i.e.\ with $n^t_F=0$.
On space-like faces, by convention, we choose $n^t_F>0$, 
i.e.\ the unit normal vector $(\bn^x_F,n^t_F)$ points forward in time.
Note that all time-like faces are rectangles of the form 
$F=F_\bx\times F_t$ with $h_{F_\bx}=|F_\bx|$ and $h_{F_t}=|F_t|$; 
we recall that $F_t=I_n$ for some $1\le n\le N$.
Finally, we denote the outward-pointing unit normal vector 
on $\deK$ by $(\bn^x_K,n^t_K)$.

We use the following notation for unions and sets of faces:\\
\begin{minipage}{0.62\textwidth}
\begin{align*}
\Fh:=&\bigcup\nolimits_{K\in\calT_h}\partial K \quad \text{(the mesh skeleton)},\\
\Fspa:=& \text{the union of the internal space-like faces,}\\
\Ftime:=& \text{the union of the internal time-like faces,}\\
\FO:=&\Omega\times\{0\},\hspace{12mm} \FT:=\Omega\times\{T\},\\ 
\FD:=&\GD\times[0,T],\qquad
\FN:=\GN\times[0,T],
\\
\IF:=&\{\text{faces of elements of }\calT_h\},\\
\IF^\star:=&\{F\in \IF, F\subset\Fh^\star\}
\qquad&&\hspace{-18mm}\text{for }\star\in\{\mathrm{time,D,N}\},\\
\IF^\star_\angle:=&\{F\in \IF^\star, F\subset\deK \text{ for }K\in\calT_h^\angle\}
&&\hspace{-18mm}\text{for }\star\in\{\mathrm{time, D,N}\},\\
\IF^\star_\odot:=&\IF^\star\setminus\IF^\star_\angle
&&\hspace{-18mm}\text{for }\star\in\{\mathrm{time,D,N}\},\\
\Fh^\angle:=&\text{the union of the faces F}\in \IFtime_\angle\cup\IFD_\angle\cup\IFN_\angle,
\\
\Fh^\odot :=&\text{the union of the faces F}\in \IFtime_\odot\cup\IFD_\odot\cup\IFN_\odot,
\\
\IF^\star_\bullet(\Upsilon):=&\{F\in\IF^\star_\bullet: F\subset \overline{\Upsilon}\},
\end{align*}
\end{minipage}\begin{minipage}{0.37\textwidth}
 \hfill\includegraphics[height = 5.5cm]{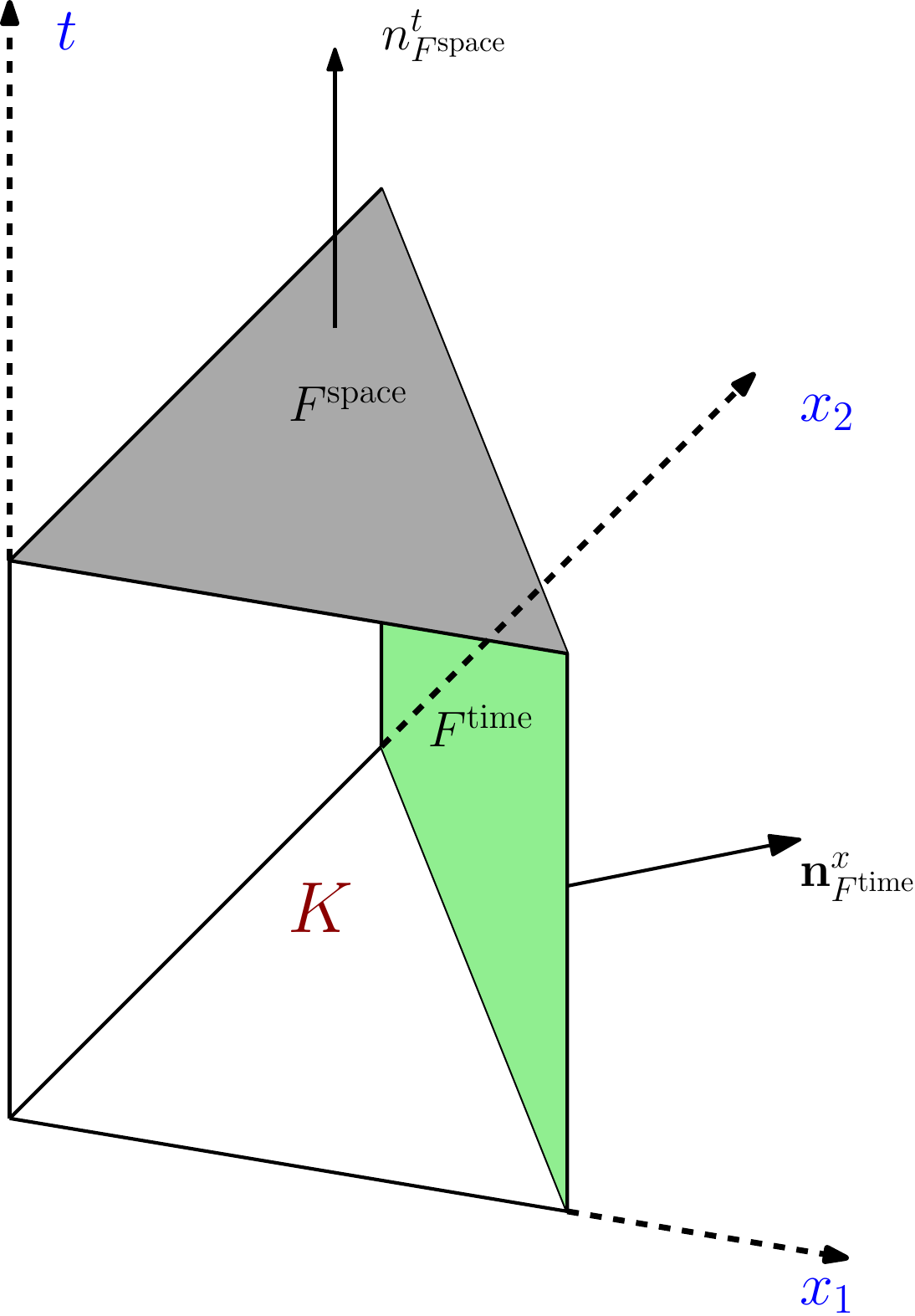}
\small 
\\
Space--time prism element $K$, with  time-like face $F^{\mathrm{time}}$ and
 space-like face $F^{\mathrm{space}}$.
\end{minipage}
\\[2mm]
for all subsets $\Upsilon\subset Q$ and
for all indices $\star,\bullet$ allowed by the previous definitions; 
in particular we will use $\IF^\star_\bullet(Q_n)$ and $\IF^\star_\bullet(K)$, $K\in\calT_h$.
The wave speed $c$ is assumed independent of time, so that
it may jump only across time-like faces in $\Ftime$.

\subsection{Averages and jumps}
\label{sec:JmpsAvg}
In the formulation and error analysis of the DG discretization, we 
require the ``usual'' notation for interelement averages and jumps.
We follow the notation in \cite{MoPe18} in order to refer to results
on space--time DG formulations which were established there.

For piecewise-continuous scalar ($w$) and vector ($\btau$) fields, we
define averages $\mvl\cdot$, space normal jumps $\jmp{\cdot}_\bN$, and
time (full) jumps $\jmp{\cdot}_t$ on internal mesh faces in the standard DG notation:
on $F=\deK_1\cap\deK_2$, $K_1,K_2\in\calT_h$,
\begin{align*}
\mvl{w}&:=\frac{w_{|_{K_1}}+w_{|_{K_2}}}2,\qquad
&\mvl{\btau}&:=\frac{\btau_{|_{K_1}}+\btau_{|_{K_2}}}2, \\
\jmp{w}_\bN&:= w_{|_{K_1}}\bn_{K_1}^x+w_{|_{K_2}} \bn_{K_2}^x,\qquad
&\jmp{\btau}_\bN&:= \btau_{|_{K_1}}\cdot\bn_{K_1}^x+\btau_{|_{K_2}} \cdot\bn_{K_2}^x,\\
\jmp{w}_t&:= w_{|_{K_1}}n_{K_1}^t+w_{|_{K_2}} n_{K_2}^t=(w^--w^+)n^t_F,\quad
&\jmp{\btau}_t&:= \btau_{|_{K_1}} n_{K_1}^t+\btau_{|_{K_2}} n_{K_2}^t= (\btau^--\btau^+)n^t_F.
\end{align*}
Here we have denoted by $w^-$ and $w^+$ the traces of the function $w$
from the adjacent elements at lower and higher times, respectively
(and similarly for $\btau^\pm$). 
The temporal jumps $\jmp{\cdot}_t$ 
are different from zero on space-like faces only.

\subsection{Some properties of \texorpdfstring{$H^{1,1}_\udelta(\Omega)$}{H11-delta} functions}
\label{sec:Smoothness}
We state in our notation some results from~\cite{TWDiss} that are needed in the 
error analysis of the space--time discretization.
\begin{lemma}\label{lem:H11}
If $u\in H^{1,1}_\udelta(\Omega)$ then
\begin{equation}\label{eq:H11trace}
u\in  L^1(F) \qquad\text{and} \qquad \jmp{u}_\bN=0 \text{ a.e. on }  F
\end{equation}
where $F$ is any segment in $\Omega$.
In general, $u\in H^{1,1}_\udelta(\Omega)$ does not guarantee that $u\in L^2(F)$.

Moreover, the embedding $H^{1,1}_\udelta(K_\bx)\subset L^2(K_\bx)$ is compact:
\begin{equation}\label{eq:smooth3}
H^{1,1}_\udelta(K_\bx)\subset\subset  L^2(K_\bx)
\end{equation}
\end{lemma}
\begin{proof}
See \cite[Lemma 1.3.2 and Remark 1.3.3]{TWDiss},
\cite[Lemma~1.3.4]{TWDiss}, 
and \cite[Proposition~A.2.5]{TWDiss}, respectively.
\end{proof}

If the IBVP solution $\vs$ satisfies the regularity condition \eqref{eq:vsRegularity2}, then $\bsigma$ has poor regularity in space, so we expect its traces on mesh edges close to the domain corners at a fixed time to be in $L^1$ but not in $L^2$. 
This is relevant, e.g., to define the DG skeleton seminorms in \S\ref{s:AbstractAnalysis}.
On the other hand, in the analysis we make use of the vanishing of the jumps of $\bsigma$ on these edges.

\subsection{Inverse inequalities}
\label{sec:InvIneq}
As the ensuing error analysis will be an ``$h$-version'' analysis, where the 
polynomial degrees are assumed fixed and convergence is achieved by 
suitable mesh refinement, we will require inverse inequalities 
in order to provide bounds on some face terms.

For any time-like face $F=F_\bx\times F_t\in\IFtime\cup\IFD$,
\begin{equation}\label{eq:L1L2Loo}
\N{\varphi}_{L^1(F_\bx)}\le h_{F_\bx}^{1/2}\N{\varphi}_{L^2(F_\bx)}
\quad \forall \varphi \in L^2(F_\bx),
\qquad
\N{\varphi}_{L^2(F_\bx)}\le h_{F_\bx}^{1/2}\N{\varphi}_{L^\infty(F_\bx)}
\quad \forall \varphi \in L^\infty(F_\bx).
\end{equation}
For polynomial functions, the following inverse inequalities hold:
\begin{equation}\label{eq:inverseL1L2Loo}
\N{P}_{L^2(F_\bx)}\le C_{2|1}h_{F_\bx}^{-1/2}\N{P}_{L^1(F_\bx)},\qquad
\N{P}_{L^\infty(F_\bx)}\le C_{\infty|2}h_{F_\bx}^{-1/2}\N{P}_{L^2(F_\bx)}
\quad \forall P \in \IP^p(F_\bx),
\end{equation}
where $\IP^p(F_\bx)$ denotes the space of polynomials of degree at most 
$p\in\IN_0$ on $F_\bx$, and the constants $ C_{2|1},C_{\infty|2}$
depend on $p$. 
The inequalities in~\eqref{eq:L1L2Loo} are obvious, while the inverse inequalities in~\eqref{eq:inverseL1L2Loo} follow from the equivalence of norms in finite dimensional spaces and scaling; note that $C_{\infty|2}\ge1$.
The constants in \eqref{eq:inverseL1L2Loo} depend linearly on the polynomial degree.
Denote by $P_p$ the Legendre polynomials and by $Q_p=\sqrt{p+\frac12}P_p$ their scalings which are \emph{orthonormal} in $L^2(-1,1)$ and which 
satisfy $\N{Q_p}_{L^\infty(-1,1)}=\sqrt{p+\frac12}$.
Then, for any polynomial $v=\sum_{j=0}^pa_jQ_j\in \IP^p(-1,1)$, 
\begin{align*}
\N{v}_{L^\infty(-1,1)}
&\le\sum_{j=0}^p |a_j|\N{Q_j}_{L^\infty(-1,1)}
\le \Big(\sum_{j=0}^p|a_j|^2\Big) ^{\frac12}
\Big(\sum_{j=0}^p (j+1/2)\Big)^{\frac12}
=\N{v}_{L^2(-1,1)}\frac{p+1}{\sqrt2},
\\
\N{v}_{L^2(-1,1)}^2
&=\sum_{j=0}^p|a_j|^2
=\sum_{j=0}^p\Big(\int_{-1}^1v\,Q_j\Big)^2\le
\N{v}_{L^1(-1,1)}^2 \sum_{j=0}^p\N{Q_j}^{2}_{L^\infty(-1,1)}
=\N{v}_{L^1(-1,1)}^2\frac{(p+1)^2}{2}.
\end{align*}
By scaling from $(-1,1)$ to a general interval we obtain \eqref{eq:inverseL1L2Loo} with
\begin{equation}\label{eq:InvIneqC}
C_{2|1}\le p+1,\qquad C_{\infty|2}\le  p+1
\qquad p\in\IN_0.
\end{equation}

The inverse inequalities \eqref{eq:inverseL1L2Loo} involve 
spaces of single-variable polynomials defined on segments; 
in the error analysis carried out in the next sections, 
we do not use inverse inequalities for spaces defined on polygons or polyhedra, so the shapes of the mesh elements do not play any role in this respect.

\section{Space--time DG discretization: Formulation}
\label{s:DG}
With the notation introduced in the preceding section,
we now derive the space--time DG discretization of 
the IBVP~\eqref{eq:IBVP}. 
We prove existence and uniqueness of discrete solutions, 
as well as abstract error estimates for these.

Throughout this section we assume the following regularity 
for the data and for the solutions $(v,\bsigma)$ of the IBVP~\eqref{eq:IBVP}:
\begin{align*}
&f\in L^2(Q), 
\qquad v_0\in H^{2,2}_\udelta\OO, 
\qquad \bsigma_0\in H^{1,1}_\udelta\OO^2,
\qquad \gD\in L^2\big(0,T;H^\half(\GD)\big),
\quad \gN\in L^2\big(0,T;L^2(\GN)\big),
\\
&\vs\in H^1\big([0,T]; H^{2,2}_\udelta\OO\big)\times H^1\big([0,T]; H^{1,1}_\udelta\OO^2\big).
\end{align*}

\subsection{Derivation of the space--time discrete formulation}
\label{s:Derivation}
We rewrite the formulation of \cite{MoPe18} in the non-Trefftz case.
This is a special case of the formulation of \cite{MoRi05} for general hyperbolic first-order systems.

To derive the DG formulation, we multiply the firstly two equations of \eqref{eq:IBVP} 
with test fields $\btau$ and $w$ and integrate by parts on a single mesh element $K\in\calT_h$:
\begin{align}
\label{eq:Elemental0}
&-\int_K\bigg(v
\Big(\nabla\cdot\btau+c^{-2}\der wt \Big) 
+\bsigma\cdot\Big(\nabla w +\der\btau t\Big)\bigg)\di V
\\
&\qquad+\int_{\partial K} \bigg((v\,\btau + \bsigma\, w)\cdot\bn_K^x
+\Big(\bsigma\cdot\btau + c^{-2} v\,w\Big)\,n_K^t\bigg)\di S
=\int_K fw\di V;
\nonumber
\end{align}
we recall that the wave speed $c$ is assumed to be 
constant in each $K\in \calT_h$.
We seek a discrete solution $(v\hp,\bsigma\hp)$ approximating
$(v,\bsigma)$ in a finite-dimensional 
(arbitrary, at this stage) space $\bVp \Th$. 
We take the test field $(w,\btau)$ in the same space $\bVp\Th$.
The traces of $v\hp$ and $\bsigma\hp$ on the mesh skeleton 
are approximated by the (single-valued) \emph{numerical fluxes} $\hVhp$ and $\hShp$, 
so that \eqref{eq:Elemental0} is rewritten as:
\begin{align}\label{eq:Elemental1}
&-\int_K\bigg(v\hp\Big(
\nabla\cdot\btau\hp+c^{-2}\der{w\hp} t \Big) 
+\bsigma\hp\cdot\Big(\nabla w\hp+\der{\btau\hp} t\Big)\bigg)\di V
\\\nonumber
&+\int_{\partial K} 
\bigg(\hVhp\Big(\btau\hp \cdot\bn_K^x+\frac{w\hp}{c^2} n_K^t\Big)
+\hShp\cdot\big(w\hp\bn_K^x+\btau\hp\,n_K^t\big)\bigg)
\di S=\int_K fw\di V.
\end{align}
We choose to define the numerical fluxes as:
\begin{align*}
\hVhp :=\begin{cases}
\Vhp^- \\
\Vhp \\
v_0\\ 
\mvl{\Vhp}+\beta \jmp{\Shp}_\bN \\
\gD\\
v\hp+\beta(\bsigma\hp\cdot\bn_\Omega^x-\gN)\\
\end{cases}
\hShp:=
\begin{cases}
\Shp^- & \oon \Fspa,\\
\Shp &\oon \FT,\\
\bsigma_0 &\oon \FO,\\
\mvl{\Shp}+\alpha \jmp{\Vhp}_\bN \hspace{-2mm}&\oon \Ftime,\\
\Shp+\alpha(v\hp-\gD)\bn_\Omega^x\hspace{-2mm}&\oon \FD,\\
\gN\bn_\Omega^x
&\oon \FN.
\end{cases}
\end{align*}
The stabilization parameters $\alpha\in L^\infty(\Ftime\cup\FD)$ and $\beta\in
L^\infty(\Ftime\cup\FN)$ will be chosen depending on the mesh,
but constant on each time-like face; 
see Corollary~\ref{cor:smooth} and \S\ref{sec:RefMesXDom} below.

These fluxes are \emph{consistent}, in the sense that they coincide
with the traces of the exact solution $(v,\bsigma)$ of the IBVP
\eqref{eq:IBVP} if they are applied to $(v,\bsigma)$ itself, which
satisfies the boundary conditions and has no jumps across mesh faces
by the $C^0([0,T];H^{1,1}_\udelta\OO)$ regularity of all components of
the solution and the property~\eqref{eq:H11trace} (see \cite[Lemma~1.3.4]{TWDiss}).
The numerical fluxes can be understood as upwind fluxes
in the usual sense on the space-like faces, and as standard DG-elliptic fluxes with jump penalisation on the time-like ones.

By summing the elemental DG equation \eqref{eq:Elemental1} over the elements $K\in \calT_h$, with the fluxes defined above, 
we obtain the space--time DG variational formulation:
\begin{align}\label{eq:DG}
\text{Find}&\; (\Vhp,\Shp)\in \bVp (\calT_h) \text{ such that }\;
\calA\DG(\Vhp,\Shp; w ,\btau )=\ell( w ,\btau )\quad 
\forall ( w ,\btau )\in \bVp (\calT_h), \text{ where}
\nonumber\\
\calA\DG(&v,\bsigma; w ,\btau ):=
-\sum_{K\in\calT_h}\int_K\bigg(v\Big(
\nabla\cdot\btau+c^{-2}\der{w} t \Big) +\bsigma\cdot\Big(\nabla w+
\der{\btau} t\Big)\bigg)\di V
\\\nonumber
&+\int_{\Fspa}\big(c^{-2}v^-\jmp{w}_t+\bsigma^-\cdot\jmp{\btau}_t\big)\di\bx
+\int_\FT (c^{-2}v  w +\bsigma \cdot\btau )\di \bx
\\
&+\int_\Ftime \big( \mvl{v}\jmp{\btau }_\bN+\mvl{\bsigma}\cdot\jmp{ w }_\bN
+\alpha\jmp{v}_\bN\cdot\jmp{ w }_\bN+ \beta\jmp{\bsigma}_\bN\jmp{\btau }_\bN\big)\di S
\nonumber\\
&+\int_\FD \big(\bsigma\cdot\bn_\Omega^x\, w +\alpha v w   \big) \di S+\int_\FN\big(v(\btau\cdot\bn_\Omega^x)
+\beta(\bsigma\cdot\bn_\Omega^x)(\btau\cdot\bn_\Omega^x)\big)\di S,
\nonumber\\
\ell( &w ,\btau ):=
\int_Q fw\di V
+\int_\FO ( c^{-2}v_0 w  +\bsigma_0\cdot \btau )\di \bx
\nonumber\\&
+\int_\FD \gD\big(\alpha  w -\btau\cdot\bn_\Omega^x\big)\di S
+\int_\FN \gN \big(\beta\,\btau\cdot\bn_\Omega^x-w\big)\di S.
\nonumber
\end{align}

\subsection{Existence and uniqueness of the discrete solution}\label{s:Uniqueness}

We prove existence and uniqueness of the DG solution under the assumption 
that all $(w\hp,\btau\hp)\in\bVp\Th$ are elementwise polynomials
(at this stage elementwise analytic suffices), 
and that we have
\begin{equation}\label{eq:dVpAssumpt}
\Big(\nabla\cdot\btau\hp+c^{-2}\der{w\hp} t,\nabla w\hp+\der{\btau\hp} t\Big)\in\bVp\Th \qquad \forall \wth\in\bVp\Th;
\end{equation}
recall that $c$ is piecewise constant in $\calT_h$.
We define the following DG seminorm:
\begin{align}\nonumber
\abs{(w ,\btau )}^2\DG&:=
\frac12 \Norm{c^{-1}\jmp{w }_t}_{L^2(\Fspa)}^2
+\frac12\Norm{\jmp{\btau }_t\rule{0pt}{3mm}}_{L^2(\Fspa)^2}^2 
+\frac12\Norm{c^{-1}w }^2_{L^2(\FO\cup\FT)} 
+\frac12\Norm{\btau \rule{0pt}{3mm}}^2_{L^2(\FO\cup\FT)^2} 
\\&\quad
+\Norm{\alpha^{1/2}\jmp{w }_\bN}_{L^2(\Ftime)^2}^2
+\Norm{\beta^{1/2}\jmp{\btau }_\bN}_{L^2(\Ftime)}^2
+\Norm{\alpha^{1/2}w }_{L^2(\FD)}^2
+\Norm{\beta^{1/2} \btau\cdot\bn_\Omega^x }_{L^2(\FN)}^2.
\label{eq:DGseminorm}
\end{align}

\begin{proposition}\label{prop:uniqueness}
Under assumption~\eqref{eq:dVpAssumpt}, the solution of the space--time DG
formulation~\eqref{eq:DG} exists and is unique.
\end{proposition}

\begin{proof}
We specialise the proof of \cite[Lemma~4.4]{MoRi05} to our setting.
We start by rewriting $\calA\DG$ in~\eqref{eq:DG} in the following
equivalent form, which is obtained by integration by parts in each element: 
\begin{align}\label{eq:Aibp}
\calA\DG(&v,\bsigma; w ,\btau )=
\sum_{K\in\calT_h}\int_K\bigg(\Big(
\nabla\cdot\bsigma+c^{-2}\der{v} t \Big)w +\Big(\nabla v+
\der{\bsigma} t\Big)\cdot\btau\bigg)\di V
\\\nonumber
&-\int_{\Fspa}\big(c^{-2}\jmp{v}_tw^+ +\jmp{\bsigma}_t\cdot \btau^+ \big)\di\bx
+\int_\FO (c^{-2}v  w +\bsigma \cdot\btau )\di \bx
\nonumber\\
&+\int_\Ftime \big(-\jmp{v}_\bN\cdot\mvl{\btau }-\jmp{\bsigma}_\bN\mvl{w}
+\alpha\jmp{v}_\bN\cdot\jmp{ w }_\bN+ \beta\jmp{\bsigma}_\bN\jmp{\btau }_\bN\big)\di S
\nonumber\\
&+\int_\FD \big(-v\btau\cdot\bn_\Omega^x +\alpha v w   \big) \di S
+\int_\FN\big(-\bsigma\cdot\bn_\Omega^x w
+\beta(\bsigma\cdot\bn_\Omega^x)(\btau\cdot\bn_\Omega^x)\big)\di S.
\nonumber
\end{align}

By taking $w=v$ and $\btau=\bsigma$, and summing the two expressions
for $\calA\DG$ given in~\eqref{eq:DG} and~\eqref{eq:Aibp}, we obtain
\begin{align}\label{eq:A=norm}
\calA\DG(&v,\bsigma;v,\bsigma)=\abs{(v,\bsigma)}\DG^2,
\end{align}
for the DG seminorm defined in \eqref{eq:DGseminorm}.

In order to prove uniquess of the discrete solution
of~\eqref{eq:DG}, due to the linearity of the problem, 
it is enough to consider the data $f=u_0=u_1=\gD=\gN=0$ so that
$\ell(w\hp,\btau\hp)=0$ for all $(w\hp,\btau\hp)\in\bVp\Th$, and 
prove that $(v\hp,\bsigma\hp)$ is zero in $Q$.

Choosing $(w\hp,\btau\hp)=(v\hp,\bsigma\hp)$, from \eqref{eq:A=norm} we deduce $\abs{(v\hp,\bsigma\hp)}\DG=0$, 
i.e.\ all jumps and traces present in the definition of the DG seminorm vanish.
Then \eqref{eq:Aibp} becomes
\begin{align*}
\calA\DG(v\hp,\bsigma\hp; w\hp ,\btau\hp )=
\sum_{K\in\calT_h}\int_K\bigg(\Big(
\nabla\cdot\bsigma\hp+c^{-2}\der{v\hp} t \Big)w\hp +\Big(\nabla v\hp+
\der{\bsigma\hp} t\Big)\cdot\btau\hp&\bigg)\di V=0\\
&
\quad \forall (w\hp,\btau\hp)\in \bVp\Th.
\end{align*}
Choosing $(w\hp,\btau\hp)=(\nabla\cdot\bsigma\hp+c^{-2}\der{v\hp} t,\nabla v\hp+\der{\bsigma\hp} t)\in\bVp\Th$, which is possible by \eqref{eq:dVpAssumpt}, we deduce that for all $K\in\calT_h$
\begin{equation}\label{eq:localPDE}
\nabla\cdot\bsigma\hp+c^{-2}\der{v\hp} t=0,\quad
\nabla v\hp+\der{\bsigma\hp} t=\bzero,
\end{equation}
i.e.\ $(v\hp,\bsigma\hp)$ is solution of the homogeneous PDE in each element.

In each $K\subset D_1$, the first time-slab, $(v\hp,\bsigma\hp)$ is a
polynomial solution (or an analytical solution) of~\eqref{eq:localPDE} with homogeneous initial conditions, so it vanishes in $K$.
Iterating over the time slabs and using that the jumps across
space-like faces vanish, we deduce that $(v\hp,\bsigma\hp)$ is zero in
the whole space--time cylinder $Q$. 

Existence of the solution follows from linearity and finite dimensionality.
\end{proof}

If the wave speed $c$ were a general (piecewise) smooth function of $\bx$, the condition \eqref{eq:dVpAssumpt} would not be satisfied.
Condition \eqref{eq:dVpAssumpt} is used in the proof of Proposition~\ref{prop:uniqueness} to show \eqref{eq:localPDE}, i.e.\ to ensure that the solution of the homogeneous variational problem are elementwise solution of the PDE.
This prevents the immediate extension of the current analysis to the case of general coefficients.

\begin{remark}\label{rem:stab}
Whenever the problem is driven by initial conditions only (namely,
$f=0$ and the boundary conditions on $\FD$ and $\FN$, if non empty,
are zero), the solution $(v\hp,\bsigma\hp)$
of the space--time DG
formulation~\eqref{eq:DG} satisfies
\[
\abs{(v\hp,\bsigma\hp)}\DG^2=\calA\DG(v\hp,\bsigma\hp;
v\hp,\bsigma\hp)=\ell(v\hp,\bsigma\hp)\le
\left(2\Norm{c^{-1}v_0}_{L^2(\FO)}^2
  +2\Norm{\bsigma_0}_{L^2(\FO)^2}^2
\right)^{1/2}
\abs{(v\hp,\bsigma\hp)}\DG,
\]
and therefore
\[\abs{(v\hp,\bsigma\hp)}\DG\le  \left(2\Norm{c^{-1}v_0}_{L^2(\FO)}^2 +2\Norm{\bsigma_0}_{L^2(\FO)^2}^2 \right)^{1/2}. \]
\end{remark}

\begin{remark}\label{rem:seminorm}
We stress that $\abs{\cdot}\DG$ defined in~\eqref{eq:DGseminorm} is only a seminorm 
in usual broken Sobolev spaces and DG discrete spaces.
In Trefftz discretization spaces, $\abs{\cdot}\DG$ is
actually a norm. Therefore, in the analysis of the Trefftz-DG method
in~\cite{MoPe18}, well-posedness and
quasi-optimality were straightforward consequences of the Lax--Milgram theorem. 
\end{remark}

\subsection{Abstract error estimates at discrete times} \label{s:AbstractAnalysis}

We firstly prove error bounds in the seminorm $\abs{\cdot}\DG$ restricted to the partial cylinder $Q_n$.
We denote this truncated seminorm 
$\abs{(w,\btau)}\DGQn:=\abs{(w\chi_{Q_n},\btau\chi_{Q_n})}\DG$, 
where $\chi_{Q_n}$ is the characteristic function of $Q_n$.
Similarly we define the truncated bilinear form
$\calA\DGQn(v,\bsigma;w,\btau):=\calA(v\chi_{Q_n},\bsigma\chi_{Q_n};w\chi_{Q_n},\btau\chi_{Q_n})$.

As $\abs{(w,\btau)}\DGQn^2$ contains the terms
$\N{c^{-1}w}_{L^2(\Omega\times\{t_n\})}^2$ and
$\N{\btau}_{L^2(\Omega\times\{t_n\})^2}^2$,
error bounds in the seminorm $\abs{\cdot}\DGQn$ imply
bounds on the $L^2$ norm of the error of the trace 
(as opposed to the time-jumps only) on the space-like interfaces
$\Ftn:=\Omega\times\{t_n\}$, namely a control of spatial volume integrals of the error in $v$ and in $\bsigma$ at each discrete time $t_n$.

Let $\vs,\vsh,\wth$ be the continuous solution, the discrete solution,
and an arbitrary discrete test field, respectively.
By \eqref{eq:A=norm} and Galerkin orthogonality
\begin{align}\nonumber
\abs{\wth-\vsh}\DGQn^2
&=\calA\DGQn\big(\wth-\vsh;\wth-\vsh\big)\\
&=\calA\DGQn\big(\wth-\vs;\wth-\vsh\big).
\label{eq:DG=A}
\end{align}
We now want to prove that the right-hand side is bounded by
$\abs{\wth-\vsh}\DGQn$.

If we choose $\wth\in\bVp$ so that 
\begin{equation}\label{eq:wthProj}
\int_K\bigg((w\hp-v)\Big(
\nabla\cdot\btau+c^{-2}\der{w} t \Big) +(\btau\hp-\bsigma)\cdot\Big(\nabla w+
\der{\btau} t\Big)\bigg)\di V=0\qquad \forall \wt\in \bVp(K),\ \forall K\subset Q_n,
\end{equation}
the volume terms in
$\calA\DGQn\big(\wth-\vs;\wth-\vsh\big)$ 
expressed as in~\eqref{eq:DG} vanish.
If assumption \eqref{eq:dVpAssumpt} is satisfied, taking $\wth$ as the orthogonal $L^2$ projection of $\vs$ on $\bVp(K)$ ensures~\eqref{eq:wthProj}.

\begin{remark}\label{rem:MonkRichter}
A straightforward application of the analysis of~\cite[\S5]{MoRi05}
is not possible here.
In \cite{MoRi05}, total degree space--time polynomial
spaces $\IP^p(K)$ were used, and the first-order wave operator
maps $\IP^p(K)^{d+1}$ into $\IP^{p-1}(K)^{d+1}$.
Then, 
a projection that is orthogonal to
$\IP^{p-1}(K)$ and $\IP^p(F)$, where $F$ is a face of $K$, was
employed in order to define $\wth$ satisfying~\eqref{eq:wthProj}.
In our case, where space--time tensor product polynomials are used, 
the first-order wave operator does not map to lower-degree polynomial spaces.
Consequently, the $L^2$ projection into the orthogonal to the image
space does not ``save'' sufficiently many degrees of freedom to be able to use them to cancel some trace.
Therefore, we define $\wth$ by projecting 
against the whole discrete space, and
we exploit the \emph{coercivity in seminorm} property in the first
line of~\eqref{eq:DG=A}, which was not used in \cite{MoRi05}; see
\S\ref{s:QOsmooth} and \S\ref{s:QOrough} below.
\end{remark}

For the sake of clarity, we firstly deduce a quasi-optimality bound under the assumption that $\vs$ does not present corner singularities (\S\ref{s:QOsmooth}), and then in the more realistic case involving weighted spaces (\S\ref{s:QOrough}).
These ``quasi-optimality'' bounds (i.e.\ \eqref{eq:AbstractErrorS} and \eqref{eq:AbstractErrorB}) differ from classical ones (from C\'ea lemma) in that they allow to control the Galerkin error in terms of the $L^2$-projection  error, as opposed to the best approximation error in the energy norm.

\subsubsection{The smooth solution case}\label{s:QOsmooth}

If the IBVP data are artificially chosen so that the solution $(v,\bsigma)\in C^0([0,T]; H^1\OO)^{3}$, then all traces on the mesh faces possess $L^2$ summability.
Therefore, if $\wth$ satisfies \eqref{eq:wthProj} (e.g.\ because assumption \eqref{eq:dVpAssumpt} is verified and $\wth$ is the orthogonal $L^2$ projection of $\vs$ on $\bVp\Th$), by using the Cauchy-Schwarz inequality, one can see that the bilinear form in \eqref{eq:DG} admits the following upper bound:
\begin{equation}\label{eq:boundAsmooth}
\begin{split}
&\calA\DGQn\big(\wth-\vs;\wth-\vsh\big)\\
&\qquad\qquad\qquad\le 2 \abs{\wth-\vs}\DGQnp\abs{\wth-\vsh}\DGQn,
\end{split}
\end{equation}
where the $\abs{\cdot}\DGQnp$ seminorm is defined by
\begin{align*}
\abs{\wt}\DGQnp^2:=&\abs{\wt}\DGQn^2
+2\N{c^{-1}w^-}_{L^2(\Fspa\cap Q_n)}^2
+2\N{\btau^-}_{L^2(\Fspa\cap Q_n)^2}^2
\\
&+\Norm{\beta^{-1/2}\mvl{w}}_{L^2(\Ftime\cap Q_n)}^2
+\Norm{\beta^{-1/2} w }_{L^2(\FN\cap Q_n)}^2\\
&
+\Norm{\alpha^{-1/2}\mvl{\btau}\cdot\bn^x_F}_{L^2(\Ftime\cap Q_n)}^2
+\Norm{\alpha^{-1/2}\btau\cdot \bn_\Omega^x}_{L^2(\FD\cap Q_n)}^2.
\end{align*} 
Therefore, from~\eqref{eq:DG=A} and~\eqref{eq:boundAsmooth} we get
\begin{equation}\label{eq:ProjErrS}
\abs{\wth-\vsh}\DGQn\le 2 \abs{\vs-\wth}\DGQnp.
\end{equation}

In the following proposition, we prove that the
$\abs{\cdot}\DGQn$ seminorm of the error (thus
the $L^2$ norm in space of the error at
each discrete time $t_n$) is bounded by the $DG^+$
seminorm on the partial cylinder $Q_n$
of the $L^2$ projection error of the solution into the discrete space.

\begin{proposition}\label{prop:QOsmooth}
Let $\vs$ and $\vsh$ be the solutions of \eqref{eq:IBVP} and \eqref{eq:DG}, respectively, 
and 
assume that $(v,\bsigma)\in C^0([0,T]; H^1\OO)^{3}$.
Let $\wth$ be the orthogonal $L^2$ projection of $\vs$ into $\bVp\Th$.
Assume that $\bVp\Th$ is a piecewise polynomial space and
that~\eqref{eq:dVpAssumpt} holds true.
Then, for each discrete time $t_n$, we have the error bounds
\begin{align}\nonumber
\frac12\N{c^{-1}(v-v\hp)}_{L^2(\Omega\times\{t_n\})}
+\frac12\N{\bsigma-\bsigma\hp}_{L^2(\Omega\times\{t_n\})^2}
&\le\abs{\vs-\vsh}\DGQn
\\&\le 3\abs{\vs-\wth}\DGQnp.
\label{eq:AbstractErrorS}
\end{align}
\end{proposition}
\begin{proof}
The first inequality comes from the definition of the $\abs{\cdot}\DGQn$ seminorm, which contains the $L^2$ norm of the traces on $\Omega\times \{t_n\}$ (as opposed to the full $\abs{\cdot}\DG$ norm, which contains the $L^2$ norm of the time-jumps only on $\Omega\times \{t_n\}$), and the bound $\frac12(A+B)\le(\frac12(A^2+B^2))^{1/2}$ for all $A,B\in\IR$.
Assumption \eqref{eq:dVpAssumpt} ensures identity \eqref{eq:wthProj} and thus bounds \eqref{eq:boundAsmooth}--\eqref{eq:ProjErrS}.
The second inequality then follows from the triangle inequality, 
the bound \eqref{eq:ProjErrS}, and $\abs{\cdot}\DGQn\le\abs{\cdot}\DGQnp$.
\end{proof}

\subsubsection{The general case}\label{s:QOrough}
We now allow for corner singularities and assume that $\vs$ satisfies \eqref{eq:vsRegularity2} for some $k_t,k_x\in\IN$.

In order to control the inter-element terms on
the right-hand side of \eqref{eq:DG=A}, 
we cannot use directly the Cauchy--Schwarz inequality 
as in the derivation of~\eqref{eq:boundAsmooth}, because for 
$\btau\in C^0((0,T); H^{1,1}_\udelta\OO^2)$ we have $\btau\in L^1(F)^2$ for $F\in\IFtime_\angle$ thanks to~\eqref{eq:H11trace}, 
but not necessarily $\btau\in L^2(F)^2$.
So we opt to apply the H\"older inequality in $L^1$--$L^\infty$ 
\cite[Appendix]{TWDiss}:
\begin{align*}
\abs{\int_F \mvl\btau\cdot \jmp{w}_\bN\di S} 
&\le 
\N{\mvl{\btau}}_{L^2(F_t;L^1(F_\bx))}
\N{\jmp{w}_\bN}_{L^2(F_t;L^\infty(F_\bx))}
\\&\le
C_{\infty|2} h_{F_\bx}^{-1/2} 
\N{\mvl{\btau}}_{L^2(F_t;L^1(F_\bx))}
\N{\jmp{w}_\bN}_{L^2(F_t;L^2(F_\bx))}
\end{align*}
for all $F\in\IFtime_\angle$, $\btau\in L^2((0,T); H^{1,1}_\udelta\OO^2)$ and elementwise-polynomial $w$.
Here $C_{\infty|2}$ stems from the inverse inequality \eqref{eq:inverseL1L2Loo} for the polynomial space,
thus it depends on the maximal polynomial degree in space admitted for $w$ in the corner elements.
The term $\int_\FD\btau\cdot\bn^x_\Omega w\di S$ is treated similarly.
The other terms appearing in $\calA\DG(\cdot,\cdot)$ can be controlled by the Cauchy--Schwarz inequality, as $\jmp{\bsigma}_\bN=0$ on $\Ftime$, $\bsigma\cdot\bn^x_\Omega=0$ on $\FN$, and all other traces are in $L^2(F)$ for the corresponding face $F$.

Proceeding as in \S\ref{s:QOsmooth}, we control the bilinear form in \eqref{eq:DG} as
\begin{equation}\label{eq:boundArough}
\begin{split}
&\calA\DGQn\big(\wth-\vs;\wth-\vsh\big)\\
&\qquad\qquad\qquad\le 2 
C_{\infty|2}\abs{\wth-\vs}\DGQnp\abs{\wth-\vsh}\DGQn
\end{split}
\end{equation}
where now 
\begin{equation}\label{eq:DG+general}
  \begin{split}
&\hspace{-5mm}\abs{\wt}\DGQnp^2:=\abs{\wt}\DGQn^2
+2\N{c^{-1}w^-}_{L^2(\Fspa\cap Q_n)}^2
+2\N{\btau^-}_{L^2(\Fspa\cap Q_n)^2}^2
\\
&+\Norm{\beta^{-1/2}\mvl{w}}_{L^2(\Ftime\cap Q_n)}^2
+\Norm{\beta^{-1/2} w }_{L^2(\FN\cap Q_n)}^2\\
&+\sum_{F\in\IFtime_\angle(Q_n)} h_{F_\bx}^{-1}
\N{\alpha^{-1/2}\mvl{\btau}\cdot\bn^x_F}_{L^2(F_t; L^1(F_\bx))}^2
+\Norm{\alpha^{-1/2}\mvl{\btau}\cdot\bn^x_F}_{L^2(\Ftime\cap\Fh^\odot\cap Q_n)}^2\\
&+\sum_{F\in\IFD_\angle(Q_n)}
h_{F_\bx}^{-1}\N{\alpha^{-1/2}\btau\cdot\bn_\Omega^x}_{L^2(F_t; L^1(F_\bx))}^2
+\Norm{\alpha^{-1/2}\btau\cdot \bn_\Omega^x}_{L^2(\FD\cap\Fh^\odot\cap Q_n)}^2.
\end{split}
\end{equation}
Therefore, from~\eqref{eq:DG=A} and~\eqref{eq:boundArough} we get
\begin{equation}\label{eq:ProjErr}
\abs{\wth-\vsh}\DGQn\le 2C_{\infty|2}\abs{\vs-\wth}\DGQnp.
\end{equation}

The following result is proved exactly as Proposition~\ref{prop:QOsmooth}, 
using \eqref{eq:ProjErr} instead of \eqref{eq:ProjErrS} and the corresponding modified $\abs{\cdot}\DGQnp$ seminorm.

\begin{proposition}\label{prop:QO}
Let $\vs$ and $\vsh$ be the solutions of \eqref{eq:IBVP} and \eqref{eq:DG}, respectively.
Assume that $\vs$ satisfies the 
weighted regularity condition \eqref{eq:vsRegularity2} for some $k_t,k_x\in\IN$.
Let $\wth$ be the orthogonal $L^2$ projection of $\vs$ into $\bVp\Th$.
Assume that $\bVp\Th$ is a piecewise polynomial space and
that~\eqref{eq:dVpAssumpt} holds true.
Then, for each discrete time $t_n$, we have the error bounds
\begin{align}\nonumber
\frac12\N{c^{-1}(v-v\hp)}_{L^2(\Omega\times\{t_n\})}
+\frac12\N{\bsigma-\bsigma\hp}_{L^2(\Omega\times\{t_n\})^2}
&\le\abs{\vs-\vsh}\DGQn
\\&\le (2 C_{\infty|2}+1)\abs{\vs-\wth}\DGQnp.
\label{eq:AbstractErrorB}
\end{align}
\end{proposition}

Recall from \eqref{eq:InvIneqC} that $C_{\infty|2}\le p+1$,
$p$ being the maximal polynomial degree in the spatial variable of the elements of $\bVp(\calT_h^\angle)$,
thus the abstract error bound \eqref{eq:AbstractErrorB} can be made fully explicit, and the bounding constant depends linearly on $p$.
Note that when a local mesh refinement strategy is used, 
since corner elements are taken to be small, $p$ can be chosen to be 0 or 1.

\begin{remark}[Energy bounds]
We briefly discuss the dissipation properties of the proposed scheme.
We define the energy of a pair $\wt$ at time $t\in[0,T]$ as 
$$
\calE(t;w,\btau):=\frac12\int_\Omega(c^{-2}w^2(\bx,t)+|\btau(\bx,t)|^2)\di\bx.
$$
Then, if $\vs$ is a solution of the IBVP \eqref{eq:IBVP} with $f=0$, we have the identity
$\calE(t;v,\bsigma)=\calE(0;v,\bsigma)-\int_{\deO\times(0,t)}v\bsigma\cdot\bn^x_\Omega\di S$
(e.g.\ \cite[equation~(15)]{MoPe18}).
If, moreover, the boundary conditions are homogeneous (i.e., $\gD=\gN=0$) 
the total energy is preserved, $\calE(t;v,\bsigma)=\calE(0;v,\bsigma)$.
Proceeding exactly as in \cite[(16)--(17)]{MoPe18} we see that the DG
method is dissipative and we quantify the energy dissipated in $Q_n$:
for $\vsh$ the solution of \eqref{eq:DG} and $1\le n\le N$,
\begin{align*}
\calE(t_n;v_h,\bsigma_h)=&\calE(0;v_0,\bsigma_0)
-\frac12 \Norm{c^{-1}\jmp{v_h}_t}_{L^2(\Fspa\cap Q_n)}^2
-\Norm{\alpha^{1/2}\jmp{v_h}_\bN}_{L^2(\Ftime\cap Q_n)^2}^2
-\Norm{\alpha^{1/2}v_h}_{L^2(\FD\cap Q_n)}^2\!
\\&\hspace{16mm}
-\frac12\Norm{\jmp{\bsigma}_t\rule{0pt}{3mm}}_{L^2(\Fspa\cap Q_n)^2}^2 
-\Norm{\beta^{1/2}\jmp{\bsigma}_\bN}_{L^2(\Ftime\cap Q_n)}^2
-\Norm{\beta^{1/2} \bsigma\cdot\bn_\Omega^x }_{L^2(\FN\cap Q_n)}^2.
\end{align*}
This means that the energy dissipated by the discrete solution 
is proportional to the jumps of the solution on the mesh skeleton and to
the mismatch with the weakly enforced (homogeneous) boundary conditions.
\end{remark}

\section{Projection and Galerkin error estimates}
\label{s:ErrorBounds}
Given a vector of elemental polynomial degrees 
$\bp=\big((p_{x,K}^v,p_{t,K}^v,p_{x,K}^\bsigma,p_{t,K}^\bsigma)\in\IN_0^4,\; K\in \calT_h\big)$, 
we choose as trial and test space the piecewise-polynomial space (with $\otimes$ denoting
the algebraic tensor product)
\begin{equation}\label{eq:Vp}
\bV_\bp\Th=
\prod_{K=K_\bx\times I_n\in\calT_h}
\Big(\IP^{p_{x,K}^v}(K_\bx)\otimes\IP^{p_{t,K}^v}(I_{n})\Big)
\times
\Big(\IP^{p_{x,K}^\bsigma}(K_\bx)\otimes\IP^{p_{t,K}^\bsigma}(I_{n})\Big)^2.
\end{equation}
Here and in what follows, 
the space $\IP^{p_{x,K}^\bullet}(K_\bx)$, for $\bullet\in\{v,\bsigma\}$, 
of polynomials of degree at most $p_{x,K}^\bullet$ in two variables can be replaced 
by the space $\mathbb{Q}^{p_{x,K}^\bullet}(K_\bx)$ of polynomials of the same degree in each of the two variables.

Bound \eqref{eq:AbstractErrorB} states that we can control the DG seminorm of the Galerkin error by the DG$^+$ seminorm of the error of its elementwise $L^2$ projection in $\bV_\bp\Th$.
So we need to derive error estimates for the \emph{traces} of the
\emph{volume} $L^2$ projection. The $p$-version of these estimates was 
derived in \cite{HSS02a} and \cite{Che12} for cubes and simplices,
respectively; here we aim at the $h$-version, for which simpler scaling argument are sufficient.
We partly follow \cite[\S5]{MuScSc18}.

\subsection{Bounds of DG seminorms in terms of volume norms}\label{s:SeminormToVolume}

As this section contains auxiliary results, we consider directly the general case of solutions admitting corner singularities.
If $\vs$ is smooth, in particular, if $\bsigma\in C^0([0,T];
H^1\OO^2)$, then in \eqref{eq:DGpBound} and \eqref{eq:DGpBound2} below,
the summations over faces in $\calF_h^\angle$ or elements in $\calT_h^\angle$ can be dropped, 
if simultaneously the terms over (subsets of) $\calF_h^\odot$ and $\calT_h^\odot$ are extended to (analogous subsets of) the whole of $\calF_h$ and $\calT_h$.

We define the space containing both the exact solution and the discrete functions
\begin{align*}
V_+:=&\Big(C^0\big([0,T]; H^{2,2}_\udelta\OO\big)\times C^0\big([0,T]; H^{1,1}_{\udelta,N}\OO^2\big)\Big)+ \bV_\bp\Th,
\quad \text{where}\\
H^{1,1}_{\udelta,N}\OO^2:=&\{\btau\in H^{1,1}_\udelta\OO^2, \btau\cdot\bn^x_\Omega=0 \oon \FN\}.
\end{align*}
Here we want to derive a bound of $\abs{\wt}\DGQnp$ in terms of elementwise sums of traces, 
tracking the dependence on spatial and temporal meshsize for all $\wt\in V_+$.
For all $\wt\in V_+$ there exists a (not necessarily unique) decomposition
in continuous\,+\,discrete components:
\begin{align}
\label{eq:wt=h+tilde}
\wt=(\tilde w,\tilde \btau)+\wth,\qquad
&(\tilde w,\tilde \btau)\in
C^0([0,T]; H^{2,2}_\udelta\OO) \times C^0([0,T]; H^{1,1}_{\udelta,N}\OO^2),
\\&\wth\in \bV_\bp\Th. 
\nonumber
\end{align}
Thus for $\wt\in V_+$, taking into account the
definition~\eqref{eq:DG+general}, we get
\begin{align}\label{eq:DGpBound}
&\abs{\wt}\DGQnp^2
\lesssim
\sum_{K=K_\bx\times I_{n'}\in\calT_h(Q_n)}\Bigg[
\N{c^{-1}w}_{L^2(K_\bx\times\{t_{n'-1},t_{n'}\})}^2
+\N{\btau}_{L^2(K_\bx\times\{t_{n'-1},t_{n'}\})^2}^2
\\
&\hspace{20mm}+\sum_{F\in\IFtime(K)\cup\IFD(K)} 
\Norm{\alpha^{1/2}w}_{L^2(F)}^2
+\sum_{F\in\IFtime(K)\cup\IFN(K)}
\Norm{\beta^{-1/2}w}_{L^2(F)}^2
\nonumber\\
&\hspace{20mm}+\sum_{F\in\IFtime_\odot(K)\cup\IFD_\odot(K)}
\Norm{\alpha^{-1/2}\btau\cdot \bn_F^x}_{L^2(F)}^2
+\sum_{F\in\IFtime_\angle(K)\cup\IFD_\angle(K)}
h_{F_\bx}^{-1}\N{\alpha^{-1/2}\btau\cdot\bn^x_F}_{L^2(F_t; L^1(F_\bx))}^2
\Bigg]
\nonumber\\
&\hspace{30mm}+\Norm{\beta^{1/2}\jmp{\btau }_\bN}_{L^2(\Ftime\cap Q_n)}^2
+\Norm{\beta^{1/2} \btau\cdot\bn_\Omega^x }_{L^2(\FN\cap \overline{Q_n})}^2,
\nonumber
\end{align}
where all traces are taken from the element $K$.
Here we used the $L^1$-$L^2$ inequality \eqref{eq:L1L2Loo} to treat the term containing $\mvl{\btau}$ on some of the faces $F\in\IFtime_\angle$; in particular, on the faces in the form $F=\deK_\angle\cap\deK_\odot$ for $K_\angle\subset\calT_h^\angle$ and $K_\odot\subset\calT_h^\odot$, the trace of $\btau|_{K_\odot}$ is lifted from the $L^2(F_t;L^1(F_\bx))$ norm to the $L^2(F)$ norm.

The two terms containing $\beta^{1/2}\btau$ are non zero only for the discrete component $\btau\hp$ of $\btau$, recall \eqref{eq:wt=h+tilde}.
Since this is a polynomial of given degree, we control these terms using the inverse inequality \eqref{eq:inverseL1L2Loo}.
The terms on the Neumann boundary can be controlled as follows:
\begin{align*}
&\Norm{\beta^{1/2} \btau\cdot\bn_\Omega^x }_{L^2(\FN\cap \overline{Q_n})}^2 
\overset{\widetilde\btau\cdot\bn_\Omega^x=0}=\Norm{\beta^{1/2} \btau\hp\cdot\bn_\Omega^x }_{L^2(\FN\cap\overline{Q_n})}^2\\
&=\sum_{F\in\IFN_\odot(Q_n)}
\Norm{\beta^{1/2} \btau\hp\cdot\bn_\Omega^x }_{L^2(F)}^2
+\sum_{F\in\IFN_\angle(Q_n)}
\Norm{\beta^{1/2} \btau\hp\cdot\bn_\Omega^x }_{L^2(F)}^2\\
&\overset{\eqref{eq:inverseL1L2Loo}}\le\sum_{F\in\IFN_\odot(Q_n)}
\Norm{\beta^{1/2} \btau\hp\cdot\bn_\Omega^x }_{L^2(F)}^2
+C_{2|1}\sum_{F\in\IFN_\angle(Q_n)}
h_{F_\bx}^{-1}
\Norm{\beta^{1/2} \btau\hp\cdot\bn_\Omega^x }_{L^2(F_t;L^1(F_\bx))}^2\\
&\overset{\widetilde\btau\cdot\bn_\Omega^x=0}=\sum_{F\in\IFN_\odot(Q_n)}
\Norm{\beta^{1/2} \btau\cdot\bn_\Omega^x }_{L^2(F)}^2
+C_{2|1}\sum_{F\in\IFN_\angle(Q_n)}
h_{F_\bx}^{-1}
\Norm{\beta^{1/2} \btau\cdot\bn_\Omega^x }_{L^2(F_t;L^1(F_\bx))}^2\\
&= \sum_{K=K_\bx\times I_{n'}\in\calT_h(Q_n)}\bigg[
\sum_{F\in\IFN_\odot(K)}\Norm{\beta^{1/2} \btau\cdot\bn_\Omega^x }_{L^2(F)}^2
+C_{2|1}\sum_{F\in\IFN_\angle(K)}h_{F_\bx}^{-1}
\Norm{\beta^{1/2} \btau\cdot\bn_\Omega^x }_{L^2(F_t;L^1(F_\bx))}^2
\bigg].
\end{align*}
The jump term is controlled similarly: 
\begin{align*}
&\Norm{\beta^{1/2}\jmp{\btau }_\bN}_{L^2(\Ftime\cap Q_n)}^2 
\overset{\jmp{\widetilde \btau}_\bN=0}=
\Norm{\beta^{1/2}\jmp{\btau\hp}_\bN}_{L^2(\Ftime\cap Q_n)}^2\\
&=\sum_{F\in\IFtime_\odot(Q_n)} \Norm{\beta^{1/2}\jmp{\btau\hp}_\bN}_{L^2(F)}^2
+\sum_{F\in\IFtime_\angle(Q_n)} \Norm{\beta^{1/2}\jmp{\btau\hp}_\bN}_{L^2(F)}^2
\\
&\overset{\eqref{eq:inverseL1L2Loo}}\le 
\sum_{F\in\IFtime_\odot(Q_n)} \Norm{\beta^{1/2}\jmp{\btau\hp}_\bN}_{L^2(F)}^2
+C_{2|1}
\sum_{F\in\IFtime_\angle(Q_n)}h_{F_\bx}^{-1} \Norm{\beta^{1/2}\jmp{\btau\hp}_\bN}_{L^2(F_t;L^1(F_\bx))}^2
\\
&\overset{\jmp{\widetilde \btau}_\bN=0}= 
\sum_{F\in\IFtime_\odot(Q_n)} \Norm{\beta^{1/2}\jmp{\btau}_\bN}_{L^2(F)}^2
+ C_{2|1}
\sum_{F\in\IFtime_\angle(Q_n)}h_{F_\bx}^{-1} \Norm{\beta^{1/2}\jmp{\btau}_\bN}_{L^2(F_t;L^1(F_\bx))}^2
\\
&\overset{\eqref{eq:L1L2Loo}}\le 
(1+C_{2|1})\sum_{K=K_\bx\times I_{n'}\in\calT_h^\odot(Q_n)}
\sum_{F\in\IFtime(K)} \Norm{\beta^{1/2}\btau\cdot\bn_K^x}_{L^2(F)}^2
\\&\hspace{15mm}
+C_{2|1}
\sum_{K=K_\bx\times I_{n'}\in\calT_h^\angle(Q_n)}
\sum_{F\in\IFtime(K)}h_{F_\bx}^{-1}
\Norm{\beta^{1/2} \btau\cdot\bn_K^x }_{L^2(F_t;L^1(F_\bx))}^2,
\end{align*}
where in the last step we used again the $L^1$-$L^2$ inequality \eqref{eq:L1L2Loo} to treat
the contributions from $\calT_h^\odot$ to $\jmp{\btau}_\bN$ on the
faces of the form $F=\deK_\angle\cap\deK_\odot$ for
$K_\angle\subset\calT_h^\angle$ and $K_\odot\subset\calT_h^\odot$.

For all elements $K=K_\bx\times I_n\in\calT_h$, 
the standard weighted trace inequality, applied in the time and space directions independently, reads
\begin{align}\label{eq:TraceT}
\N{\varphi}_{L^2(K_\bx\times\{t_{n-1},t_{n}\})}^2
&\lesssim h_n^{-1} \N{\varphi}_{L^2(K)}^2+h_n \abs{\varphi}_{H^1(I_n;L^2(K_\bx))}^2
\qquad \forall \varphi\in H^1\big(I_n;L^2(K_\bx)\big),\\
\N{\varphi}_{L^2(\deK_\bx\times I_n)}^2
&\lesssim h_{K_\bx}^{-1} \N{\varphi}_{L^2(K)}^2
+h_{K_\bx}\abs{\varphi}_{L^2(I_n;H^1(K_\bx))}^2
\qquad \forall \varphi\in L^2\big(I_n;H^1(K_\bx)\big),
\label{eq:TraceX}
\end{align}
where the hidden constant in \eqref{eq:TraceX} only depends on the
shape-regularity parameter of the space mesh 
(see e.g.\ \cite[\(1.6.6\)]{BRS94} or \cite[Lemma~2]{MoPe18} for star-shaped elements).
The corresponding result for weighted spaces follows 
from Lemma~\ref{lem:H11} (\!\!\cite[Lemma~1.3.2c]{TWDiss}):
\begin{align}\label{eq:TraceH11d}
\N{\varphi}_{L^2(I_n;L^1(\deK_\bx))}^2
&\lesssim \N{\varphi}_{L^2(K)}^2
+h_{K_\bx}^{2-2\delta_K}\abs{\varphi}_{L^2(I_n;H^{1,1}_\udelta(K_\bx))}^2
\qquad \forall \varphi\in L^2\big(I_n;H^{1,1}_\udelta(K_\bx)\big).
\end{align}

We recall that the numerical flux parameters $\alpha,\beta$ are
assumed to be constant on each element face on which they are defined.
For each element $K\in\calT_h$, we write
\begin{align}\nonumber
\alpha^K_{\min}:=&\min_{F\in\IFtime(K)\cup\IFD(K)}\{\alpha|_F\},
\qquad\qquad
\beta^K_{\min}:=\min_{F\in\IFtime(K)\cup\IFN(K)}\{\beta|_F\},
\\
\alpha^K_{\max}:=&\max_{F\in\IFtime(K)\cup\IFD(K)}\{\alpha|_F\},
\qquad\qquad
\beta^K_{\max}:=\max_{F\in\IFtime(K)\cup\IFN(K)}\{\beta|_F\},
\label{eq:AlphaBetaK}
\\
\tta_K:=&\max\{(\alpha^K_{\max})^{1/2},(\beta^K_{\min})^{-1/2}\},
\qquad
\ttb_K:=\max\{(\beta^K_{\max})^{1/2},(\alpha^K_{\min})^{-1/2}\}.
\nonumber
\end{align}

As the spatial elements are assumed not to have degenerating faces, 
there holds
\begin{equation}\label{eq:hFhK}
h_{F_\bx}\approx h_{K_\bx},\quad\text{for } F\in\IFtime(K)\cup\IFD(K)\cup\IFN(K), \quad K\in\calT_h.
\end{equation}

We have now all the ingredients to derive an upper bound on the $\abs{\cdot}\DG$ seminorm 
in terms of standard space--time (weighted, Bochner) Sobolev norms of its argument over the mesh elements.
Note that we cannot expect to obtain a uniform bound in term of the spatial and temporal meshsize (e.g.\ for $w=1$, $\btau=\bzero$, and constant $\alpha$ and $\beta$, the seminorm $\abs{\wt}\DGp^2$ is proportional to the 2-dimensional Hausdorff measure of $\calF_h$, which is not bounded for $h_\bx,h_t\searrow0$).

\begin{proposition}\label{prop: DGpBound2}
For all $\wt\in V_+$, the following bound holds true:
\begin{align}\label{eq:DGpBound2}
&\abs{\wt}\DGQnp^2\\&\lesssim\!\!
\sum_{K=K_\bx\times I_{n'}\in\calT_h(Q_n)}\!\bigg[
h_{n'}^{-1}\Big(\N{c^{-1}w}_{L^2(K)}^2
+\N{\btau}_{L^2(K)^2}^2\Big)
+h_{n'}\Big(\abs{c^{-1}w}_{H^1(I_{n'};L^2(K_\bx))}^2
+\abs{\btau}_{H^1(I_{n'};L^2(K_\bx)^2)}^2\Big)
\nonumber\\
&\hspace{30mm}
+h_{K_\bx}^{-1}\N{\tta_K w}_{L^2(K)}^2
+h_{K_\bx}\abs{\tta_K w}_{L^2(I_{n'};H^1(K_\bx))}^2
+\underbrace{h_{K_\bx}^{-1}}_{h_{K_\bx}\lesssim h_{F_\bx}}
\N{\ttb_K\btau}_{L^2(K)^2}^2\bigg]
\nonumber\\
&+\sum_{K=K_\bx\times I_{n'}\in\calT_h^\odot(Q_n)}
h_{K_\bx}\abs{\ttb_K\btau}_{L^2(I_{n'};H^1(K_\bx)^2)}^2
+\sum_{K=K_\bx\times I_{n'}\in\calT_h^\angle(Q_n)}
\underbrace{h_{K_\bx}^{1-2\delta_K}}_{h_{K_\bx}\lesssim h_{F_\bx}}\abs{\ttb_K\btau}_{L^2(I_{n'};H^{1,1}_\udelta(K_\bx)^2)}^2,
\nonumber
\end{align}
where the hidden constant depends linearly on $\max\{p_{x,K}^\bsigma,\; K\in \calT_h^\angle\}$ through 
the inverse inequality constant $C_{2|1}$ (recall~\eqref{eq:InvIneqC}).
\end{proposition}

\begin{proof}
From the bound \eqref{eq:DGpBound}, after treating the terms with
$\beta^{1/2}\btau$ as described, by using the trace inequalities
\eqref{eq:TraceT}, \eqref{eq:TraceX} and \eqref{eq:TraceH11d},
notation \eqref{eq:AlphaBetaK}, and assumption \eqref{eq:hFhK},
for all $\wt\in V_+$, we obtain~\eqref{eq:DGpBound2}.
\end{proof}

We note that the only polynomial degree affecting the bounding
constant in \eqref{eq:DGpBound2} is the degree of the space associated with the corner
elements, which is typically taken to be $0$ or $1$.
So in this case the bounding constant only depends on the element
shapes.

\subsection{\texorpdfstring{$L^2$}{L2} projection error estimates for a scalar function}\label{s:ProjEstimates}

In this section we prove error bounds in Bochner norms for the $L^2$ projection on polynomial spaces.

\subsubsection{The smooth case}\label{s:L2smooth}

For functions without corner singularities we have the following result.

\begin{proposition}\label{prop:L2proj}
Let $K=K_{\bx}\times I_n\in\calT_h$ be a prismatic space--time element,
i.e.\ $K_\bx$ is a shape-regular polygon,
with $h_{K_\bx}=\diam(K_{\bx})$ and $h_n=\abs{I_n}$. Denote by
$\Pi_\bp$, $\bp=(p_x,p_t)\in\IN_0^2,$ the $L^2(K)$-orthogonal
projection: $L^2(K)\to\IP^\bp(K)=\IP^{p_x}(K_\bx)\otimes\IP^{p_t}(I_n)$.

Let 
$\varphi\in H^{k_t+1}(I_n;L^2(K_{\bx}))\cap L^2(I_n;H^{k_x+1}(K_{\bx}))$, $k_x,k_t\in\IN_0$. 
Then
\begin{align}\nonumber
\N{\varphi-\Pi_\bp \varphi}_{L^2(I_n;L^2(K_\bx))}+&
h_n \abs{\varphi-\Pi_\bp \varphi}_{H^1(I_n;L^2(K_\bx))}+
h_{K_\bx}\abs{\varphi-\Pi_\bp \varphi}_{L^2(I_n;H^1(K_\bx))}\\
&\lesssim
h_n^{s_t+1}\abs{\varphi}_{H^{s_t+1}(I_n;
L^2(K_\bx))}+h_{K_\bx}^{s_x+1}\abs{\varphi}_{L^2(I_n;H^{s_x+1}(K_\bx))},
\label{eq:L2proj}
\end{align}
where $s_t:=\min\{k_t,p_t\}$, $s_x:=\min\{k_x,p_x\}$, and
the hidden constants 
are independent of $h_{K_\bx}$ and $h_n$, but depend on the polynomial
degrees $p_t$ and $p_x$, and on the shape-regularity constant of $K_\bx$.
\end{proposition}

\begin{proof}
Consider a reference element $\wK=\wK_{\wbx}\times\wI$,  and denote by
$\wPi_\bp$, $\bp\in\IN_0^2$, the $L^2(\wK)$-orthogonal
projection: $L^2(\wK)\to\IP^\bp(\wK)=\IP^{p_x}(\wK_\wbx)\otimes\IP^{p_t}(\wI)$. 
Let $\whvarphi$ be in 
$H^{k_t+1}(\wI;L^2(\wK_{\wbx}))\cap L^2(\wI;H^{k_x+1}(\wK_{\wbx}))$.
We prove that
\begin{align}\nonumber
\N{\whvarphi-\wPi_\bp \whvarphi}_{L^2(\wI;L^2(\wK_\wbx))}
+\abs{\whvarphi-\wPi_\bp \whvarphi}_{H^1(\wI;L^2(\wK_\wbx))}
&+\abs{\whvarphi-\wPi_\bp \whvarphi}_{L^2(\wI;H^1(\wK_\wbx))}\\
&\lesssim
\abs{\whvarphi}_{H^{s_t+1}(\wI; L^2(\wK_\wbx))}+
\abs{\whvarphi}_{L^2(\wI;H^{s_x+1}(\wK_\wbx))}.
\label{eq:hatBound}
\end{align}
The assertion \eqref{eq:L2proj} then follows from a scaling argument.

We recall that the Peetre--Tartar lemma, as formulated in \cite[Lemma~A.38]{EG04}, states that if $X,Y,Z$ are Banach spaces, $A:X\to Y$ is injective, $T:X\to Z$ is compact, 
$\N{x}_X\lesssim\N{Ax}_Y+\N{Tx}_Z$ for all $x\in X$, then 
$\N{x}_X\lesssim\N{Ax}_Y$ for all $x\in X$.
Within this setting, we fix 
\begin{align*}
X&=H^{s_t+1}(\wI;L^2(\wK_{\wbx}))\cap L^2(\wI;H^{s_x+1}(\wK_{\wbx})),\\ 
Y&=\Big(\IP^{s_x}(\wK_\wbx)\otimes\IP^{s_t}(\wI)\Big)
\times L^2(\wK) \times L^2(\wK)^{s_x+2},\\
Z&=L^2(\wK),\\
A&:\whpsi\mapsto\Big(\wPi_{(s_x,s_t)}\whpsi, \quad\partial_{\widehat t}^{s_t+1} \whpsi,
\quad D^{s_x+1}_\wbx \whpsi\Big)
\end{align*}
where $D^{s_x+1}_\wbx\whpsi$ denotes the collection of the $s_x+2$ space partial derivatives of order exactly $s_x+1$ of $\whpsi$,
and $T:X\to Z$ is the natural embedding.

The intersection $L^2(\wI;H^1(\wK_\wbx))\cap H^1(\wI;L^2(\wK_\wbx))$ is isomorphic to $H^1(\wK)$, thus it is compactly\footnote{Recall that if $U\subset V$ is a compact inclusion of Banach spaces, the embedding of the corresponding Bochner spaces $L^2(\wI;U)\subset L^2(\wI;V)$ is not compact; see a counterexample in \cite[Appendix~A]{PeSp18}.} embedded in $L^2(\wK)$.
Since $X\subset L^2(\wI;H^1(\wK_\wbx))\cap H^1(\wI;L^2(\wK_\wbx))$,
the embedding operator $T:X\to Z$, is compact.

If $A\whpsi=0$ then $\whpsi$ has vanishing $(s_t+1)$th time derivative and
$(s_x+1)$th space derivatives, so it is polynomial of degree at most
$s_t$ in time and $s_x$ in space; furthermore $A\whpsi=0$ implies that $\whpsi$
has vanishing projection on the polynomial space
$\IP^{s_x}(\wK_\wbx)\otimes\IP^{s_t}(\wI)$, so $\whpsi$ vanishes
identically in $\wK$ and $A$ is injective.

By the corollary on page 21 of \cite{Mazya} (see also the definitions on page 7), 
\begin{align*}
\N{\whpsi}_X
=\N{\whpsi}_{H^{s_t+1}(\wI;L^2(\wK_\wbx))}
+\N{\whpsi}_{L^2(\wI;H^{s_x+1}(\wK_\wbx))}
&\lesssim
\N{\whpsi}_{L^2(\wK)}+
\N{\partial_{\widehat t}^{s_t+1} \whpsi}_{L^2(\wK)}+
\N{D^{s_x+1}_\wbx \whpsi}_{L^2(\wK)^{s_x+2}}
\\&\le \N{A\whpsi}_Y+\N{T\whpsi}_Z,
\end{align*}
therefore, the Peetre--Tartar lemma gives
\begin{equation}\label{eq:vhat}
\N{\whpsi}_X 
=\N{\whpsi}_{H^{s_t+1}(\wI;L^2(\wK_\wbx))}
+\N{\whpsi}_{L^2(\wI;H^{s_x+1}(\wK_\wbx))}
\lesssim  \N{A\whpsi}_Y,\qquad\forall \whpsi\in X.
\end{equation}

When $s_t=p_t$ and $s_x=p_x$, then $\partial_{\widehat t}^{s_t+1} \wPi_\bp\whvarphi=0$ and $D^{s_x+1}_\wbx\wPi_\bp\whvarphi=0$.
Otherwise, if $s_t<p_t$, in general  $\partial_{\widehat t}^{s_t+1} \wPi_\bp\whvarphi\ne0$.
However, 
denoting $\widehat{\pi}_x^{s_x}:L^2(\widehat K_{\widehat\bx})\to\IP^{s_x}(\widehat K_{\widehat\bx})$ and 
$\widehat{\pi}_t^{s_t}:L^2(\widehat I)\to\IP^{s_t}(\widehat I)$ the $L^2$-orthogonal projections,
we split the projector 
$\wPi_\bp :=(I\otimes\widehat{\pi}_x^{s_x})\circ (\widehat{\pi}_t^{s_t}\otimes I)$,
and obtain
\begin{align*}
\N{\partial_{\widehat t}^{s_t+1} \wPi_\bp\whvarphi}_{L^2(\wK)}
&=
\N{\widehat\Pi_x^{s_x}\partial_{\widehat t}^{s_t+1} \widehat\Pi_t^{s_t}\whvarphi}_{L^2(\wK)} 
\le
\N{\partial_{\widehat t}^{s_t+1} \widehat\Pi_t^{s_t}\whvarphi}_{L^2(\wK)} \\
&\le
\N{\partial_{\widehat t}^{s_t+1}\whvarphi}_{L^2(\wK)} 
+\N{\partial_{\widehat t}^{s_t+1}(\whvarphi-\widehat\Pi_t^{s_t}\whvarphi)}_{L^2(\wK)} 
\lesssim 
\N{\partial_{\widehat t}^{s_t+1}\whvarphi}_{L^2(\wK)}
\end{align*}
by the Bramble--Hilbert lemma (we have used the $L^2$ continuity of the
$L^2$ projection in the first inequality).
Similarly, if $s_x<p_x$, then
\begin{align*}
\N{D_\wbx^{s_x+1} \wPi_\bp\whvarphi}_{L^2(\wK)}
&\lesssim 
\N{D_\wbx^{s_x+1} \whvarphi}_{L^2(\wK)}.
\end{align*}
In all cases, bound~\eqref{eq:vhat}
applied to $\whpsi=\whvarphi-\wPi_\bp\whvarphi$, 
together with $\wPi_\bp(\whvarphi-\wPi_\bp\whvarphi)=0$, 
gives
\begin{align}\label{eq:strongerbound}
\nonumber
&\N{\whvarphi-\wPi_\bp \whvarphi}_{H^{s_t+1}(\wI;L^2(\wK_\wbx))}
+\N{\whvarphi-\wPi_\bp \whvarphi}_{L^2(\wI;H^{s_x+1}(\wK_\wbx))}
=\N{\whvarphi-\wPi_\bp \whvarphi}_X
\lesssim
\N{A(\whvarphi-\wPi_\bp \whvarphi)}_Y
\\&\qquad
\lesssim
\N{\partial_{\widehat t}^{s_t+1} \whvarphi}_{ L^2(\wK)}
+\N{D^{s_x+1}_\wbx \whvarphi}_{ L^2(\wK)^{s_x+2}}
=\abs{\whvarphi}_{H^{s_t+1}(\wI; L^2(\wK_\wbx))}+
\abs{\whvarphi}_{L^2(\wI;H^{s_x+1}(\wK_\wbx))},
\end{align}
which admits \eqref{eq:hatBound} as a particular case.
\end{proof}

\begin{remark}
The proof of Proposition~\ref{prop:L2proj} actually provides a slightly stronger bound
under the same assumptions. 
Indeed, applying a scaling argument to bound~\eqref{eq:strongerbound} gives
\begin{align*}
\sum_{z_t=0}^{s_t} h_n^{z_t} \abs{\varphi-\Pi_\bp \varphi}_{H^{z_t}(I_n;L^2(K_\bx))}
&+\sum_{z_x=1}^{s_x} h_{K_\bx}^{z_x}\abs{\varphi-\Pi_\bp \varphi}_{L^2(I_n;H^{z_x}(K_\bx))}
\\&\lesssim
h_n^{s_t+1}\abs{\varphi}_{H^{s_t+1}(I_n;
L^2(K_\bx))}+h_{K_\bx}^{s_x+1}\abs{\varphi}_{L^2(I_n;H^{s_x+1}(K_\bx))}.
\end{align*}
\end{remark}

\subsubsection{The general case}\label{s:L2rough}

The next proposition is concerned with functions with
$H^{1,1}_\udelta$ regularity in space.

\begin{proposition}\label{prop:nearcorners}
Let $K=K_\bx\times I_n\in\calT_h^\angle$, with $h_{K_\bx}=\diam(K_{\bx})$ and $h_n=\abs{I_n}$. 
Denote by  $\Pi_{(0,p_t)}$, $p_t\in\IN_0$, the $L^2(K)$-orthogonal
projection: $L^2(K)\to\IP^0(K_\bx)\otimes\IP^{p_t}(I_n)$.

Let $\varphi\in H^{k_t+1}(I_n;L^2(K_{\bx}))\cap L^2(I_n;H^{1,1}_\udelta(K_{\bx}))$, $k_t\in\IN_0$. 
Then
\begin{align}\nonumber
\N{\varphi-\Pi_{(0,p_t)}\varphi}_{L^2(K_\bx)}+
&h_n\abs{\varphi-\Pi_{(0,p_t)}\varphi}_{H^1(I_n;L^2(K_\bx))}+
h_{K_\bx}^{1-\delta_K}\abs{\varphi-\Pi_{(0,p_t)}\varphi}_{L^2(I_n;H^{1,1}_\udelta(K_\bx))}
\\
&\lesssim
h_n^{s_t+1}\abs{\varphi}_{H^{s_t+1}(I_n;
L^2(K_\bx))}+h_{K_\bx}^{1-\delta_K}\abs{\varphi}_{L^2(I_n;H^{1,1}_\udelta(K_\bx))},
\label{eq:L2projCorner}
\end{align}
where $s_t:=\min\{k_t,p_t\}$, and the hidden constants are independent of $h_{K_\bx}$ and $h_n$, but depend on the polynomial degree $p_t$ and on the shape of $K_\bx$.
\end{proposition}
\begin{proof}
It suffices to repeat the reasoning in the proof of Proposition~\ref{prop:L2proj},
using the compactness of the embedding 
$H^{1,1}_\udelta(K_\bx)\subset L^2(K_\bx)$ recalled in~\eqref{eq:smooth3} 
(see \cite[Proposition~A.2.5]{TWDiss}).
\end{proof}

\subsection{Galerkin error estimates}
\label{s:GalErrEst}
As in \S\ref{s:AbstractAnalysis} and \S\ref{s:ProjEstimates}, 
we start by considering the case where the problem data $f,v_0,\bsigma_0,\gD,\gN$ are such that the solution $(v,\bsigma)$ 
does not present singularities due to the domain corners, then we
consider the general case.

\subsubsection{The smooth solution case}\label{s:GALsmooth}

\begin{proposition}\label{prop:SmoothError}
Let $\vs$ and $\vsh\in\bV_\bp\Th$ be the solutions of \eqref{eq:IBVP} and \eqref{eq:DG}, respectively.
Assume that
\begin{align*}
&v|_K\in H^{k_{t,K}^v+1}\big(I_{n'};{H^2(K_\bx)}\big)\cap L^2\big(I_{n'};H^{k_{x,K}^v+1}(K_\bx)\big)
\qquad\forall K=K_\bx\times I_{n'}\in \calT_h,
\\
&\bsigma|_K\in H^{k_{t,K}^\bsigma+1}\big(I_{n'};L^2(K_\bx)^2\big)\cap L^2\big(I_{n'};H^{k_{x,K}^\bsigma+1}(K_\bx)^2\big)
\qquad\forall K=K_\bx \times I_{n'}\in \calT_h,
\end{align*}
and $\bV_\bp\Th$ as in \eqref{eq:Vp}, with
$$
k_{t,K}^v,k_{x,K}^v,k_{t,K}^\bsigma,k_{x,K}^\bsigma,
p_{t,K}^v,p_{x,K}^v,p_{t,K}^\bsigma,p_{x,K}^\bsigma
\in \IN_0,\qquad |p_{x,K}^\bsigma-p_{x,K}^v|\le1, \qquad p_{t,K}^\bsigma=p_{t,K}^v.
$$
Define
$$
s_{t,K}^v:=\min\{k_{t,K}^v,p_{t,K}^v\},\;
s_{x,K}^v:=\min\{k_{x,K}^v,p_{x,K}^v\},\;
s_{t,K}^\bsigma:=\min\{k_{t,K}^\bsigma,p_{t,K}^\bsigma\},\;
s_{x,K}^\bsigma:=\min\{k_{x,K}^\bsigma,p_{x,K}^\bsigma\}.
$$
Then,
\begin{align*}
&\frac12\N{c^{-1}(v-v\hp)}_{L^2(\Omega\times\{t_n\})}
                 +\frac12\N{\bsigma-\bsigma\hp}_{L^2(\Omega\times\{t_n\})^2}
                 \le\abs{\vs-\vsh}\DGQn
\\
&\lesssim\hspace{-2mm}
\sum_{\substack{K=K_\bx\times I_{n'}\\\in\calT_h(Q_n)}}
\bigg[
(h_{n'}^{-1/2}c_K^{-1}+h_{K_\bx}^{-1/2}\tta_K)
\Big(
h_{n'}^{s_{t,K}^v+1}{h_{K_\bx}}\abs{v}_{H^{s_{t,K}^v+1}(I_{n'}; {H^2(K_\bx)})}
+h_{K_\bx}^{s_{x,K}^v+1}
\abs{v}_{L^2(I_{n'};H^{s_{x,K}^v+1}(K_\bx))}
\Big)\\
&\hspace{20mm}+(h_{n'}^{-1/2}+h_{K_\bx}^{-1/2}\ttb_K)
\Big(
h_{n'}^{s_{t,K}^\bsigma+1}\abs{\bsigma}_{H^{s_{t,K}^\bsigma+1}(I_{n'};L^2(K_\bx)^2)}
+h_{K_\bx}^{s_{x,K}^\bsigma+1}\abs{\bsigma}_{L^2(I_{n'};H^{s_{x,K}^\bsigma+1}(K_\bx)^2)}
\Big)\bigg].
\end{align*}
The hidden constant depends on the spatial and temporal polynomial degrees, 
the regularity exponents and the spatial element shapes.
\end{proposition}

\begin{proof}
The DG scheme \eqref{eq:DG} is well-posed because the conditions on the polynomial degrees defining $\bV_\bp\Th$ ensure that condition \eqref{eq:dVpAssumpt} is satisfied, namely that the (piecewise) first-order wave operator maps the discrete space into itself. 
We use the quasi-optimality result of Proposition \ref{prop:QOsmooth}, 
the bound \eqref{eq:DGpBound2} on the $\abs{\cdot}\DGQnp$ seminorm
of the $L^2$ projection error $\wt|_K:=\vs|_K-(\Pi_{\bp^v_K} v|_K,\Pi_{\bp^\bsigma_K}\bsigma|_K)$, 
$\bp^v_K=(p^v_{x,K},p^v_{t,K})$ and
$\bp^\bsigma_K=(p^\bsigma_{x,K},p^\bsigma_{t,K})$
(neglecting the summation over elements in $\calT_h^\angle(Q_n)$, and replacing $\calT_h^\odot(Q_n)$ by $\calT_h(Q_n)$),
and the projection estimate \eqref{eq:L2proj} applied to each
component of $\wt$ 
\begin{align*}
&\frac12\N{c^{-1}(v-v\hp)}_{L^2(\Omega\times\{t_n\})}
+\frac12\N{\bsigma-\bsigma\hp}_{L^2(\Omega\times\{t_n\})^2}
\le\abs{\vs-\vsh}\DGQn
\overset{\eqref{eq:AbstractErrorS}}\lesssim\abs{\wt}\DGQnp
\\
&\overset{\eqref{eq:DGpBound2}}\lesssim
\bigg(\sum_{K=K_\bx\times I_{n'}\in\calT_h(Q_n)}\bigg[
h_{n'}^{-1}\Big(\N{c_K^{-1}(v-\Pi_{\bp_K^v} v)}_{L^2(K)}^2
+\N{\bsigma-\Pi_{\bp_K^\bsigma}\bsigma}_{L^2(K)^2}^2\Big)\\
&\hspace{8mm}
+h_{n'}\Big(\abs{c_K^{-1}(v-\Pi_{\bp_K^v}v)}_{H^1(I_{n'};L^2(K_\bx))}^2
+\abs{\bsigma-\Pi_{\bp_K^\bsigma}\bsigma}_{H^1(I_{n'};L^2(K_\bx)^2)}^2\Big)
\\
&\hspace{8mm}
+h_{K_\bx}^{-1}\N{\tta_K (v-\Pi_{\bp_K^v} v)}_{L^2(K)}^2
+h_{K_\bx}\abs{\tta_K (v-\Pi_{\bp_K^v} v)}_{L^2(I_{n'};H^1(K_\bx))}^2
+h_{K_\bx}^{-1}
\N{\ttb_K(\bsigma-\Pi_{\bp_K^\bsigma}\bsigma)}_{L^2(K)^2}^2\bigg]
\\
&\quad+\sum_{K=K_\bx\times I_{n'}\in\calT_h(Q_n)}
h_{K_\bx}\abs{\ttb_K(\bsigma-\Pi_{\bp_K^\bsigma}\bsigma)}_{L^2(I_{n'};H^1(K_\bx)^2)}^2
\bigg)^{1/2}
\\
&\overset{\eqref{eq:L2proj}}\lesssim
\bigg(\sum_{K=K_\bx\times I_{n'}\in\calT_h(Q_n)}\hspace{-8mm}
(h_{n'}^{-1}c_K^{-2}+h_{K_\bx}^{-1}\tta_K^2)
\Big(
h_{n'}^{s_{t,K}^v+1}\abs{v}_{H^{s_{t,K}^v+1}(I_{n'};L^2(K_\bx))}
+h_{K_\bx}^{s_{x,K}^v+1}\abs{v}_{L^2(I_{n'};H^{s_{x,K}^v+1}(K_\bx))}
\Big)^2
\\
&\hspace{5mm}+\sum_{K=K_\bx\times I_{n'}\in\calT_h(Q_n)}\hspace{-8mm}
(h_{n'}^{-1}+h_{K_\bx}^{-1}\ttb_K^2)
\Big(
h_{n'}^{s_{t,K}^\bsigma+1}\abs{\bsigma}_{H^{s_{t,K}^\bsigma+1}(I_{n'};L^2(K_\bx)^2)}
+h_{K_\bx}^{s_{x,K}^\bsigma+1}\abs{\bsigma}_{L^2(I_{n'};H^{s_{x,K}^\bsigma+1}(K_\bx)^2)}
\Big)^2 \bigg)^{1/2}\\
&\lesssim
\bigg(\sum_{K=K_\bx\times I_{n'}\in\calT_h(Q_n)}\hspace{-8mm}
(h_{n'}^{-1}c_K^{-2}+h_{K_\bx}^{-1}\tta_K^2)
\Big(
h_{n'}^{s_{t,K}^v+1}h_{K_\bx}\abs{v}_{H^{s_{t,K}^v+1}(I_{n'};L^\infty(K_\bx))}
+h_{K_\bx}^{s_{x,K}^v+1}\abs{v}_{L^2(I_{n'};H^{s_{x,K}^v+1}(K_\bx))}
\Big)^2
\\
&\hspace{5mm}+\sum_{K=K_\bx\times I_{n'}\in\calT_h(Q_n)}\hspace{-8mm}
(h_{n'}^{-1}+h_{K_\bx}^{-1}\ttb_K^2)
\Big(
h_{n'}^{s_{t,K}^\bsigma+1}\abs{\bsigma}_{H^{s_{t,K}^\bsigma+1}(I_{n'};L^2(K_\bx)^2)}
+h_{K_\bx}^{s_{x,K}^\bsigma+1}\abs{\bsigma}_{L^2(I_{n'};H^{s_{x,K}^\bsigma+1}(K_\bx)^2)}
\Big)^2 \bigg)^{1/2},
\end{align*}
where in the last inequality we have used the inclusion $H^2(K_\bx) \subset C^0(\overline{K_\bx})$, and
 $\abs{v}_{H^{s_{t,K}^v+1}(I_{n'};L^2(K_\bx))}\lesssim h_{K_\bx} \abs{v}_{H^{s_{t,K}^v+1}(I_{n'};L^\infty(K_\bx))}$. 
The assertion follows by using again the inclusion $H^2(K_\bx)\subset C^0(\overline{K_\bx})$.
\end{proof}

Proposition~\ref{prop:SmoothError} provides the same error
  convergence rates in both the $L^2\OO$ norm at a discrete time $t_n$ and in the $\abs{\cdot}\DGQn$ seminorm.
However, one would expect higher rates in the former norm: when the mesh is refined uniformly, in the computation of the $\abs{\cdot}\DGQn$ seminorm the error is integrated on larger and larger domains (the mesh skeletons).
This suggests that the convergence rates with respect to the mesh size for the $L^2(\Omega\times\{t_n\})$ norm of the error are suboptimal.
The same reasoning applies to the bounds in Propositions~\ref{prop:errorBound} and \ref{prop:refinedMeshBounds}.

Without using the information on the elementwise regularity of $\vs$, but considering only the global regularity, 
the statement of Proposition~\ref{prop:SmoothError} simplifies as in the next corollary.

\begin{cor}\label{cor:smooth}
Assume that, for the IBVP solution $(v,\bsigma)$ we have
\[
(v,\bsigma)\in C^{k_t-1}\big([0,T]; H^{k_x+1}\OO\big)\times 
C^{k_t}\big([0,T]; H^{k_x}\OO^2\big),
\]
for some $k_x,k_t\in\IN$ and $k_t\ge2$.
Assume also that the space $\bVp\Th$ and the local polynomial degrees
are defined as in 
Proposition~\ref{prop:SmoothError}.
Then the assertion of Proposition~\ref{prop:SmoothError} holds true with 
\begin{equation*}
k_{t,K}^v=k_t-2,\quad 
k_{x,K}^v=k_x,\quad
k_{t,K}^\bsigma=k_t-1,\quad
k_{x,K}^\bsigma=k_x-1\qquad
\forall K\in \calT_h.
\end{equation*}
Assuming in particular the velocity $c$ to be constant in $Q$, 
and that the mesh satisfies the quasi-uniformity condition
\begin{equation*}
h:=\max\Big\{\max_{K\in\calT_h(Q_n)}h_{K_\bx},\; \max_{n'=1,\ldots,n}ch_{n'}\Big\}
\le \rho
\min\Big\{\min_{K\in\calT_h(Q_n)}h_{K_\bx},\; \min_{n'=1,\ldots,n}ch_{n'}\Big\}
\end{equation*}
for some $\rho\ge1$,
that the polynomial degrees are uniform in space and time, i.e. 
$p_{t,K}^v=p_{x,K}^v=p_{t,K}^\bsigma=p_{x,K}^\bsigma=p$, 
with the numerical flux parameters chosen as $\alpha^{-1}=\beta=c$,
then
\begin{align*}
\frac12\N{c^{-1}(v-v\hp)}_{L^2(\Omega\times\{t_n\})}
                 +\frac12\N{\bsigma-\bsigma\hp}_{L^2(\Omega\times\{t_n\})^2}
                 &\le\abs{\vs-\vsh}\DGQn\\
&\lesssim C\vs\, h^{\min\{k_t-\frac12,k_x-\frac12,p+\frac12\}},
\end{align*}
where $C\vs$ depends on the mesh $\calT_h$ 
only through the quasi-uniformity parameter $\rho$ and the shape of
the elements.
\end{cor}

\subsubsection{The general case}\label{s:GALrough}
\begin{proposition}\label{prop:errorBound}
Let $\vs$ and $\vsh\in\bV_\bp\Th$ be the solutions of \eqref{eq:IBVP} and \eqref{eq:DG}, respectively.
Assume that
\begin{align*}
&v|_K \in H^{k_{t,K}^v+1}\big(I_{n'};H^{2,2}_\udelta(K_\bx)\big) \cap L^2\big(I_{n'};H^{k_{x,K}^v+1}(K_\bx)\big)
&&\forall K=K_\bx \times I_{n'} \in \calT_h,
\\
&\bsigma|_K\in H^{k_{t,K}^\bsigma+1}\big(I_{n'};L^2(K_\bx)^2\big)\cap L^2\big(I_{n'};H^{k_{x,K}^\bsigma+1}(K_\bx)^2\big)
&&\forall K=K_\bx\times I_{n'}\in \calT_h^\odot,
\\
&\bsigma|_K\in H^{k_{t,K}^\bsigma+1}\big(I_{n'};L^2(K_\bx)^2\big)\cap L^2\big(I_{n'};H^{1,1}_{\delta_K}(K_\bx)^2\big)
&&\forall K=K_\bx\times I_{n'}\in \calT_h^\angle,
\end{align*}
with
$\bV_\bp\Th$ as in \eqref{eq:Vp} and the indices 
$k^a_{b,K}$, $p^a_{b,K}$, $s^a_{b,K}$ (for $a\in\{v,\bsigma\}$, $b\in\{t,x\}$, $K\in\calT_h(Q_n)$)
as in Proposition~\ref{prop:SmoothError}. 
Then
\begin{align*}
&\frac12\N{c^{-1}(v-v\hp)}_{L^2(\Omega\times\{t_n\})}
+\frac12\N{\bsigma-\bsigma\hp}_{L^2(\Omega\times\{t_n\})^2}
\le\abs{\vs-\vsh}\DGQn
\\
&\lesssim
\sum_{\substack{K=K_\bx\times I_{n'}\\\in\calT_h(Q_n)}}
(h_{n'}^{-1/2}c_K^{-1}+h_{K_\bx}^{-1/2}\tta_K)
\Big(
h_{n'}^{s_{t,K}^v+1}h_{K_\bx}\abs{v}_{H^{s_{t,K}^v+1}(I_{n'}; H^{2,2}_\udelta(K_\bx))}
+h_{K_\bx}^{s_{x,K}^v+1}\abs{v}_{L^2(I_{n'};H^{s_{x,K}^v+1}(K_\bx))}
\Big)\\
&+\sum_{K=K_\bx\times I_{n'}\in\calT_h^\odot(Q_n)}
(h_{n'}^{-1/2}+h_{K_\bx}^{-1/2}\ttb_K)
\Big(
h_{n'}^{s_{t,K}^\bsigma+1}\abs{\bsigma}_{H^{s_{t,K}^\bsigma+1}(I_{n'};L^2(K_\bx)^2)}
+h_{K_\bx}^{s_{x,K}^\bsigma+1}\abs{\bsigma}_{L^2(I_{n'};H^{s_{x,K}^\bsigma+1}(K_\bx)^2)}
\Big)\\
&+\sum_{K=K_\bx\times I_{n'}\in\calT_h^\angle(Q_n)}
(h_{n'}^{-1/2}+h_{K_\bx}^{-1/2}\ttb_K)
\Big(
h_{n'}^{s_{t,K}^\bsigma+1}\abs{\bsigma}_{H^{s_{t,K}^\bsigma+1}(I_{n'};L^2(K_\bx)^2)}
+h_{K_\bx}^{1-\delta_K}\abs{\bsigma}_{L^2(I_{n'};H^{1,1}_\udelta(K_\bx)^2)}
\Big).
\end{align*}
As in Proposition~\ref{prop:SmoothError}, 
the hidden constant depends on the spatial and temporal polynomial degrees, 
the solution regularity indices and on the shape-regularity constant of the spatial elements.
\end{proposition}
\begin{proof}
We proceed as in the proof of
  Proposition~\ref{prop:SmoothError}.
By using the quasi-optimality result of Proposition~\ref{prop:QO}, the bound
\eqref{eq:DGpBound2}, and the projection estimates~\eqref{eq:L2proj}
and~\eqref{eq:L2projCorner}, we get 
\begin{align*}
&\frac12\N{c^{-1}(v-v\hp)}_{L^2(\Omega\times\{t_n\})}
+\frac12\N{\bsigma-\bsigma\hp}_{L^2(\Omega\times\{t_n\})^2}
\le\abs{\vs-\vsh}\DGQn
\overset{\eqref{eq:AbstractErrorB}}
\lesssim\abs{\wt}\DGQnp
\\
&\overset{\eqref{eq:DGpBound2}}\lesssim
\bigg(\sum_{K=K_\bx\times I_{n'}\in\calT_h(Q_n)}\bigg[
h_{n'}^{-1}\Big(\N{c_K^{-1}(v-\Pi_{\bp_K^v} v)}_{L^2(K)}^2
+\N{\bsigma-\Pi_{\bp_K^\bsigma}\bsigma}_{L^2(K)^2}^2\Big)\\
&\hspace{8mm}
+h_{n'}\Big(\abs{c_K^{-1}(v-\Pi_{\bp_K^v}v)}_{H^1(I_{n'};L^2(K_\bx))}^2
+\abs{\bsigma-\Pi_{\bp_K^\bsigma}\bsigma}_{H^1(I_{n'};L^2(K_\bx)^2)}^2\Big)
\\
&\hspace{8mm}
+h_{K_\bx}^{-1}\N{\tta_K (v-\Pi_{\bp_K^v} v)}_{L^2(K)}^2
+h_{K_\bx}\abs{\tta_K (v-\Pi_{\bp_K^v} v)}_{L^2(I_{n'};H^1(K_\bx))}^2
+h_{K_\bx}^{-1}
\N{\ttb_K(\bsigma-\Pi_{\bp_K^\bsigma}\bsigma)}_{L^2(K)^2}^2\bigg]
\\
&\quad+\sum_{K=K_\bx\times I_{n'}\in\calT_h^\odot(Q_n)}
h_{K_\bx}\abs{\ttb_K(\bsigma-\Pi_{\bp_K^\bsigma}\bsigma)}_{L^2(I_{n'};H^1(K_\bx)^2)}^2
\\
&\quad+\sum_{K=K_\bx\times I_{n'}\in\calT_h^\angle(Q_n)}
h_{K_\bx}^{1-2\delta_K}\abs{\ttb_K(\bsigma-\Pi_{(0,p_{t,K}^\bsigma)}\bsigma)}_{L^2(I_{n'};H^{1,1}_\udelta(K_\bx)^2)}^2
\bigg)^{1/2}
\\
&\overset{\eqref{eq:L2proj},\eqref{eq:L2projCorner}}\lesssim
\bigg(\sum_{K=K_\bx\times I_{n'}\in\calT_h(Q_n)}\hspace{-8mm}
(h_{n'}^{-1}c_K^{-2}+h_{K_\bx}^{-1}\tta_K^2)
\Big(
h_{n'}^{s_{t,K}^v+1}\abs{v}_{H^{s_{t,K}^v+1}(I_{n'};L^2(K_\bx))}
+h_{K_\bx}^{s_{x,K}^v+1}\abs{v}_{L^2(I_{n'};H^{s_{x,K}^v+1}(K_\bx))}
\Big)^2
\\
&\hspace{4mm}+\sum_{K=K_\bx\times I_{n'}\in\calT_h^\odot(Q_n)}\hspace{-8mm}
(h_{n'}^{-1}+h_{K_\bx}^{-1}\ttb_K^2)
\Big(
h_{n'}^{s_{t,K}^\bsigma+1}\abs{\bsigma}_{H^{s_{t,K}^\bsigma+1}(I_{n'};L^2(K_\bx)^2)}
+h_{K_\bx}^{s_{x,K}^\bsigma+1}\abs{\bsigma}_{L^2(I_{n'};H^{s_{x,K}^\bsigma+1}(K_\bx)^2)}
\Big)^2\\
&\hspace{4mm}+\sum_{K=K_\bx\times I_{n'}\in\calT_h^\angle(Q_n)}
(h_{n'}^{-1}+h_{K_\bx}^{-1}\ttb_K^2)
\Big(
h_{n'}^{s_{t,K}^\bsigma+1}\abs{\bsigma}_{H^{s_{t,K}^\bsigma+1}(I_{n'};L^2(K_\bx)^2)}
+h_{K_\bx}^{1-\delta_K}\abs{\bsigma}_{L^2(I_{n'};H^{1,1}_\udelta(K_\bx)^2)}
\Big)^2  \bigg)^{1/2}\\
 &\lesssim
\bigg(\sum_{K=K_\bx\times I_{n'}\in\calT_h(Q_n)}\hspace{-8mm}
(h_{n'}^{-1}c_K^{-2}+h_{K_\bx}^{-1}\tta_K^2)
\Big(
h_{n'}^{s_{t,K}^v+1}h_{K_\bx}\abs{v}_{H^{s_{t,K}^v+1}(I_{n'};L^\infty(K_\bx))}
+h_{K_\bx}^{s_{x,K}^v+1}\abs{v}_{L^2(I_{n'};H^{s_{x,K}^v+1}(K_\bx))}
\Big)^2
\\ 
 &\hspace{4mm}+\sum_{K=K_\bx\times I_{n'}\in\calT_h^\odot(Q_n)}\hspace{-8mm}
(h_{n'}^{-1}+h_{K_\bx}^{-1}\ttb_K^2)
\Big(
h_{n'}^{s_{t,K}^\bsigma+1}\abs{\bsigma}_{H^{s_{t,K}^\bsigma+1}(I_{n'};L^2(K_\bx)^2)}
+h_{K_\bx}^{s_{x,K}^\bsigma+1}\abs{\bsigma}_{L^2(I_{n'};H^{s_{x,K}^\bsigma+1}(K_\bx)^2)}
\Big)^2\\
&\hspace{4mm}+\sum_{K=K_\bx\times I_{n'}\in\calT_h^\angle(Q_n)}
(h_{n'}^{-1}+h_{K_\bx}^{-1}\ttb_K^2)
\Big(
h_{n'}^{s_{t,K}^\bsigma+1}\abs{\bsigma}_{H^{s_{t,K}^\bsigma+1}(I_{n'};L^2(K_\bx)^2)}
+h_{K_\bx}^{1-\delta_K}\abs{\bsigma}_{L^2(I_{n'};H^{1,1}_\udelta(K_\bx)^2)}
\Big)^2
\bigg)^{1/2},
\end{align*}
where in the last inequality we have used the inclusion $H^{2,2}_\udelta(K_\bx)\subset C^0(\overline{K_\bx})$
(see~\cite[Formula~(2.2)]{BKPi79}, recalled in~\cite[Remark~1.2.2~b)]{TWDiss}), and
$\abs{v}_{H^{s_{t,K}^v+1}(I_{n'};L^2(K_\bx))}\lesssim h_{K_\bx} \abs{v}_{H^{s_{t,K}^v+1}(I_{n'};L^\infty(K_\bx))}$. 
The assertion follows by using again the inclusion $H^{2,2}_\udelta(K_\bx)\subset C^0(\overline{K_\bx})$.
\end{proof}

When the solution exhibits point singularities, either at corners of $\Omega$ or at multimaterial interfaces in the interior of $\Omega$, the convergence rate of discretizations on quasi-uniform families of partitions of $\Omega$ is well known to degrade, due to the limited solution regularity in scales of standard Sobolev spaces.
Local mesh refinement, with the local refinement taking place {in space only}, can restore the largest possible convergence rates.
We describe suitable mesh families in the next section, and derive the corresponding orders of convergence in Proposition~\ref{prop:rough_errorBounds_lrefMeshes}.
We remark that, if the numerical flux coefficients $\alpha^{-1}$ and $\beta$ 
are chosen to be proportional to the ratio between global and local space
mesh size, the estimate in Proposition~\ref{prop:errorBound} guarantees that, under local spatial mesh refinement, the solution remains bounded. 
On the other hand, numerical results in \S\ref{ss:test_1.2} below seem to suggest that this condition is actually not necessary.

\section{Error bounds on locally refined meshes}\label{sec:RefMesXDom}

For each corner $\bc \in \calS$, define the distance
\begin{equation}\label{eq:Rc}
 R_{\bc} := \frac{1}{2} \min_{\bc' \in \calS \backslash \{\bc\}} |\bc - \bc'|
\end{equation}
and let $\dist(K_{\bx}, \bc)$ be the spatial distance 
of the element $K_{\bx}$ from the corner $\bc$.
Assume $p_{x,K}^{\bsigma} = p_{x}^{\bsigma} \in \IN_{0}$ 
for all elements $K$ of any given mesh.

We use the algorithm proposed by Gaspoz and Morin \cite{GaspozMorin} 
to generate the meshes with appropriate grading towards the corners of a polygon:

\begin{defin}[Corner-refined, regular simplicial triangulation of $\Omega$]
In a polygon $\Omega$ with ${\rm diam}(\Omega) \leq 1$, 
a corner-refined, regular simplicial triangulation $\calTspace$
of meshwidth $h_{\bx}\in (0,1]$ 
for spatial polynomial degree $p_{\bx}\geq 1$ and 
corner weights $\{ \delta_{\bc}: \bc \in \calS\}\subset 
[0,1)^{M}$ as in \S\ref{s:Regularity} 
is a regular partition obtained by the following procedure.
Let $\JJ \in \IN_{0}$ be such that for each $\bc\in\calS$
\begin{equation}
2^{- \frac{(1-\delta_{\bc})(\JJ+1)}{p_{x}^\bsigma+1}} 
\leq h_{\bx} 
\lesssim 2^{- \frac{(1-\delta_{\bc})\JJ}{p_{x}^\bsigma+1}}.
\label{eq:numDiscsRefined}
\end{equation}
with the hidden constant being independent of $h_{\bx}$.
Corner-refined, regular simplicial triangulations of $\Omega$ are constructed by regular refinement of the regular, 
simplicial initial triangulation $\calTspaceinit$ 
using Algorithms~\ref{alg:bisection_meshGen} and \ref{alg:bisect_marked}, 
and newest vertex bisection \cite[Section 3.3.1]{MullerPhD}.

\begin{algorithm}[H]
\SetAlgoLined
\KwData{$\calTspaceinit$, $h_{\bx}>0$, $\bdelta \in [0,1)^M$, 
        $p_{x}^{\bsigma} \in \IN_{0}$, $\{R_{\bc}\}_{\bc \in \calS}$,
        $bool\_conforming$}
\KwResult{Mesh $\calTspace$ with local refinement towards corners}
 // Step 1: refine $\calTspaceinit$ until $h_{K_{\bx}} < h_{\bx}\ \forall\ K_{\bx} \in \calTspaceinit$\\
 $\calTspace$ := $\calTspaceinit$\\
 $\cM := \{ K_{\bx} \in \calTspace\ |\ h_{K_{\bx}} > h_{\bx} \}$\\
 \While{$\cM \neq \emptyset $}{
  $\calTspace$ := bisect\_marked($\calTspace$, $\cM$, $bool\_conforming$)\\
  $\cM := \{ K_{\bx} \in \calTspace\ |\ h_{K_{\bx}} > h_{\bx} \}$\\
 }
 // Step 2: local refinement in nested discs\\
 \For {$\bc \in \calS$} {
  $\JJ := \lceil {-\log_2(h_{\bx})\frac{p_{x}^{\bsigma}+1}{1-\delta_{\bc}} - 1} \rceil$\\
  \For {$j=0,1,\ldots,2\JJ+1$} {
    $\cM := \{ K_{\bx} \in \calTspace\ |\ \dist(K_{\bx}, \bc) \leq 2^{-j/2}R_{\bc}\ 
    \land\ h_{K_{\bx}} > h_{\bx}
    2^{-j(p_{x}^{\bsigma}+\delta_{\bc})/(2(p_{x}^{\bsigma}+1))}
    \}$\\
    $\calTspace$ := bisect\_marked($\calTspace$, $\cM$, $bool\_conforming$)\\
  }
 }
\caption{bisection\_tree\_refinement($\calTspaceinit$, $h_{\bx}$, 
				     $\bdelta$, $p_{x}^{\bsigma}$,
				     $R_{\bc}$, $bool\_conforming$)}
\label{alg:bisection_meshGen}
\end{algorithm}

\

\begin{algorithm}[H]
\SetAlgoLined
\KwData{$\calTspaceinit$, $\cM \subset \calTspaceinit$, $bool\_conforming$}
\KwResult{$\calTspacemarked_\JJ$}
 \uIf{\text{bool\_conforming}}{
 $j := 0$\\
 $\cM_{0} := \cM$\\
 \While{$\cM_{j} \neq \emptyset $}{
  \For{$K_{\bx} \in \cM_{j}$} {
   $\calTspacemarked_{j+1}$ := newest\_vertex\_bisection($K_{\bx}, \calTspacemarked_{j}$)
  }
  $\cM_{j+1} := \{ K_{\bx} \in \calTspacemarked_{j+1}\
                 |\ K_{\bx} \text{ has a hanging node} \}$\\
  $j := j+1$
 }
 Set $\JJ := j$
 }
 \Else {
  \For{$K_{\bx} \in \cM$} {
   $\calTspacemarked_{J}$ := newest\_vertex\_bisection($K_{\bx}, \calTspaceinit$)
  }
 }
\caption{bisect\_marked($\calTspaceinit, \cM$, $bool\_conforming$)}
\label{alg:bisect_marked}
\end{algorithm}
\end{defin}

\begin{remark}\label{rmk:Refine}
Algorithms~\ref{alg:bisection_meshGen} and~\ref{alg:bisect_marked} construct 
a spatial mesh of meshwidth at most $h_\bx$ obtained as bisection tree refinement of a regular initial partition $\calT_0^\bx$ of $\Omega$.
For $j\le 2\JJ+1$ and 
$\JJ= \lceil {-\log_2(h_{\bx})\frac{p_{x}^{\bsigma}+1}{1-\delta_{\bc}} - 1} \rceil$, 
the elements at distance not exceeding $2^{-j/2}R_\bc$ (cf. \eqref{eq:Rc}) from the closest vertex $\bc$ have diameter at most $h_{\bx} 2^{-\frac{j(p_{x}^{\bsigma}+\delta_{\bc})}{2(p_{x}^{\bsigma}+1)}}$.
Note that at fixed target meshwidth $h_{\bx}$, $J\uparrow\infty$ as $\delta_{\bc}\uparrow 1$ (cf. Remark \ref{rmk:deltai>0}).
The flag ``{\it bool\_conforming}'' in Algorithm 1 allows to choose between meshes with or without hanging nodes.
\end{remark}

If one knows a priori that, in a neighbourhood of a corner
$\bc_i\in\calS$, the IBVP solution $\vs$ has regularity better than
that associated to the exponent $\delta_i$, then one can choose a
smaller value of $\delta_i$ in~\eqref{eq:numDiscsRefined} and 
in the choice of $J$ in Algorithm 1. 
This does not alter the asymptotic complexity of the algorithm, but allows to use fewer local refinement levels.

\begin{proposition}\label{prop:refinedMeshBounds}
 Let $p_{x}^{\bsigma} \in \IN$, $h_{\bx}>0$, $\bdelta \in [0,1)^M$
 and $\calTspaceinit$ be an initial regular triangulation of the domain $\Omega$.
 For the mesh $\calTspace := \text{\emph{bisection\_tree\_refinement}}
 (\calTspaceinit, h_{\bx}, \bdelta, p_x^{\bsigma}, R_{\bc}, \cdot)$, 
 there holds the following bound:
 \begin{equation}
  \#\calTspace - \#\calTspaceinit \leq C h_{\bx}^{-2},
 \end{equation}
 where the constant $C$ is independent of $h_{\bx}$
 but depends on $\delta$. Furthermore, the constant depends on $\calTspaceinit$ and 
 increases unboundedly if $\delta_{\bc} \to 1$ (see Remark \ref{rmk:deltai>0})
 or if $p_x^{\bsigma} \to \infty$.
\end{proposition}
\begin{proof}
This complexity estimate has been proved in \cite[Lemma 4.9]{GaspozMorin} for conforming meshes. 
The estimate for non-conforming meshes follows as these have, asymptotically for $h_{\bx} \rightarrow 0$, the same number of elements.
\end{proof}

For solutions with point singularities and for locally refined triangulations of $\Omega$ spatial meshes constructed with Algorithms~\ref{alg:bisection_meshGen} and~\ref{alg:bisect_marked}, we now prove error bounds with the same dependence on the meshwidth $h$ as the bounds proved in Corollary~\ref{cor:smooth} for smooth solutions.

\begin{proposition}\label{prop:rough_errorBounds_lrefMeshes}
Assume that the IBVP solution $(v,\bsigma)$ admits the regularity
in \eqref{eq:vsRegularity2} for some 
smoothness orders $k_x,k_t\in\IN$, $k_t\ge2$ and $\bdelta\in[0,1)^M$.
Furthermore, assume that the space $\bVp\Th$ and 
the local polynomial degrees are defined as in Proposition~\ref{prop:SmoothError}.
Given a spatial meshwidth bound $h_{\bx}>0$ and an initial regular triangulation 
$\calTspaceinit$ of the domain $\Omega$,
the locally refined mesh $\calTspace$ is generated using Algorithm~\ref{alg:bisection_meshGen}. 

In particular, assume that the velocity $c$ is constant in $Q$,
the polynomial degrees are 
$p_{t,K}^v=p_{x,K}^v=p_{t,K}^\bsigma=p_{x,K}^\bsigma=p$,
the numerical flux parameters are chosen as
\[
\alpha|_F = c^{-1} \frac{h_{\bx}}{h_{F_\bx}}\quad\forall F\in\Ftime\cap\FD,
\qquad\quad
\beta|_F = c \frac{h_{F_\bx}}{h_{\bx}}\quad\forall F\in\Ftime\cap\FN
\]
We assume that the space--time mesh satisfies the condition
\begin{equation}\label{eq:xtrho}
h:=\max\Big\{h_{\bx},\; \max_{n'=1,\ldots,n}ch_{n'}\Big\}
\le \rho
\min\Big\{h_{\bx},\; \min_{n'=1,\ldots,n}ch_{n'}\Big\},
\end{equation}
for each discrete time $t_n$ and some $\rho > 1$.
Define $h_t:=\max_{n'=1,\ldots,n}h_{n'}$.
Then 
\begin{align}\nonumber
\frac12\N{c^{-1}(v-v\hp)}_{L^2(\Omega\times\{t_n\})}
                 +\frac12\N{\bsigma-\bsigma\hp}_{L^2(\Omega\times\{t_n\})^2}
                 &\le\abs{\vs-\vsh}\DGQn\\
&\le C(c,v, \bsigma)~h^{\min\{k_t-{\frac12},k_x-\frac12,p+\frac12\}},
\end{align}
where $C(c,v, \bsigma)$ depends on the space--time mesh
$\calT_h$ only through the shape-regularity constant of its spatial elements.
\end{proposition}
\begin{proof}
We adapt the proof of Proposition 3.3.9 in \cite{MullerPhD} to our setting.
For $K_{\bx}\in\calTspace$ and $\bc\in\calS$, define
\begin{align*}
 r_{\bc}(\bx) &:= \abs{\bx - \bc},\quad r_{K_{\bx}; \bc} := \inf_{\bx \in K_{\bx}} r_{\bc}(\bx).
\end{align*}
Without loss of generality, assume $R_{\bc} = 1$ (recall \eqref{eq:Rc}). 
Given $\bc\in\calS$, partition the mesh $\calTspace$ into sets
\begin{align*}
 D_{j;\bc} &:= 
 \left\{ 
 K_{\bx}\in\calTspace : r_{K_{\bx}; \bc} \in \left(2^{-\frac{j+1}{2}}, 2^{-\frac{j}{2}}\right]
 \right\},\quad j=0,1,\ldots,2J+1,\\
 D_{2(J+1);\bc} &:= 
 \left\{ 
 K_{\bx}\in\calTspace : r_{K_{\bx}; \bc} \in [0, 2^{-(J+1)}] 
 \right\},\\
 \calK &:= \calTspace \backslash \left( \bigcup_{\bc\in\calS} \bigcup_{j=0}^{2(J+1)} D_{j;\bc}\right).
\end{align*}
By the construction of Algorithm \ref{alg:bisection_meshGen}, all the elements 
$K_{\bx} \in \left\{ \calK, \{D_{j;\bc}\}_{j=0,\ldots, 2J+1}, D_{2(J+1);\bc} \right\}$
satisfy a quasi-uniformity condition within their respective sets,
that is for some $\hat{\rho} > 1$,
\begin{align}\label{eq:hxhKx}
 &\hat{\rho}^{-1}~\widehat{C}_{K_\bx} h_{\bx} \leq h_{K_{\bx}} \leq \widehat{C}_{K_\bx} h_{\bx}, \quad \text{where}\\
 \widehat{C}_{K_\bx} &= \begin{cases}
1 &\ \forall K_{\bx} \in \calK,\\
C(j,p,\delta_{\bc})
&\ \forall K_{\bx} \in D_{j;\bc},\ j=0,\ldots,2J+1,\\
C(2J+1,p,\delta_{\bc}) &\ \forall K_{\bx} \in D_{2(J+1);\bc},
\end{cases}
\qquad\text{and}\quad C(j,p,\delta_{\bc}) := 2^{\frac{-j(p+\delta_{\bc})}{2(p+1)}}.
\nonumber
\end{align}
Under the assumptions of the current proposition, recalling the definition of $\tta_K$ and $\ttb_K$ in~\eqref{eq:AlphaBetaK}, we have
\[
c_{K} = c,\quad \tta_K^{-1} = \ttb_K =
\left(\frac{c\max h_{F_{\bx}}}{h_\bx}\right)^{1/2}\overset{\eqref{eq:hFhK}}\approx
\left(\frac{ch_{K_{\bx}}}{h_\bx}\right)^{1/2}\quad \forall K \in \calT_h,
\]
where the maximum is taken over the time-like faces of $K$, $\tta_K$
and $\ttb_K$ are defined in~\eqref{eq:AlphaBetaK}. 
Moreover, due to the definition of $h_t$, Step~1 of
Algorithm~\ref{alg:bisection_meshGen} (see also \eqref{eq:hxhKx}), and assumption \eqref{eq:xtrho}, we have
\begin{equation}\label{eq:hNEW}
h_{n'} \leq h_t,\quad h_{K_{\bx}} \leq h_{\bx},
  \quad h_{n'}^{-1}\le \rho c h^{-1},
  \quad h_{\bx}^{-1} \le \rho h^{-1}.
\end{equation}
Additionally, Corollary~\ref{cor:smooth} implies that $\forall K \in \calT_h$,
\begin{align*}
 s_{t}^{v} := s_{t,K}^{v} = \min\{k_t - 2, p\},&\quad s_{x}^{v} = s_{x,K}^{v} = \min\{k_x, p\},\\ 
 s_{t}^{\bsigma} := s_{t,K}^{\bsigma} = \min\{k_t - 1, p\},&\quad s_{x}^{\bsigma} := s_{x,K}^{\bsigma} = \min\{k_x-1, p\}.
\end{align*}
Proceeding as in the proof of Proposition~\ref{prop:errorBound} and
using the information presented above,
\begin{align*}
&\frac12\N{c^{-1}(v-v\hp)}_{L^2(\Omega\times\{t_n\})}
                 +\frac12\N{\bsigma-\bsigma\hp}_{L^2(\Omega\times\{t_n\})^2}
                 \le\abs{\vs-\vsh}\DGQn
\\
&\lesssim
\bigg(\underbrace{\sum_{K=K_\bx\times I_{n'}\in\calT_h(Q_n)}\hspace{-8mm}
     c^{-1}h_\bx^2(h_{n'}^{-1}c^{-1}+h_\bx^{-1})
\Big(h_{t}^{s_{t}^v+1}\abs{v}_{H^{s_{t}^v+1}(I_{n'};H^{2,2}_\udelta(K_\bx))}
     +h_{\bx}^{s_{x}^v}\abs{v}_{L^2(I_{n'};H^{s_{x}^v+1}(K_\bx))}
\Big)^2}_{(\ast)_1}
\\
&\hspace{5mm}+\underbrace{\sum_{K=K_\bx\times I_{n'}\in\calT_h^\odot(Q_n)}\hspace{-8mm}
c(h_{n'}^{-1}c^{-1}+h_\bx^{-1})
\Big(
h_{t}^{s_{t}^\bsigma+1}\abs{\bsigma}_{H^{s_{t}^\bsigma+1}(I_{n'};L^2(K_\bx)^2)}
+h_{K_\bx}^{s_{x}^\bsigma+1}\abs{\bsigma}_{L^2(I_{n'};H^{s_{x}^\bsigma+1}(K_\bx)^2)}
\Big)^2}_{(\ast)_2}\\
&\hspace{5mm}+\underbrace{\sum_{K=K_\bx\times I_{n'}\in\calT_h^\angle(Q_n)}
c(h_{n'}^{-1}c^{-1}+h_\bx^{-1})
\Big(
h_{t}^{s_{t}^\bsigma+1}\abs{\bsigma}_{H^{s_{t}^\bsigma+1}(I_{n'};L^2(K_\bx)^2)}
+h_{K_\bx}^{1-\delta_K}\abs{\bsigma}_{L^2(I_{n'};H^{1,1}_\udelta(K_\bx)^2)}
\Big)^2}_{(\ast)_3}
\bigg)^{1/2}.
\end{align*}
From~\eqref{eq:hNEW}, we have
\[h_{n'}^{-1}c^{-1}+h_\bx^{-1}\le 2\rho h^{-1}.\]

For $\vec{a} = (a_i),\vec{b} = (b_i) \in \ell_2(\IN)$, the triangle inequality in $\ell_2(\IN)$ implies 
$ \sum_{i} (a_i + b_i)^2 
= \| \vec{a} + \vec{b} \|_2^2 
\leq \left(\| \vec{a} \|_2 + \| \vec{b} \|_2\right)^2 $.
Using this inequality, we find
\begin{align*}
(\ast)_1 & =
\sum_{K=K_\bx\times I_{n'}\in\calT_h(Q_n)}
c^{-1}h_\bx^2(h_{n'}^{-1}c^{-1}+h_\bx^{-1})
\Big( h_{t}^{s_{t}^v+1}\abs{v}_{H^{s_{t}^v+1}(I_{n'};H^{2,2}_\udelta(K_\bx))}
     +h_{\bx}^{s_{x}^v}\abs{v}_{L^2(I_{n'};H^{s_{x}^v+1}(K_\bx))} \Big)^2\\
&\lesssim  c^{-1} h\Big( h_{t}^{s_{t}^v+1} \abs{v}_{H^{s_{t}^v+1}((0,t_n); H^{2,2}_\udelta(K_\bx))}
          +h_{\bx}^{s_{x}^v} \abs{v}_{L^2((0,t_n);H^{s_{x}^v+1}(\Omega))}\Big)^2.
\end{align*}

Now, for all $K = K_{\bx} \times I_{n'} \in \calT_h^\angle(Q_n)$, $r_{K_{\bx}; \bc} = 0$, 
which implies that $K_{\bx}$ belongs to the set $D_{2(J+1);\bc}$.
On element $K_{\bx}$, we have $\delta_K = \delta_{\bc}$ and, from \eqref{eq:hxhKx},
$h_{K_{\bx}} \leq h_{\bx}~C(2J+1,p,\delta_{\bc})$.
Then, we write
\begin{align*}
h_{K_\bx}^{1-\delta_{\bc}}\abs{\bsigma}_{L^2(I_{n'};H^{1,1}_\udelta(K_\bx)^2)}
 &\overset{\eqref{eq:hxhKx}}\leq h_{\bx}^{1 - \delta_{\bc}}
  2^{-\frac{(2J+1)(p+\delta_{\bc})(1-\delta_{\bc})}{2(p+1)}}
  \abs{\bsigma}_{L^2(I_{n'};H^{1,1}_\udelta(K_\bx)^2)}\\
 &= h_{\bx}^{1 - \delta_{\bc}}
 ~2^{-\frac{(2(J+1))(p+\delta_{\bc})(1-\delta_{\bc})}{2(p+1)}}
 ~2^{+\frac{(p+\delta_{\bc})(1-\delta_{\bc})}{2(p+1)}}
 \abs{\bsigma}_{L^2(I_{n'};H^{1,1}_\udelta(K_\bx)^2)}\\
  &
    \overset{\eqref{eq:numDiscsRefined}} 
 \leq h_{\bx}^{1 - \delta_{\bc}}
 ~h_{\bx}^{p+\delta_{\bc}}
 ~\underbrace{2^{\frac{(p+\delta_{\bc})(1-\delta_{\bc})}{2(p+1)}}}_{=:\ \tilde{C}(p, \delta_c)}
 \abs{\bsigma}_{L^2(I_{n'};H^{1,1}_\udelta(K_\bx)^2)} \\
  &
    \leq \tilde{C}(p, \delta_{\bc})~h_{\bx}^{p + 1}
 \abs{\bsigma}_{L^2(I_{n'};H^{1,1}_\udelta(K_\bx)^2)}
 \leq \tilde{C}(p, \delta_{\bc})~h_{\bx}^{p + 1}
   \abs{\bsigma}_{L^2(I_{n'};H^{s_x^{\bsigma}+1,1}_\udelta(K_\bx)^2)}\\
  &
    \leq \tilde{C}(p, \delta_{\bc})~h_{\bx}^{s_x^{\bsigma} + 1}
   \abs{\bsigma}_{L^2(I_{n'};H^{s_x^{\bsigma}+1,1}_\udelta(K_\bx)^2)}.  
\end{align*}
Elements $K = K_{\bx} \times I_{n'} \in \calT_h^\odot(Q_n)$ 
may belong to either of the sets 
$\calK$, $D_{2(J+1);\bc}$ and $D_{j;\bc}$, $j=0,\ldots, 2J+1$.
Thus, we compute estimates for each case separately.

\emph{For elements $K_{\bx} \in \calK$}: 
As $1 < r_{K_{\bx}; \bc} < r_{\bc}$
and, from~\eqref{eq:hxhKx}, $h_{K_\bx}\leq  h_{\bx}$, we have
\begin{align*}
 h_{K_\bx}^{s_{x}^{\bsigma}+1}\abs{\bsigma}_{L^2(I_{n'};H^{s_x^{\bsigma}+1} (K_\bx)^2)}
&\leq h_{\bx}^{s_x^{\bsigma}+1}
\abs{\bsigma}_{L^2(I_{n'};H^{s_x^{\bsigma}+1}(K_\bx)^2)} 
 \leq h_{\bx}^{s_x^{\bsigma}+1}
\abs{\bsigma}_{L^2(I_{n'};H^{s_x^{\bsigma}+1,1}_{\udelta}(K_\bx)^2)}.  
\end{align*}
\emph{For elements $K_{\bx} \in D_{2(J+1);\bc}$, $r_{K_{\bx}; \bc} > 0$}: 
Multiplying and dividing $\abs{\bsigma}_{L^2(I_{n'};H^{s_x^{\bsigma}+1} (K_\bx)}$
by the factor $r_{K_{\bx}; \bc}^{\delta_{\bc} + s_x^{\bsigma}}$ 
we write it in weighted spaces,
and further, using $h_{K_{\bx}} \lesssim r_{K_{\bx}; \bc}$
from \cite[Lemma 4.6]{GaspozMorin} we obtain
\begin{align*}
h_{K_\bx}^{s_{x}^{\bsigma}+1}\abs{\bsigma}_{L^2(I_{n'};H^{s_x^{\bsigma}+1}(K_\bx)^2)}
&\lesssim
~r_{K_{\bx}; \bc}^{s_x^{\bsigma}+1}~r_{K_{\bx}; \bc}^{-\delta_{\bc}-s_x^{\bsigma}}
\abs{\bsigma}_{L^2(I_{n'};H^{s_x^{\bsigma}+1,1}_{\udelta}(K_\bx)^2)}\\ 
&\lesssim
~r_{K_{\bx}; \bc}^{1-\delta_{\bc}}
\abs{\bsigma}_{L^2(I_{n'};H^{s_x^{\bsigma}+1,1}_{\udelta}(K_\bx)^2)}
\lesssim 
~2^{-(1-\delta_{\bc})(J+1)}
\abs{\bsigma}_{L^2(I_{n'};H^{s_x^{\bsigma}+1,1}_{\udelta}(K_\bx)^2)}\\
&\overset{\eqref{eq:numDiscsRefined}}
\lesssim
~h_{\bx}^{p+1} \abs{\bsigma}_{L^2(I_{n'};H^{s_x^{\bsigma}+1,1}_{\udelta}(K_\bx)^2)}
\lesssim
~h_{\bx}^{s_{x}^{\bsigma}+1}\abs{\bsigma}_{L^2(I_{n'};H^{s_x^{\bsigma}+1,1}_{\udelta}(K_\bx)^2)}.
\end{align*}
\emph{For elements $K_{\bx} \in D_{j;\bc}$, $j=0,\ldots,2J+1$}: 
From \eqref{eq:hxhKx}, we have 
$h_{K_{\bx}} \leq h_{\bx}~C(j,p,\delta_{\bc})$,
which implies
\begin{align*}
 h_{K_\bx}^{s_{x}^{\bsigma}+1}\abs{\bsigma}_{L^2(I_{n'};H^{s_x^{\bsigma}+1} (K_\bx))}
 &\leq h_{K_\bx}^{s_{x}^{\bsigma}+1}~r_{K_{\bx}; \bc}^{-\delta_{\bc}-s_x^{\bsigma}}
\abs{\bsigma}_{L^2(I_{n'};H^{s_x^{\bsigma}+1,1}_{\udelta}(K_\bx)^2)}\\
 &\overset{\eqref{eq:hxhKx}}\leq h_{\bx}^{s_x^{\bsigma}+1}
 ~\underbrace{2^{-j\frac{(p+\delta_{\bc})(s_x^{\bsigma}+1)}{2(p+1)}}
 ~2^{j\frac{\delta_{\bc}+s_x^{\bsigma}}{2}}}_{\leq 1}
 \abs{\bsigma}_{L^2(I_{n'};H^{s_x^{\bsigma}+1,1}_{\udelta}(K_\bx)^2)}\\
&\leq h_{\bx}^{s_x^{\bsigma}+1}
\abs{\bsigma}_{L^2(I_{n'};H^{s_x^{\bsigma}+1,1}_{\udelta}(K_\bx)^2)}.
\end{align*}
Putting together these estimates and proceeding similarly to
$(\ast)_1$ yields
\begin{align*}
 (\ast)_2 + (\ast)_3 
  \lesssim &\; c h^{-1}
 \Bigg(
   h_t^{s_t^{\bsigma}+1}
   \Big(
    \sum_{K=K_\bx\times I_{n'}\in\calT_h(Q_n)}
    \abs{\bsigma}_{H^{s_{t}^\bsigma+1}(I_{n'};L^2(K_\bx)^2)}^{2}
   \Big)^{1/2}\\
   &
 +\Big(\sum_{K=K_\bx\times I_{n'}\in\calT^{\odot}_h(Q_n)}
    h_{K_{\bx}}^{2(s_x^{\bsigma}+1)} 
    \abs{\bsigma}_{L^2(I_{n'};H^{s_{x}^{\bsigma}+1}(K_\bx)^2)}^{2}\\
    &+\quad
    \sum_{K=K_\bx\times I_{n'}\in\calT^{\angle}_h(Q_n)}
    h_{K_{\bx}}^{2(1-\udelta_{K})} 
    \abs{\bsigma}_{L^2(I_{n'};H^{1,1}_{\udelta}(K_\bx)^2)}^{2}
  \Big)^{1/2}
 \Bigg)^{2}\\
 \lesssim &\;c h^{-1}
 \Big( h_t^{s_t^{\bsigma}+1}\abs{\bsigma}_{H^{s_{t}^\bsigma+1}((0,t_n);L^2(\Omega)^2)}
    +  h_{\bx}^{s_x^{\bsigma}+1} \abs{\bsigma}_{L^2((0,t_n);H_{\udelta}^{s_{x}^{\bsigma}+1,1}(\Omega)^2)}
 \Big)^{2}.
\end{align*}

Finally,
\begin{align*}
&\N{c^{-1}(v-v\hp)}_{L^2(\Omega\times\{t_n\})}
+\N{\bsigma-\bsigma\hp}_{L^2(\Omega\times\{t_n\})^2}\\
&\lesssim
2(c^{-1} h)^{1/2}
\Big( h_{t}^{s_{t}^v+1} \abs{v}_{H^{s_{t}^v+1}((0,t_n); H^{2,2}_\udelta(K_\bx))}
   + h_{\bx}^{s_{x}^v} \abs{v}_{L^2((0,t_n);H^{s_{x}^v+1}(\Omega))}\Big)
  \\&\qquad+
  2(c h^{-1})^{1/2}\Big(h_t^{s_t^{\bsigma}+1}\abs{\bsigma}_{H^{s_{t}^\bsigma+1}((0,t_n);L^2(\Omega)^2)}
    +   h_{\bx}^{s_x^{\bsigma}+1} \abs{\bsigma}_{L^2((0,t_n);H_{\udelta}^{s_{x}^{\sigma}+1,1}(\Omega)^2)}
\Big)\\
&\lesssim
C (c, v, \bsigma)~h^{\min\{k_t-\frac12,k_x-\frac12,p+\frac12\}}.
\end{align*}
\end{proof}

The space--time discrete space considered in Proposition~\ref{prop:rough_errorBounds_lrefMeshes} has dimension 
$$
\dim\big(\bVp\Th\big)=(\#\calT_h) \;3\; \frac{(p+1)(p+2)}2 (p+1),
$$
where $\frac{(p+1)(p+2)}{2}$ is the dimension of the scalar elemental polynomial space in a spatial element, 
$p+1$ is the dimension of the elemental polynomial space in a time interval, 
$3$ is the number of components (i.e.\ $v,\sigma_1,\sigma_2$).
Then, from Proposition~\ref{prop:refinedMeshBounds} and the quasi-uniformity of the temporal partition, we have
$$\dim\big(\bVp\Th\big)=O(p^3 h_\bx^{-2}h_t^{-1})=O(p^3 h^{-3}).$$

\section{Implementation aspects and numerical experiments}
\label{s:implement_numexp}

In all numerical experiments presented in this section the following choices have been adopted.

The partition $\calT_h$ for the space--time domain $Q=\Omega \times I$, 
is generated as a Cartesian product
of meshes $\calT_{h_{\bx}}^{\bx}$ in the spatial domain $\Omega$
and a partition $\calT_{h_t}^{t}$ of the time interval $I = (0,T)$:
\begin{align*}
\calT_h &:=  \calT_{h_{\bx}}^{\bx} \times \calT_{h_t}^{t}
= \{ K = K_{\bx} \times I_n, \text{ for } K_{\bx} \in \calT_{h_x}^{\bx} ,\ I_n \in \calT_{h_t}^{t} \}\;.
\end{align*}
The temporal mesh $\calT_{h_t}^{t}$ is uniform, i.e.\ it consists of $N_t = 2^{l_t}$ 
time steps of equal size $T2^{-l_t}$, for any given temporal refinement level $l_t \in \IN$.
The spatial mesh $\calT_{h_{\bx}}^{\bx}$ is a regular mesh of triangles which are 
shape-regular with respect to the maximum meshwidth of $\calT_{h_{\bx}}^{\bx}$
given by
\[
h_{\bx} := \max_{K_{\bx} \in \calT_{h_{\bx}}^{\bx}}{h_{K_{\bx}}}. 
\]
The wave speed is fixed at $c=1$, 
unless specified otherwise.

The Cartesian tensor-product structure of the space--time mesh allows for an
easy implementation of the space--time DG solver as a space-only DG discretization with
DG time-stepping. We implement this scheme in Python 3.7.6 using the libraries
\texttt{NumPy 1.16.4}, \cite{NumPy}, and \texttt{SciPy 1.2.0}, \cite{SciPy}, 
and FEniCS, \cite{AlnaesBlechta2015a}, for the space-only DG discretization.
The resulting linear systems of equations are solved using the sparse direct solver
\texttt{spsolve}, included in the \texttt{SciPy} submodule \texttt{scipy.sparse.linalg}. 
The library FEniCS is also used to generate the uniform meshes in Test 1.1, 
\S \ref{ss:test_1.1},
and the local bisection-tree refined meshes in Test 1.2, \S \ref{ss:test_1.2}.

In our tests, 
we estimate the order of convergence by
measuring the relative error with respect to the exact solution 
in $L^2(\Omega)$-norm at a given time $t$, indicated by $L^2(\Omega \times \{t\})$-norm.
\subsection{Test 1: smooth solution, \texorpdfstring{domain $\Omega = (0,1)^2$}{square domain}}
\label{ss:test_1.1}
Choose the space--time domain $Q = (0,1)^2 \times (0,1)$.
Consider the exact smooth solution
\begin{subequations}
 \label{eq:test11_exactSol}
 \begin{align}
  u(\bx,t) &= \sin(\pi x_1)\sin(\pi x_2)\sin(\sqrt{2}\pi t),\quad \forall (\bx,t)\in Q,\\
  v = \partial_t u,&\quad \bsigma = -\nabla_{\bx} u,
 \end{align}
\end{subequations}
to the IBVP \eqref{eq:IBVP},
with homogeneous Dirichlet boundary conditions and source $f = 0$.

We use uniform meshes with $h_{\bx} \approx h_{t} =  2^{-l}$, $l \in \IN$.
The convergence rates are given in the Tables 
\ref{tab:test11_convgRatesFG1} and \ref{tab:test11_convgRatesFG2}
for different choices of the stabilization parameters.

The tables show that for the choice $\alpha=\beta=1$ for the DG stabilization parameters 
and for uniform polynomial degrees $p = p_{v}^x = p_{v}^{t} = p_{{\bsigma}}^{t} = p_{{\bsigma}}^{x}$ 
we obtain 
$\N{v-v_h}_{L^2(\Omega\times T)}\approx\N{\bsigma-\bsigma_h}_{L^2(\Omega\times T)^2}=O(h^{p+1})$.
(Recall that here $c=1$: for constant $c\gg1$ or $c\ll1$ we expect that the parameters need to be scaled as $\alpha^{-1}=\beta=c$.)
If we over-penalize the spatial jumps of $v$ by choosing $\alpha=h_{F_\bx}^{-1}$, 
or under-penalize the normal jumps of $\bsigma$ by choosing $\beta=h_{F_\bx}$, or we combine the two strategies, the error in $v$ is not substantially affected, while the accuracy in $\bsigma$ is reduced to $O(h^p)$.
These penalization strategies have previously been proposed for the analogous DG formulation in time-harmonic regime, 
i.e.\ for the Helmholtz equation, both for Trefftz and non-Trefftz schemes, 
see e.g.\ \cite[Table 1]{TrefftzSurvey} and references therein.

If we reduce by one the spatial polynomial degree of the $\bsigma$ component, 
namely we take $p^x_\bsigma=p-1$, reflecting the mismatch in regularities (e.g.\ \eqref{eq:vsRegularity2}), 
we observe (Table~\ref{tab:test11_convgRatesFG2}) a degradation of the convergence order in the case $\alpha=\beta=1$.
With all other choices of parameters considered, 
the order of convergence for $\bsigma$ is apparently not affected by this reduction of $p^x_\bsigma$.

The last columns of Tables \ref{tab:test11_convgRatesFG1} and \ref{tab:test11_convgRatesFG2} show the values of the errors measured in $\abs{\cdot}\DG$ seminorm and the corresponding experimental convergence rates.
For $p^x_\bsigma=p$, we observe convergence rates between $O(h^{p})$ and $O(h^{p+\frac12})$, depending on the choice of the numerical flux parameters, while for $p^x_\bsigma=p-1$ we observe rates between $O(h^{p-\frac12})$ and $O(h^{p})$.
In particular, the result for $\alpha=\beta=1$ and $p^x_\bsigma=p$ demonstrates that the estimate $\abs{\vs-\vsh}\DG\lesssim C\vs h^{p+\frac12}$ of Corollary~\ref{cor:smooth} is sharp for this test case. 
In the same way as for the $L^2(\Omega\times T)$ norms, the choice $\alpha=\beta$ is preferable for $p^x_\bsigma=p$ but not for  $p^x_\bsigma=p-1$.

\begin{table}[htpb]
\begin{center}
\begin{tabular}{l l| l l| c c| c c| c c}
\multicolumn{4}{c|}{}
& \multicolumn{2}{c|}{$v_h$}
& \multicolumn{2}{c|}{$\bsigma_h$}
& \multicolumn{2}{c}{$\abs{\cdot}\DG$}\\
\hline
$p$ & DOFs & $\alpha$ & $\beta$ & Error & Rate & Error & Rate & Error & Rate\\
\hline
   &         & 1            & 1       & 3.1420E-04 & 2.08 & 1.8280E-04 & 2.11 & 1.2097E-02 & 1.51\\
 1 & 9437184 & $h_{F_{\bx}}^{-1}$ & 1       & 2.1087E-04 & 2.20 & 9.8420E-04 & 1.55 & 1.2232E-02 & 1.51\\
   &         & 1            & $h_{F_{\bx}}$ & 5.3481E-04 & 2.13 & 2.3935E-03 & 1.11 & 2.8166E-02 & 1.17\\
   &         & $h_{F_{\bx}}^{-1}$ & $h_{F_{\bx}}$ & 1.1571E-03 & 2.01 & 3.8674E-03 & 1.09 & 4.1202E-02 & 1.06\\
\noalign{\vskip 1mm}
\hline
\noalign{\vskip 1mm}
   &          & 1            & 1       & 3.1949E-06 & 2.97 & 1.2851E-06 & 3.02 & 1.0190E-04 & 2.48\\
 2 & 28311552 & $h_{F_{\bx}}^{-1}$ & 1       & 4.3013E-06 & 2.78 & 5.4307E-05 & 1.90 & 3.9421E-04 & 1.97\\
   &          & 1            & $h_{F_{\bx}}$ & 1.3603E-06 & 3.02 & 7.4015E-06 & 2.26 & 1.7307E-04 & 2.24\\
   &          & $h_{F_{\bx}}^{-1}$ & $h_{F_{\bx}}$ & 2.5958E-06 & 3.01 & 7.4462E-05 & 1.98 & 4.3786E-04 & 2.00\\
\noalign{\vskip 1mm}
\hline
\noalign{\vskip 1mm}
   &          & 1            & 1       & 1.3258E-08 & 4.00 & 6.9321E-09 & 4.00 & 5.0930E-07 & 3.49\\
 3 & 62914560 & $h_{F_{\bx}}^{-1}$ & 1       & 1.5043E-08 & 3.81 & 2.2632E-07 & 2.95 & 1.4669E-06 & 3.05\\
   &          & 1            & $h_{F_{\bx}}$ & 7.2375E-09 & 4.01 & 6.2051E-08 & 3.13 & 1.2966E-06 & 3.14\\
   &          & $h_{F_{\bx}}^{-1}$ & $h_{F_{\bx}}$ & 1.0360E-08 & 4.01 & 3.3719E-07 & 2.99 & 2.2607E-06 & 3.01\\
\end{tabular}
\caption{\small {
Empirical convergence rates of the {full-tensor space--time DG scheme} for \emph{Test 1}, 
as described in \S \ref{ss:test_1.1},
on uniform meshes with $h_{\bx} \approx h_t = 2^{-l}$, $l \in \IN$, and
fix $p = p_{v}^x = p_{v}^{t} = p_{{\bsigma}}^{t} = p_{{\bsigma}}^{x}$.
The column labelled ``DOFs'' shows the total number of degrees of freedom in space--time, and
the columns labelled ``Error'' show the errors in $L^2(\Omega\times\{T\})$-norm and $\abs{\cdot}\DG$ seminorm, for the mesh level $l = 6$.
The columns labelled ``Rate'' show the estimated orders of convergence computed using the mesh levels $l = 4,5,6$.
}}
\label{tab:test11_convgRatesFG1}
\end{center}
\begin{center}
\begin{tabular}{l l| l l| c c| c c| c c}
\multicolumn{4}{c|}{}
& \multicolumn{2}{c|}{$v_h$}
& \multicolumn{2}{c|}{$\bsigma_h$}
& \multicolumn{2}{c}{$\abs{\cdot}\DG$}\\
\hline
$p$ & DOFs & $\alpha$ & $\beta$ & Error & Rate & Error & Rate & Error & Rate\\
\hline
   &         & 1            & 1       & 8.0082E-03 & 0.82 & 4.8002E-02 & 0.96 & 3.2155E-01 & 0.53\\
 1 & 5242880 & $h_{F_{\bx}}^{-1}$ & 1       & 7.4382E-03 & 0.79 & 4.7774E-02 & 0.96 & 3.2314E-01 & 0.53\\
   &         & 1            & $h_{F_{\bx}}$ & 3.9547E-04 & 1.82 & 1.8237E-02 & 1.00 & 4.6940E-02 & 1.06\\
   &         & $h_{F_{\bx}}^{-1}$ & $h_{F_{\bx}}$ & 9.9308E-04 & 1.98 & 1.7012E-02 & 1.02 & 4.9905E-02 & 1.04\\
\noalign{\vskip 1mm}
\hline
\noalign{\vskip 1mm}
   &          & 1            & 1       & 2.8767E-04 & 2.02 & 1.7415E-04 & 2.00 & 3.6441E-03 & 1.53\\
 2 & 18874368 & $h_{F_{\bx}}^{-1}$ & 1       & 2.9474E-05 & 2.92 & 1.9826E-04 & 1.97 & 9.8431E-04 & 2.02\\
   &          & 1            & $h_{F_{\bx}}$ & 1.2984E-04 & 2.06 & 1.6404E-04 & 2.00 & 2.3004E-03 & 1.63\\
   &          & $h_{F_{\bx}}^{-1}$ & $h_{F_{\bx}}$ & 4.0016E-06 & 3.06 & 1.6247E-04 & 2.00 & 8.9843E-04 & 2.00\\
\noalign{\vskip 1mm}
\hline
\noalign{\vskip 1mm}
   &          & 1            & 1       & 3.6718E-06 & 2.99 & 1.1350E-06 & 3.00 & 3.8272E-05 & 2.49\\
 3 & 46137344 & $h_{F_{\bx}}^{-1}$ & 1       & 3.2268E-06 & 2.99 & 1.1832E-06 & 2.99 & 3.2379E-05 & 2.51\\
   &          & 1            & $h_{F_{\bx}}$ & 2.1222E-07 & 3.67 & 1.4085E-06 & 2.93 & 7.4792E-06 & 2.97\\
   &          & $h_{F_{\bx}}^{-1}$ & $h_{F_{\bx}}$ & 5.6900E-08 & 4.00 & 1.1043E-06 & 3.00 & 5.9181E-06 & 2.99\\
\end{tabular}
\caption{\small {
Empirical convergence rates of the {full-tensor space--time DG scheme} for \emph{Test 1}, 
as described in \S \ref{ss:test_1.1},
on uniform meshes with $h_{\bx} \approx h_t = 2^{-l}$, $l \in \IN$, and
fix $p = p_{v}^x = p_{v}^{t} = p_{{\bsigma}}^{t}$ and $p_{{\bsigma}}^{x} = p-1$.
The column labelled ``DOFs'' shows the total number of degrees of freedom in space--time, and
the columns labelled ``Error'' show the errors in $L^2(\Omega\times\{T\})$-norm and $\abs{\cdot}\DG$ seminorm, for the mesh level $l = 6$.
The columns labelled ``Rate'' show the estimated orders of convergence computed using the mesh levels $l = 4,5,6$.
}}
\label{tab:test11_convgRatesFG2}
\end{center}
\end{table}

\begin{table}[htpb]
\begin{center}
\begin{tabular}{c| c| c| c| c}
 & \multicolumn{4}{c}{parameter $J$}\\
\hline
mesh level $l$ & $p_{{\bsigma}}^{x}=0$ & $p_{{\bsigma}}^{x}=1$ & $p_{{\bsigma}}^{x}=2$ & $p_{{\bsigma}}^{x}=3$\\
\hline
 1 & 1 & 2 & 4 & 5 \\
 2 & 2 & 5 & 8 & 11 \\
 3 & 4 & 8 & 13 & 17 \\
 4 & 5 & 11 & 17 & 23 \\
 5 & 7 & 14 & 22 & 29 \\
 6 & 8 & 17 & 26 & 35 \\
\hline
\end{tabular}
\caption{\small The parameters used in Algorithm~\ref{alg:bisection_meshGen} for generating meshes with local refinement for the $\Gamma$-shaped spatial domain 
$\Omega := (-\frac{1}{2},\frac{1}{2})^2\backslash \{(0,\frac{1}{2})\times(-\frac{1}{2},0)\}$ of {\em Test 2}.
Local refinement is performed only at the corner $\bx_{0} = (0,0)^{\top}$.
The cut-off radius $R_{c} = 0.245$ and refinement weight $\delta = 1/3$ (see Remark \ref{rmk:deltai>0}). 
The number of local refinements, given by parameter $J$, are presented for mesh resolutions $h_{\bx} = 2^{-l}$, $l\in\{1,2,\ldots,6\}$, and polynomial degrees $p_{{\bsigma}}^{x} \in \{0,1,2,3\}$.
}
\label{tab:test12_bisMeshData}
\end{center}
\end{table}

\subsection{Test 2: singular solution, \texorpdfstring{$\Gamma$}{Gamma}-shaped domain}
 \label{ss:test_1.2}
Choose the $\Gamma$-shaped spatial domain 
$\Omega := (-\frac{1}{2},\frac{1}{2})^2\backslash \{(0,\frac{1}{2})\times(-\frac{1}{2},0)\}$
and the end time $T=1$.
The polar coordinates centered at the re-entrant corner are
\begin{align*}
r = \N{\bx}_2,\quad \theta = \arctan(x_2/x_1),\quad \forall\ \bx \in \Omega.
\end{align*}
Consider the singular solution
\begin{subequations}
\label{eq:test12_exactSol}
\begin{align}
 u(\bx,t) = r^{\gamma}\sin(\gamma \theta)\sin(\sqrt{2} \pi t),
   &\quad\text{for } \gamma = \frac{2}{3},\quad \forall (\bx,t)\in Q,\\
 v = \partial_t u,&\quad \bsigma = -\nabla_{\bx} u,
\end{align}
\end{subequations}
to the IBVP \eqref{eq:IBVP}, 
with Neumann boundary conditions
and with the prescribed source term
\begin{equation*}
 f(\bx,t) = -2 \pi^2 r^{\gamma} \sin(\gamma \theta)\sin(\sqrt{2}\pi t),\quad \forall (\bx,t)\in Q.
\end{equation*}
The parameters used in Algorithm~\ref{alg:bisection_meshGen} for generating locally refined meshes in this experiment are presented in Table~\ref{tab:test12_bisMeshData} and some examples of these meshes are shown in Figure~\ref{fig:test12_meshes}.
The convergence rates are given in Tables~\ref{tab:test12_convgRatesFG1} (quasi-uniform meshes), \ref{tab:test12_convgRatesFG2} and \ref{tab:test12_convgRatesFG3} 
(locally refined meshes,  $p = p_{v}^t = p_{\bsigma}^{t} = p_{v}^{x}=p_\bsigma^x$ 
and $p = p_{v}^t = p_{\bsigma}^{t} = p_{v}^{x}=p_\bsigma^x+1$, respectively), 
for various choices of the stabilization parameters.

The tables show that for the singular solution \eqref{eq:test12_exactSol}, 
the corner-refined meshes, as described in \S\ref{sec:RefMesXDom}, 
allow the space--time DG-scheme to preserve the 
same orders of convergence as for the smooth-solution case of \S\ref{ss:test_1.1}.

\begin{figure}[htpb]
\begin{center}
\includegraphics[height = 0.24\textwidth,clip, trim=50 60 60 60]{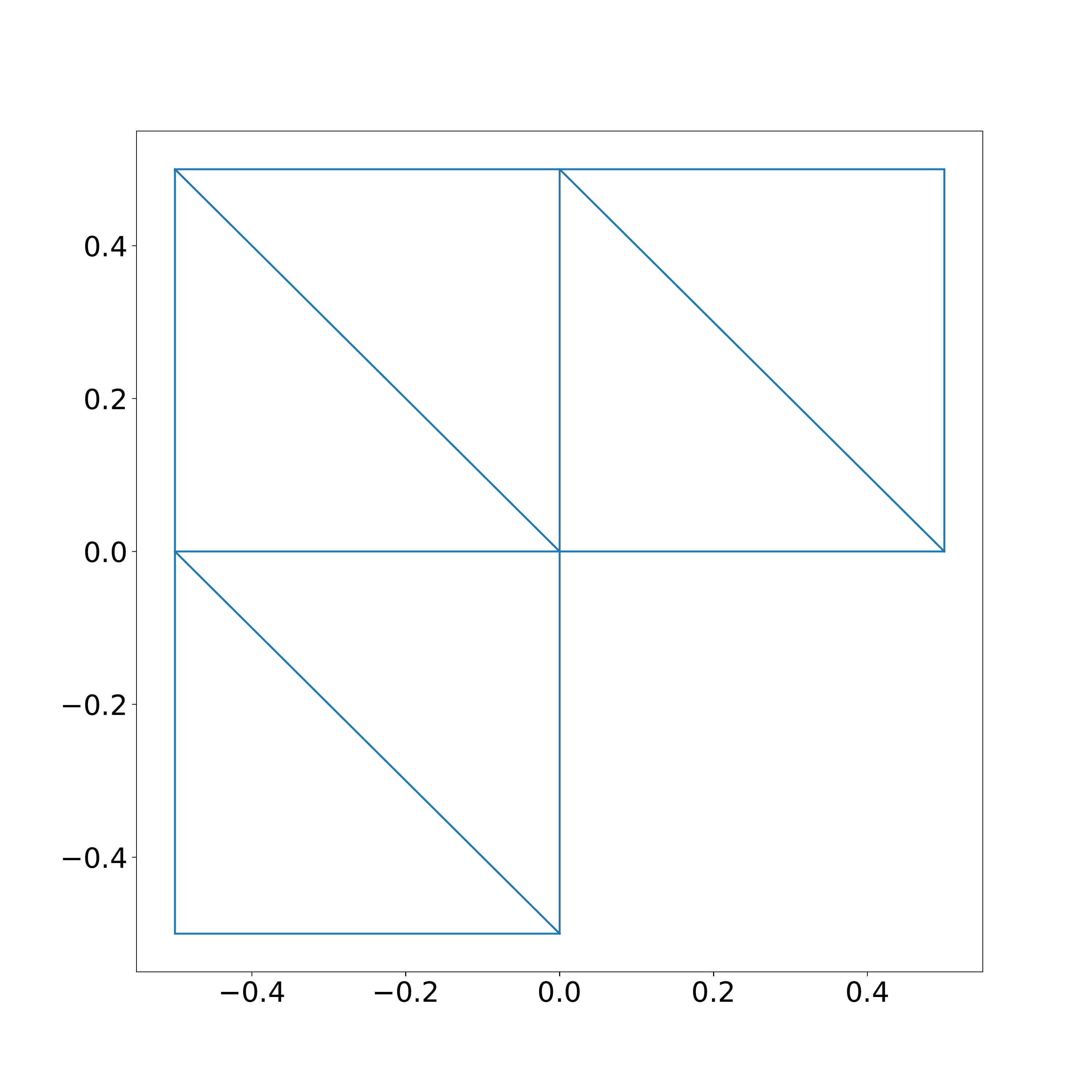}
\includegraphics[height = 0.24\textwidth,clip, trim=50 60 60 60]{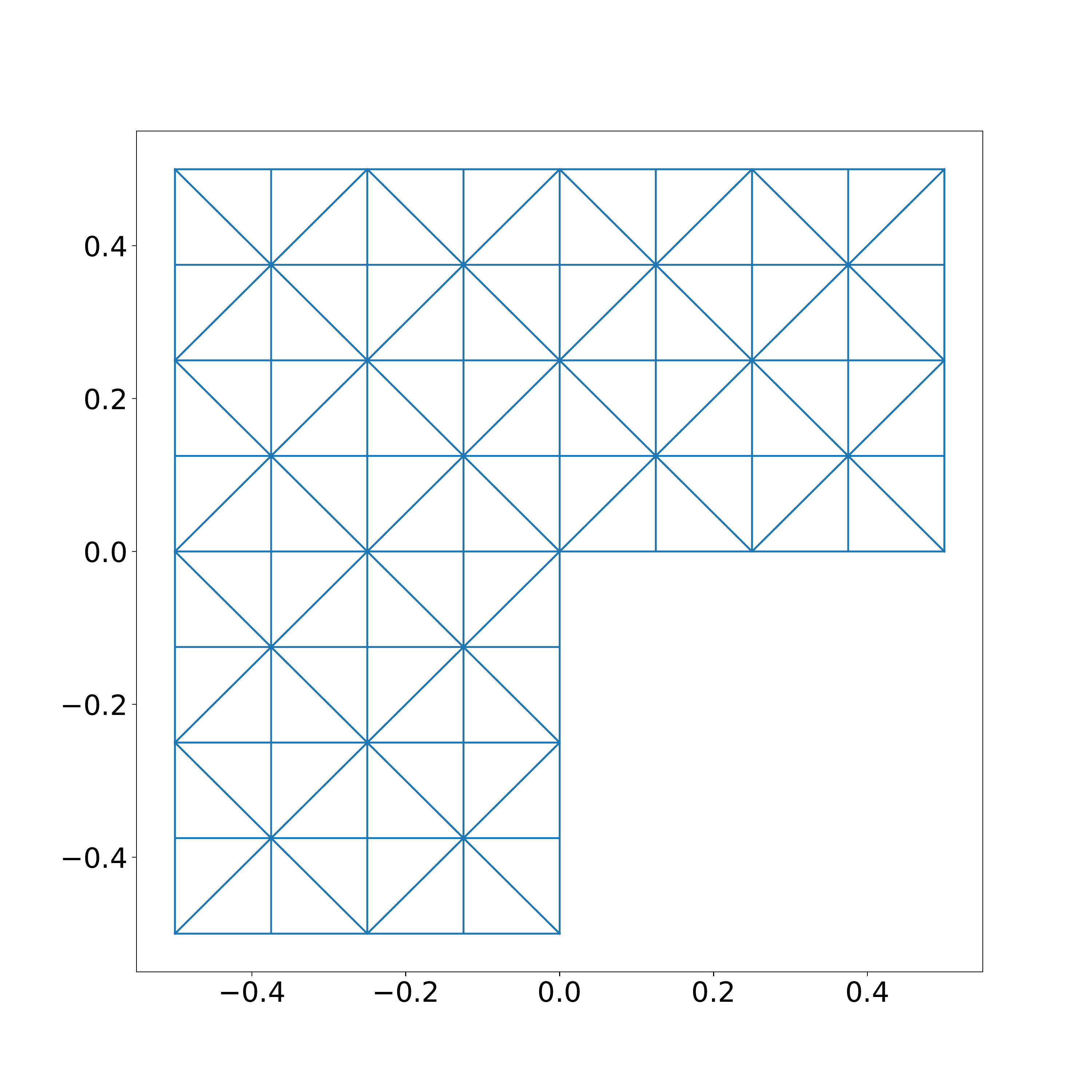}
\includegraphics[height = 0.24\textwidth,clip, trim=50 60 60 60]{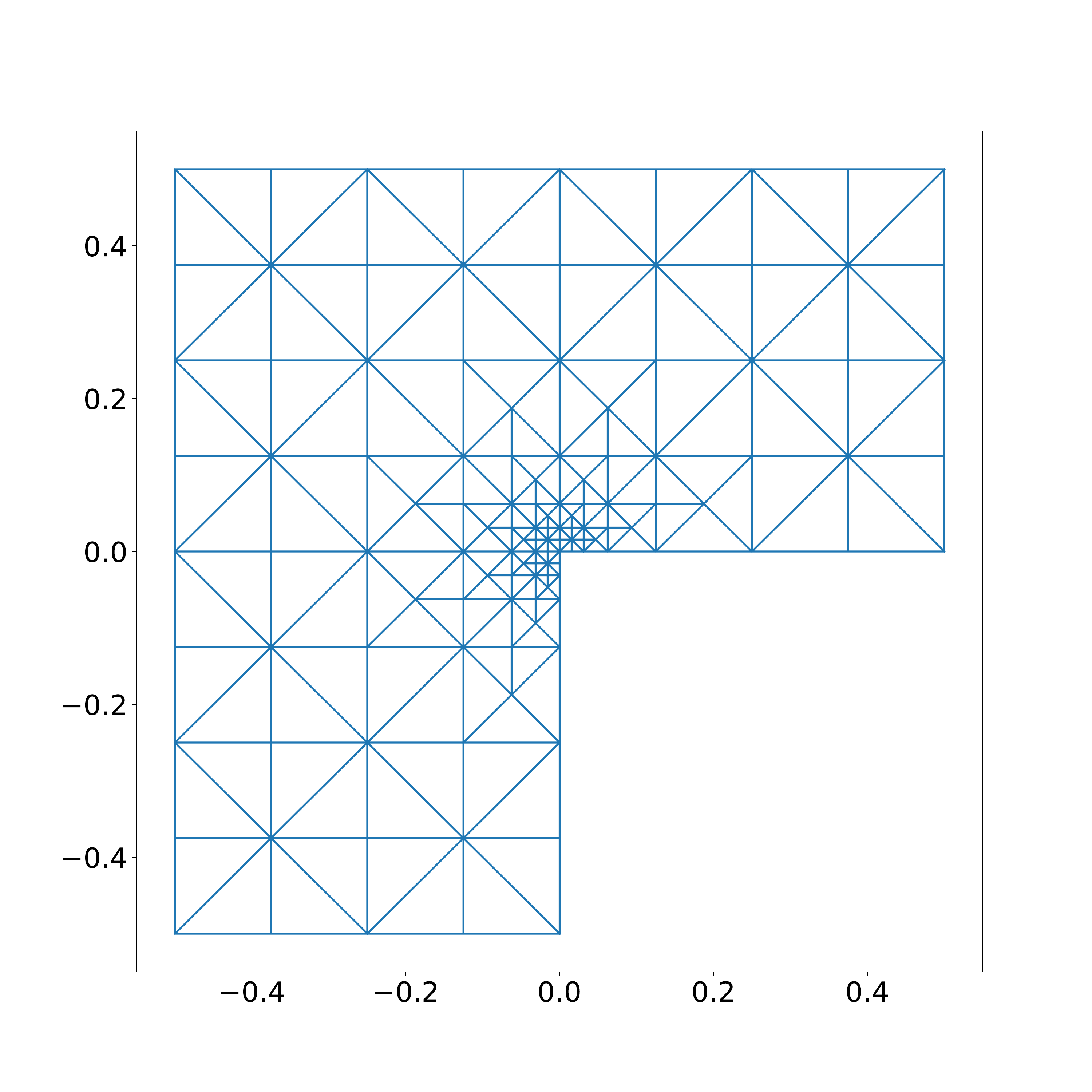}
\includegraphics[height = 0.24\textwidth,clip, trim=50 60 60 60]{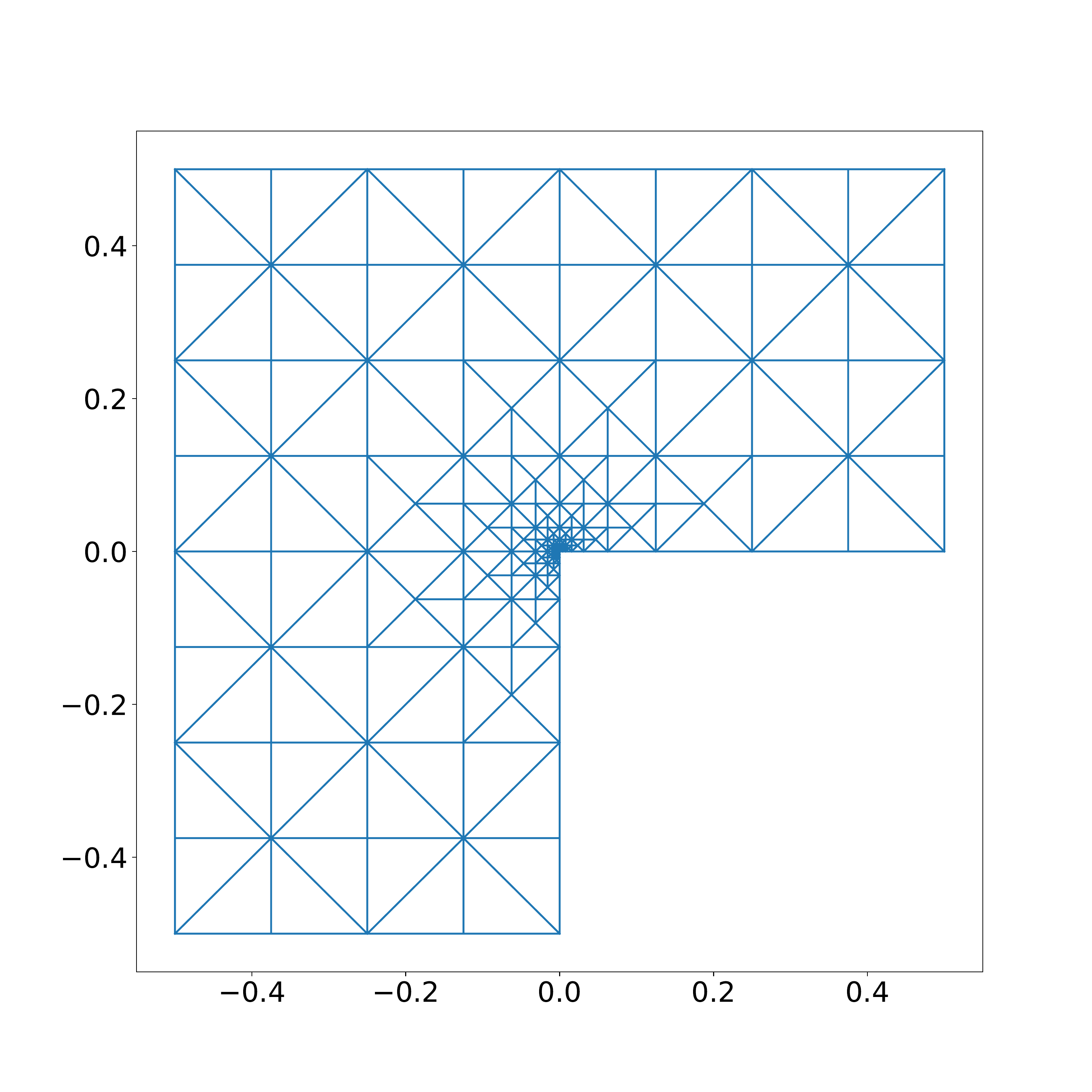}
\caption{\small {$\Gamma$-shaped spatial domain 
$\Omega := (-\frac{1}{2},\frac{1}{2})^2\backslash \{(0,\frac{1}{2})\times(-\frac{1}{2},0)\}$.
 [Column 1] Initial mesh $\calTspaceinit$.
 [Column 2] Uniform mesh for $h_{\bx} = 2^{-2}$.
 Meshes generated from the initial mesh $\calTspaceinit$ using Algorithm \ref{alg:bisection_meshGen} with 
 parameters given in Table~\ref{tab:test12_bisMeshData}, 
 only for the corner $\bx_{0} = (0,0)^{\top}$:
 {[column 3]} $p_x^{\bsigma} = 1$, $h_{\bx} = 2^{-2}$;
 {[column 4]} $p_x^{\bsigma} = 2$, $h_{\bx} = 2^{-2}$.}}
 \label{fig:test12_meshes}
\end{center}
\end{figure}

\begin{table}[htpb]
\begin{center}
\begin{tabular}{l| c c| c c | c c}
& \multicolumn{2}{c|}{$v_h$}
& \multicolumn{2}{c|}{$\bsigma_h$}
& \multicolumn{2}{c}{$\abs{\cdot}\DG$}\\
\hline
$p$ & Error & Rate & Error & Rate & Error & Rate\\
\hline
 1 & 6.5455E-02 & 0.66 & 5.2710E-02 & 0.60 & 3.9363E-02 & 0.69\\
 2 & 4.8317E-02 & 0.67 & 3.9329E-02 & 0.60 & 2.8139E-02 & 0.63\\
 3 & 6.8525E-02 & 0.66 & 5.3033E-02 & 0.59 & 3.6141E-02 & 0.61\\
\hline
\end{tabular}
\caption{\small {
Empirical convergence rates of the {full-tensor space--time DG scheme} for \emph{Test 2}, 
described in \S \ref{ss:test_1.2},
on \textbf{uniform} meshes with $h_{\bx} \approx h_t = 2^{-l}$, $l \in \IN$, 
and with fixed stabilization parameters $\alpha = 1$ and $\beta = 1$,
and polynomial degree 
$p = p_{v}^x = p_{v}^{t} = p_{{\bsigma}}^{t} = p_{{\bsigma}}^{x}$.
The column labelled ``Error'' 
shows the estimated numerical error in $L^2(\Omega\times\{T\})$-norm and $\abs{\cdot}\DG$ seminorm: 
if $p \in \{1,2\}$, it is given for the mesh level $l = 6$, 
else if $p = 3$, for $l = 5$.
The column labelled ``Rate'' shows the estimated convergence rate: 
if $p \in \{1,2\}$, it is computed with the mesh levels $l = 4,5,6$, 
else if $p = 3$, with $l=3,4,5$.
}}
\label{tab:test12_convgRatesFG1}
\end{center}
\begin{center}
\begin{tabular}{l l| l l| c c| c c| c c}
\multicolumn{4}{c|}{}
& \multicolumn{2}{c|}{$v_h$}
& \multicolumn{2}{c|}{$\bsigma_h$}
& \multicolumn{2}{c}{$\abs{\cdot}\DG$}\\
\hline
$p$ & DOFs & $\alpha$ & $\beta$ & Error & Rate & Error & Rate & Error & Rate\\
\hline
   &         & 1            & 1       & 9.9964E-04 & 1.95 & 1.1227E-03 & 1.88 & 1.5032E-03 & 1.91\\
 1 & 9780480 & $h_{F_{\bx}}^{-1}$ & 1       & 1.0349E-03 & 1.94 & 4.0737E-03 & 1.11 & 4.2031E-03 & 1.27\\
   &         & 1            & $h_{F_{\bx}}$ & 9.9909E-04 & 1.95 & 1.3705E-03 & 1.75 & 1.6960E-03 & 1.83\\
   &         & $h_{F_{\bx}}^{-1}$ & $h_{F_{\bx}}$ & 1.0356E-03 & 1.94 & 4.2489E-03 & 1.13 & 4.3733E-03 & 1.27\\
\noalign{\vskip 1mm}
\hline
\noalign{\vskip 1mm}
   &          & 1            & 1       & 1.8713E-05 & 3.00 & 2.4310E-05 & 2.85 & 3.0678E-05 & 2.91\\
 2 & 35541504 & $h_{F_{\bx}}^{-1}$ & 1       & 1.8715E-05 & 3.00 & 2.7236E-05 & 2.78 & 3.3047E-05 & 2.86\\
   &          & 1            & $h_{F_{\bx}}$ & 1.8714E-05 & 3.00 & 2.4686E-05 & 2.84 & 3.0978E-05 & 2.90\\
   &          & $h_{F_{\bx}}^{-1}$ & $h_{F_{\bx}}$ & 1.8718E-05 & 3.00 & 3.5567E-05 & 2.59 & 4.0191E-05 & 2.73\\
\noalign{\vskip 1mm}
\hline
\noalign{\vskip 1mm}
   &          & 1            & 1       & 1.0517E-05 & 4.00 & 1.3890E-05 & 3.79 & 1.7422E-05 & 3.88\\
 3 & 12441600 & $h_{F_{\bx}}^{-1}$ & 1       & 1.0518E-05 & 4.00 & 1.4359E-05 & 3.77 & 1.7799E-05 & 3.87\\
   &          & 1            & $h_{F_{\bx}}$ & 1.0518E-05 & 4.00 & 1.3904E-05 & 3.79 & 1.7434E-05 & 3.88\\
   &          & $h_{F_{\bx}}^{-1}$ & $h_{F_{\bx}}$ & 1.0518E-05 & 4.00 & 1.4555E-05 & 3.76 & 1.7957E-05 & 3.86\\
\end{tabular}
\caption{\small{Empirical convergence rates of the {full-tensor space--time DG scheme} 
for \emph{Test 2}, as described in \S \ref{ss:test_1.2},
on \textbf{locally refined} meshes with 
$h_{\bx} \approx h_t = 2^{-l}$, $l \in \IN$, 
and with fixed polynomial degree 
$p = p_{v}^x = p_{v}^{t} = p_{{\bsigma}}^{t} = p_{{\bsigma}}^{x}$.
The column labelled ``DOFs'' shows the total number of degrees of freedom in space--time, and the column labelled ``Error''  shows the estimated numerical error in $L^2(\Omega\times\{T\})$-norm and $\abs{\cdot}\DG$ seminorm: 
if $p \in \{1,2\}$, they are given for the mesh level $l = 6$, 
else if $p = 3$, for $l = 5$.
The column labelled ``Rate'' shows the estimated convergence rate: 
if $p \in \{1,2\}$, 
it is computed with the mesh levels $l = 4,5,6$, 
else if $p = 3$, with $l=3,4,5$.
}}
\label{tab:test12_convgRatesFG2}
\end{center}
\begin{center}
\begin{tabular}{l l| l l| c c| c c| c c}
\multicolumn{4}{c|}{}
& \multicolumn{2}{c|}{$v_h$}
& \multicolumn{2}{c|}{$\bsigma_h$}
& \multicolumn{2}{c}{$\abs{\cdot}\DG$}\\
\hline
$p$ & DOFs & $\alpha$ & $\beta$ & Error & Rate & Error & Rate & Error & Rate\\
\hline
   &         & 1            & 1       & 2.4743E-02 & 0.92 & 2.6202E-02 & 0.83 & 3.6038E-02 & 0.87\\
 1 & 4147200 & $h_{F_{\bx}}^{-1}$ & 1       & 2.4702E-02 & 0.91 & 2.6246E-02 & 0.83 & 3.6042E-02 & 0.87\\
   &         & 1            & $h_{F_{\bx}}$ & 2.5666E-02 & 0.96 & 2.6355E-02 & 0.85 & 3.6788E-02 & 0.91\\
   &         & $h_{F_{\bx}}^{-1}$ & $h_{F_{\bx}}$ & 2.5740E-02 & 0.96 & 2.4506E-02 & 0.88 & 3.5540E-02 & 0.92\\
\noalign{\vskip 1mm}
\hline
\noalign{\vskip 1mm}
   &          & 1            & 1       & 7.6086E-04 & 2.00 & 8.4941E-04 & 1.85 & 1.1404E-03 & 1.93\\
 2 & 19560960 & $h_{F_{\bx}}^{-1}$ & 1       & 7.5923E-04 & 2.00 & 8.4984E-04 & 1.85 & 1.1396E-03 & 1.93\\
   &          & 1            & $h_{F_{\bx}}$ & 7.5511E-04 & 2.00 & 8.4990E-04 & 1.85 & 1.1369E-03 & 1.93\\
   &          & $h_{F_{\bx}}^{-1}$ & $h_{F_{\bx}}$ & 7.5479E-04 & 2.00 & 8.3428E-04 & 1.86 & 1.1250E-03 & 1.93\\
\noalign{\vskip 1mm}
\hline
\noalign{\vskip 1mm}
   &          & 1            & 1       & 1.6852E-05 & 3.00 & 2.1861E-05 & 2.84 & 2.7602E-05 & 2.91\\
 3 & 57919488 & $h_{F_{\bx}}^{-1}$ & 1       & 1.6821E-05 & 3.00 & 2.1854E-05 & 2.84 & 2.7579E-05 & 2.91\\
   &          & 1            & $h_{F_{\bx}}$ & 1.6732E-05 & 3.00 & 2.1896E-05 & 2.84 & 2.7557E-05 & 2.91\\
   &          & $h_{F_{\bx}}^{-1}$ & $h_{F_{\bx}}$ & 1.6726E-05 & 3.00 & 2.1606E-05 & 2.85 & 2.7324E-05 & 2.91\\
\end{tabular}
\caption{\small{Convergence rates of the {full-tensor space--time DG scheme} for \emph{Test 2}, 
as described in \S \ref{ss:test_1.2},
on \textbf{locally refined} meshes with $h_{\bx} \approx h_t = 2^{-l}$, $l \in \IN$, and
fix $p = p_{v}^x = p_{v}^{t} = p_{{\bsigma}}^{t}$ and $p_{{\bsigma}}^{x} = p-1$.
The column labelled ``DOFs'' shows the total number of degrees of freedom in space--time, and
the column labelled ``Error'' shows the estimated numerical error in $L^2(\Omega\times\{T\})$-norm and $\abs{\cdot}\DG$ seminorm, for the mesh level $l = 6$.
The column labelled ``Rate'' shows the estimated order of convergence 
computed using the mesh levels $l = 4,5,6$.
}}
\label{tab:test12_convgRatesFG3}
\end{center}
\end{table}

\subsection{Test 3: smooth solution, heterogeneous medium}
\label{ss:test_1.3}
In this test case, we simulate the reflection and transmission of a wave 
at the interface of two different media.
We follow the experimental setup provided in \cite[{\S6.7}]{StockerSchoeberl}
and use its results as reference for comparison.

Consider the space--time domain $Q = (0,2)^2 \times (0,1)$. The wavespeed is the piecewise
constant function
\[
    c(\bx) = \begin{cases}
              c_1 = 1,& x_1 \leq 1.2,\\
              c_2 = 3,& x_1 >    1.2.
             \end{cases}
\]
As the initial condition, we take a Gaussian wave given by
\[
    u_{0} = \exp(-\N{\bx - \bx_0}^2/\delta^2), \;\;
    v_{0} = 0, \;\;
    \bsigma_{0} = - \nabla_{\bx} u,
\]
where $\bx_0 = (1,1)^{\top}$ and $\delta = 0.01$. 
We consider homogeneous Dirichlet boundary conditions. 
The computations are performed with polynomial degrees $p_x^{v} = p_{x}^{\bsigma} + 1 = 4$ and $p_t^{v} = p_t^{\bsigma} = 1$. 
The quasi-uniform spatial mesh used in this experiment is shown in Figure~\ref{fig:test13_merged} (center panel).
The uniform temporal mesh has step size $h_t = 2^{-6}$. 
We use approximately $7.35\times10^7$ degrees of freedom in space--time; 
the size of each linear system solved is approximately $1.15\times10^{6}$.

\begin{figure}[htpb]
 \centering
 \includegraphics[width = 0.27\textwidth,clip, trim=0 0 0 0]{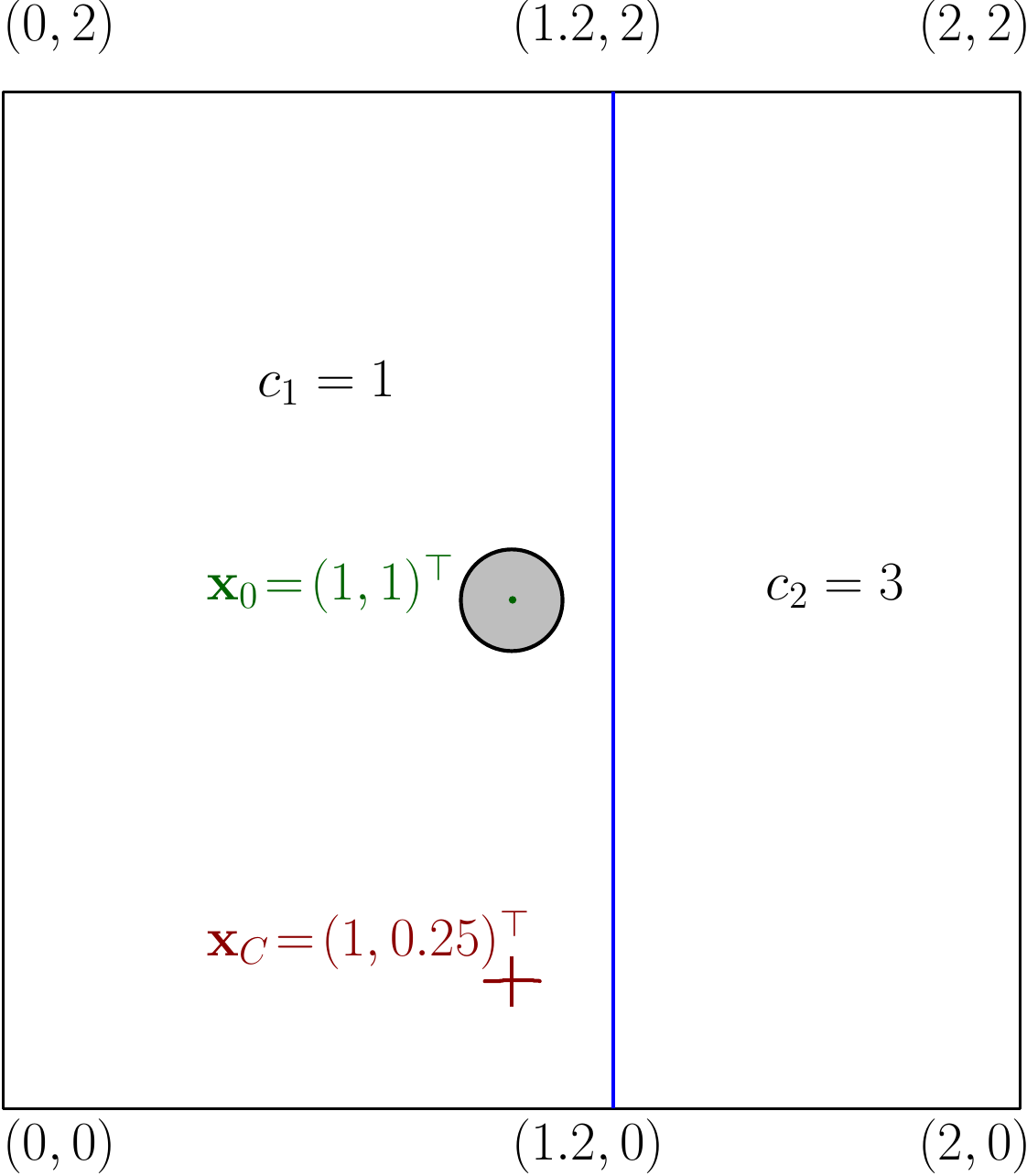}
 \includegraphics[width = 0.31\textwidth,clip, trim=50 70 70 70]{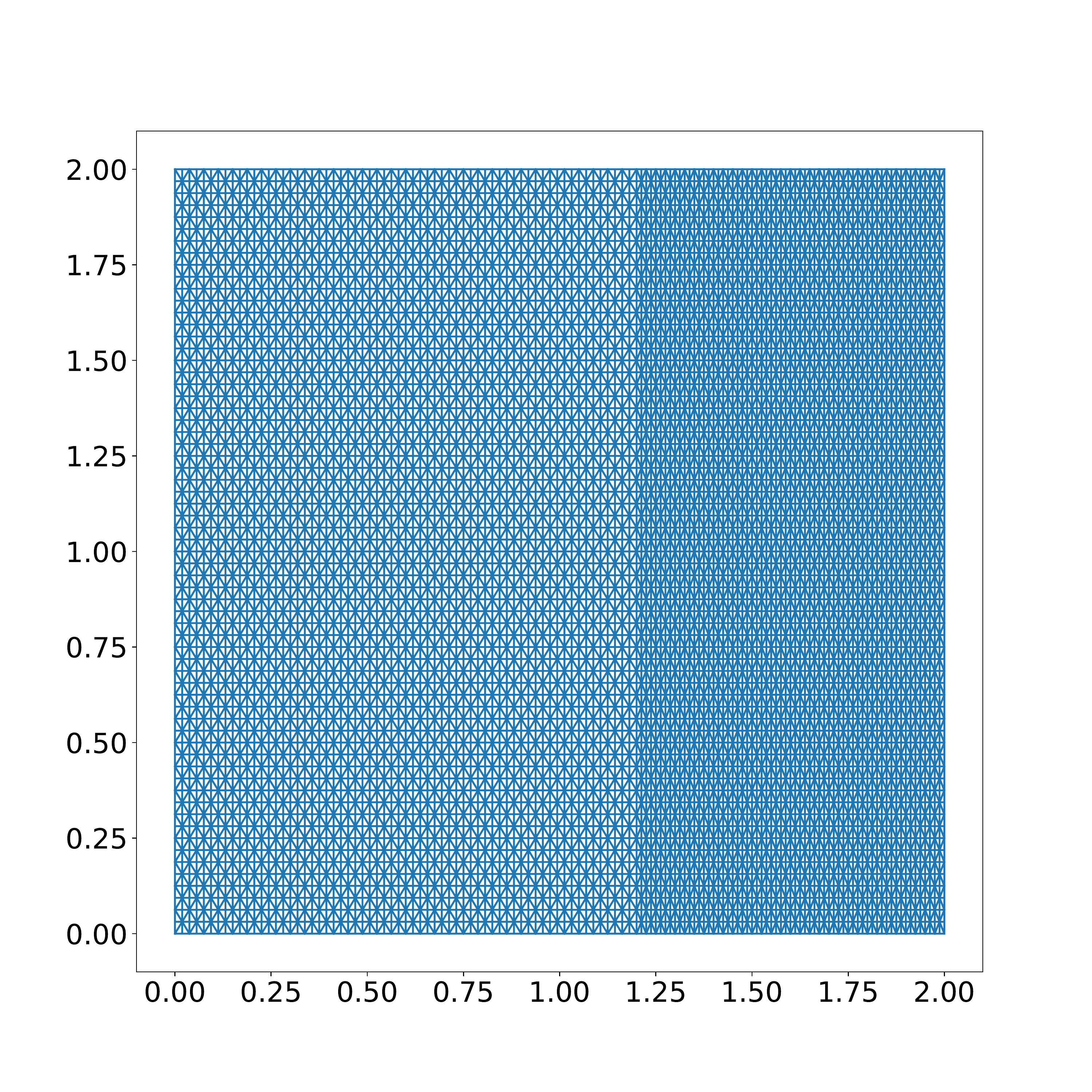}
 \includegraphics[width = 0.40\textwidth,clip, trim=50 50 110 90]{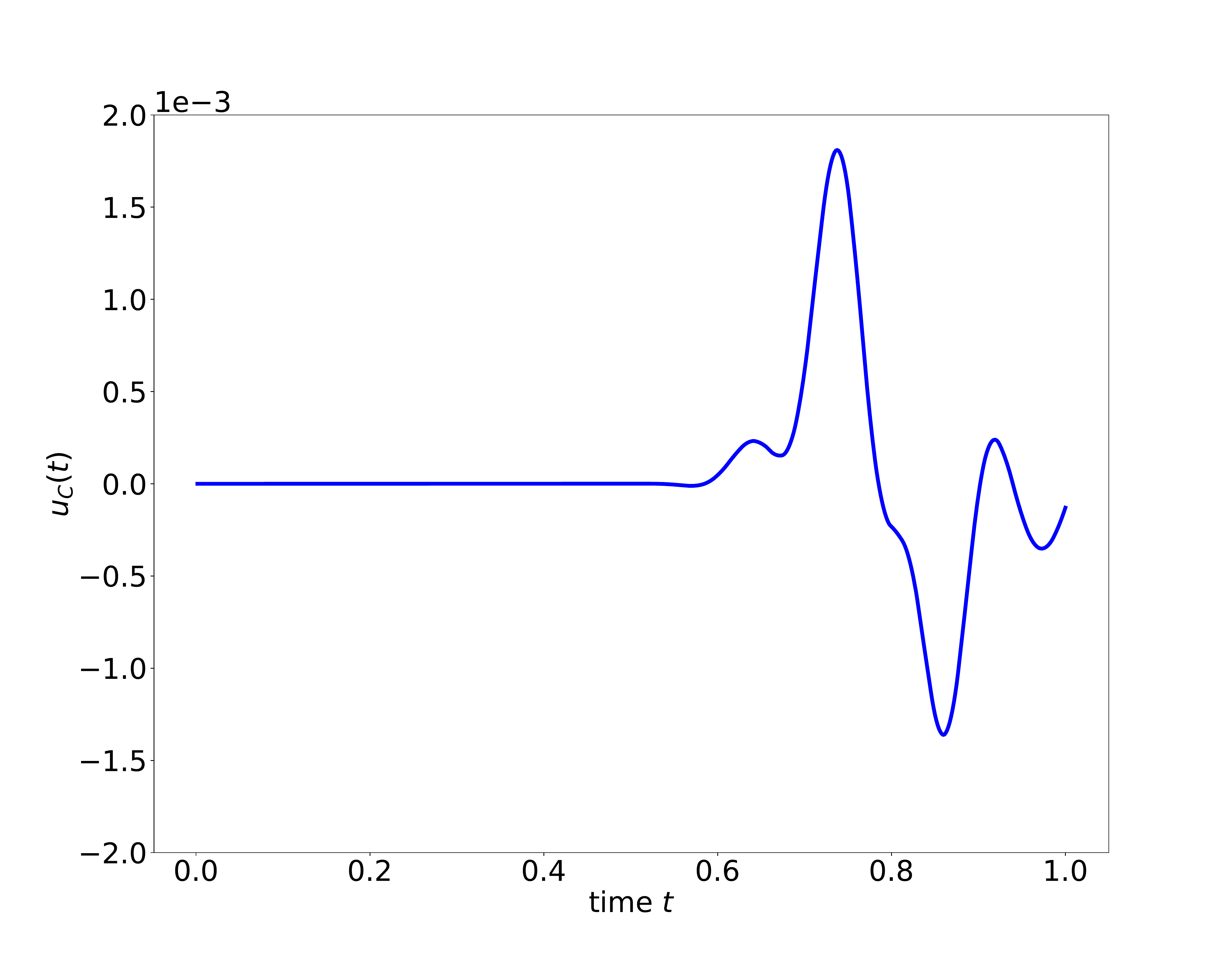}
\caption{\small {[left] Set-up of \emph{Test 3}, \S \ref{ss:test_1.3}:
the grey region indicates the initial Gaussian wave with centre $\bx_0 = (1,1)^{\top}$,
the material interface is at $x_1=1.2$, indicated by the blue line, 
with wave speed $c_1$ to the left and $c_2$ to the right of the interface,
and the point $\bx_{C} = (1, 0.25)^{\top}$ in whose vicinity we measure the wave signal.\newline
[center] The quasi-uniform spatial mesh with meshwidth $h_{\bx} \approx 0.0365$ conforming to the material interface.
\newline
[right]
Solution signal $u_C$ measured on the cell $\Omega_C$ containing the point $\bx_C = (1,0.25)^{\top}$. 
The observed signal is in agreement with the measurements available for the same experiment in \cite{StockerSchoeberl}.}}
\label{fig:test13_merged}
\end{figure}

Snapshots of the solution are shown in Figure~\ref{fig:test13_snapshots}.
First, the initial condition evolves in the left homogeneous medium.
At time $t = 0.2$, the wave
crosses over into the medium with higher wave speed. The snapshot at $t = 0.3$
shows that a part of the incident wave is transmitted across the interface with a higher
speed and a shallow wavefront, and another part is reflected back.
Finally, at $t = 0.4$, we also observe the weaker Huygens wave phenomenon.

Let $\Omega_C$ be the mesh element that contains the point $\bx_{C} = (1,0.25)$.
We measure $v_C(t) = \int_{\Omega_C} v(\bx,t) d\bx$ as a time series signal and compute
an approximation of $u_C(t) = \int_{\Omega_C} u(\bx,t) d\bx$ 
by integrating the signal $v_C(t)$ in time using the trapezoidal rule.
Figure~\ref{fig:test13_merged} (right panel) suggests that we numerically separate the three incoming waves: 
the very weak Huygens wave arrives first, followed by the initial wave and the reflected wave.

These observations are in agreement with the results for the same experimental set-up in \cite[Fig.~11]{StockerSchoeberl}.
The signals are proportional but their amplitudes differ because the mesh elements $\Omega_C$ 
on which they are integrated are different.
They differ also in that here we plot the integral of the signal over $\Omega_C$, 
as opposed to its $L^1(\Omega_C)$ norm, in order to capture the sign change of the reflected wave.

\begin{figure}[htb]
 \begin{minipage}{0.245\textwidth} \centering
 \includegraphics[width = \textwidth, clip, trim = 80 0 80 0]{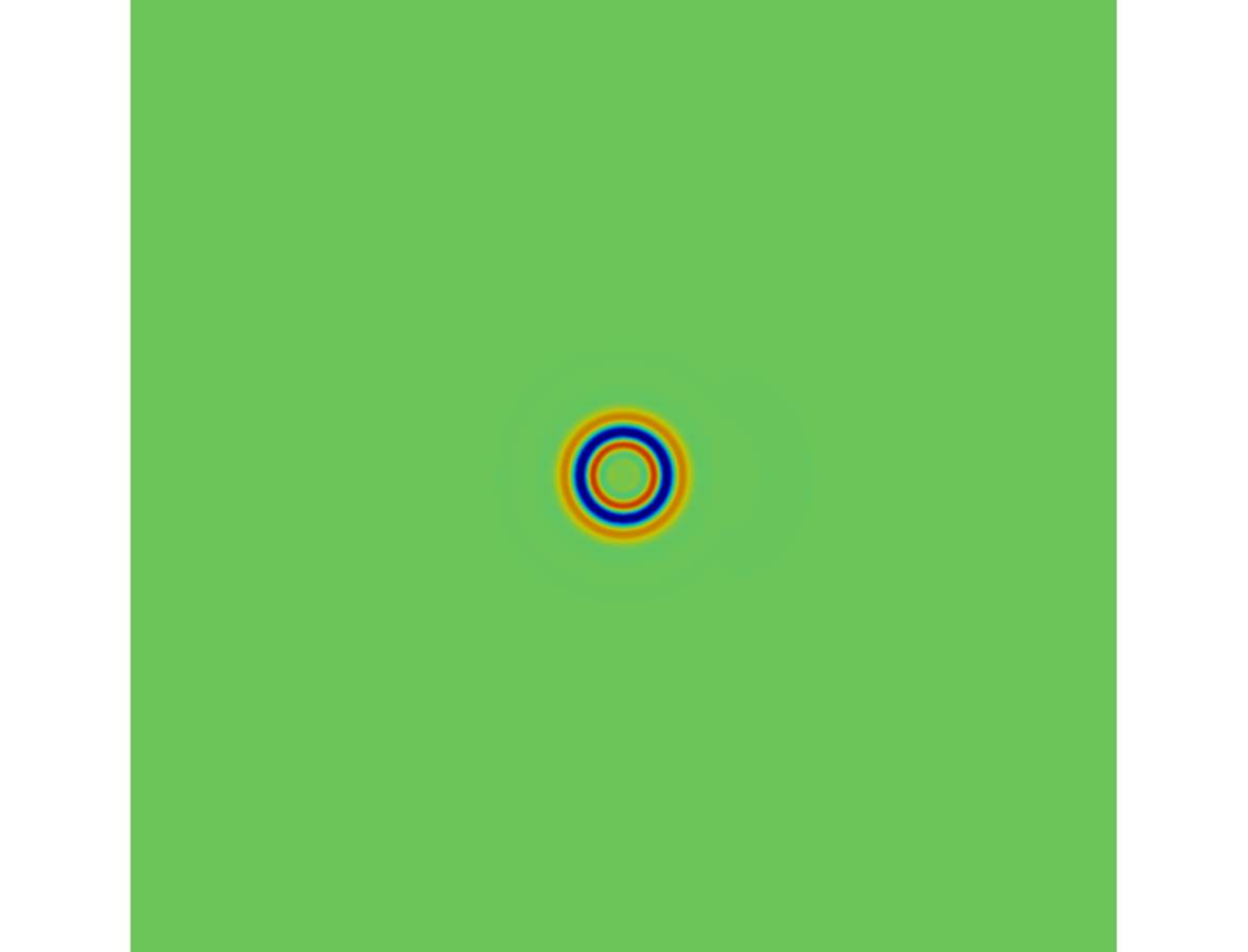}
 \subcaption{\small{$t=0.1$}}
\end{minipage}\begin{minipage}{0.245\textwidth} \centering
 \includegraphics[width = \textwidth, clip, trim = 80 0 80 0]{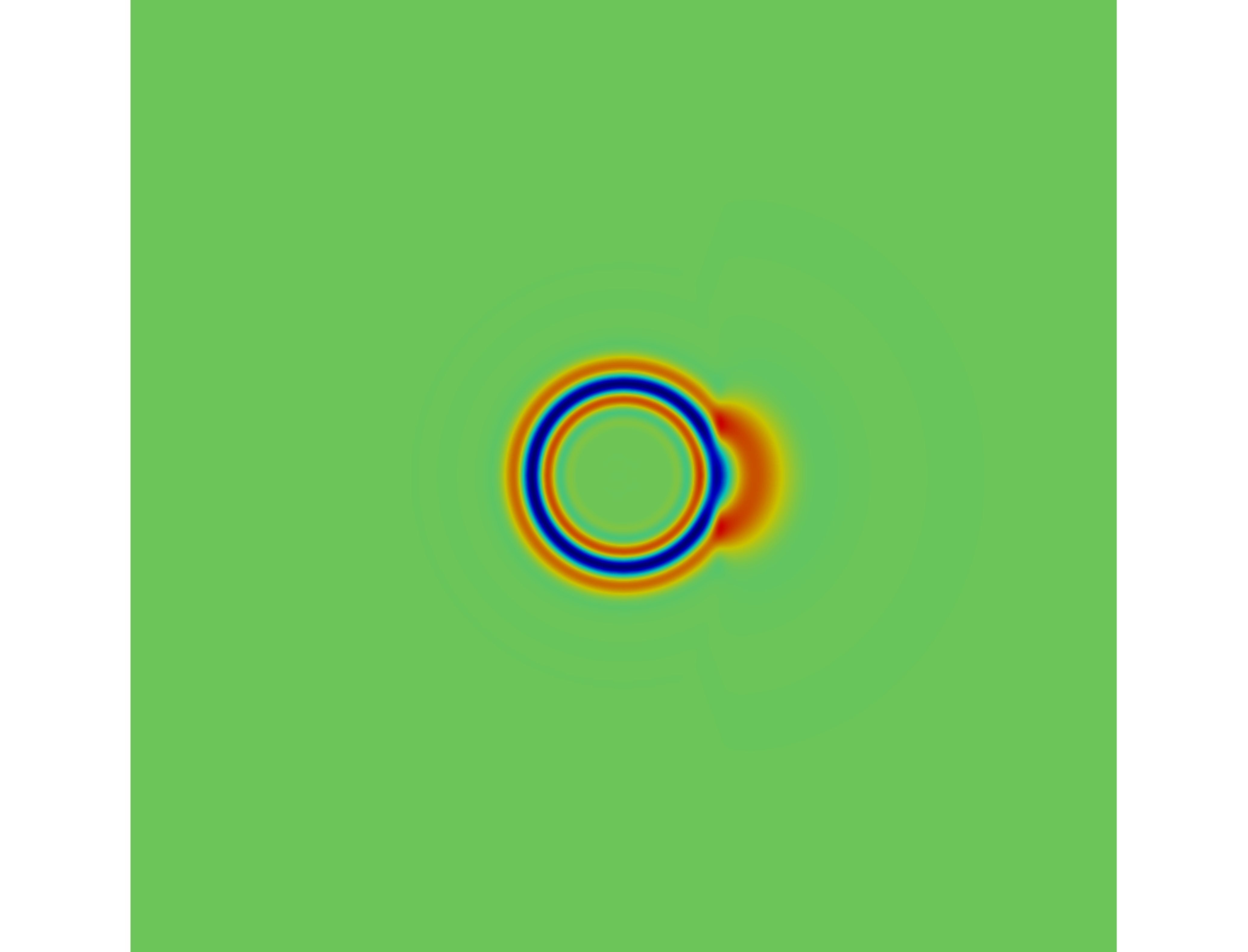}
 \subcaption{\small{$t=0.2$}}
\end{minipage}\begin{minipage}{0.245\textwidth} \centering
 \includegraphics[width = \textwidth, clip, trim = 80 0 80 0]{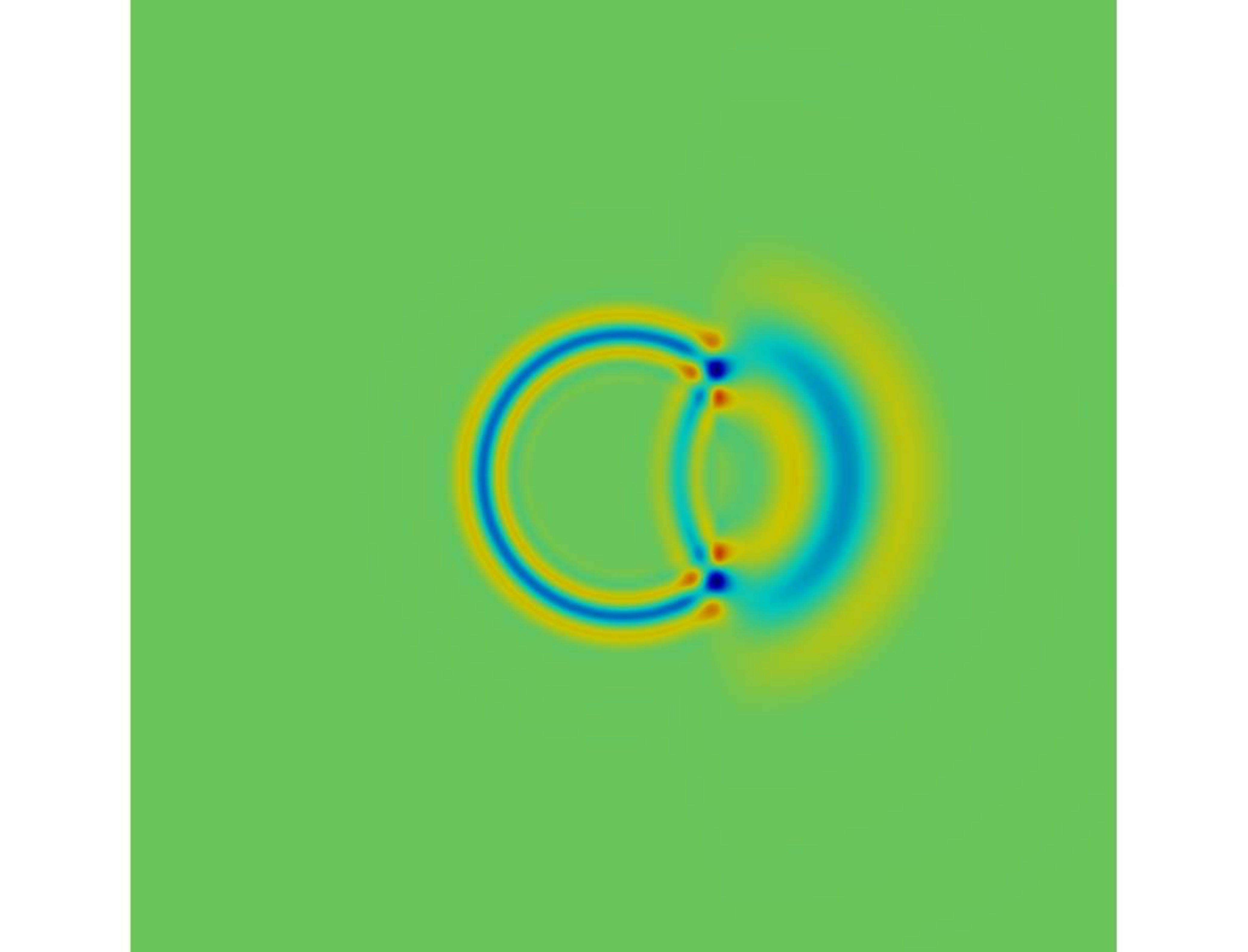}
 \subcaption{\small{$t=0.3$}}
\end{minipage}\begin{minipage}{0.245\textwidth} \centering
 \includegraphics[width = \textwidth, clip, trim = 80 0 80 0]{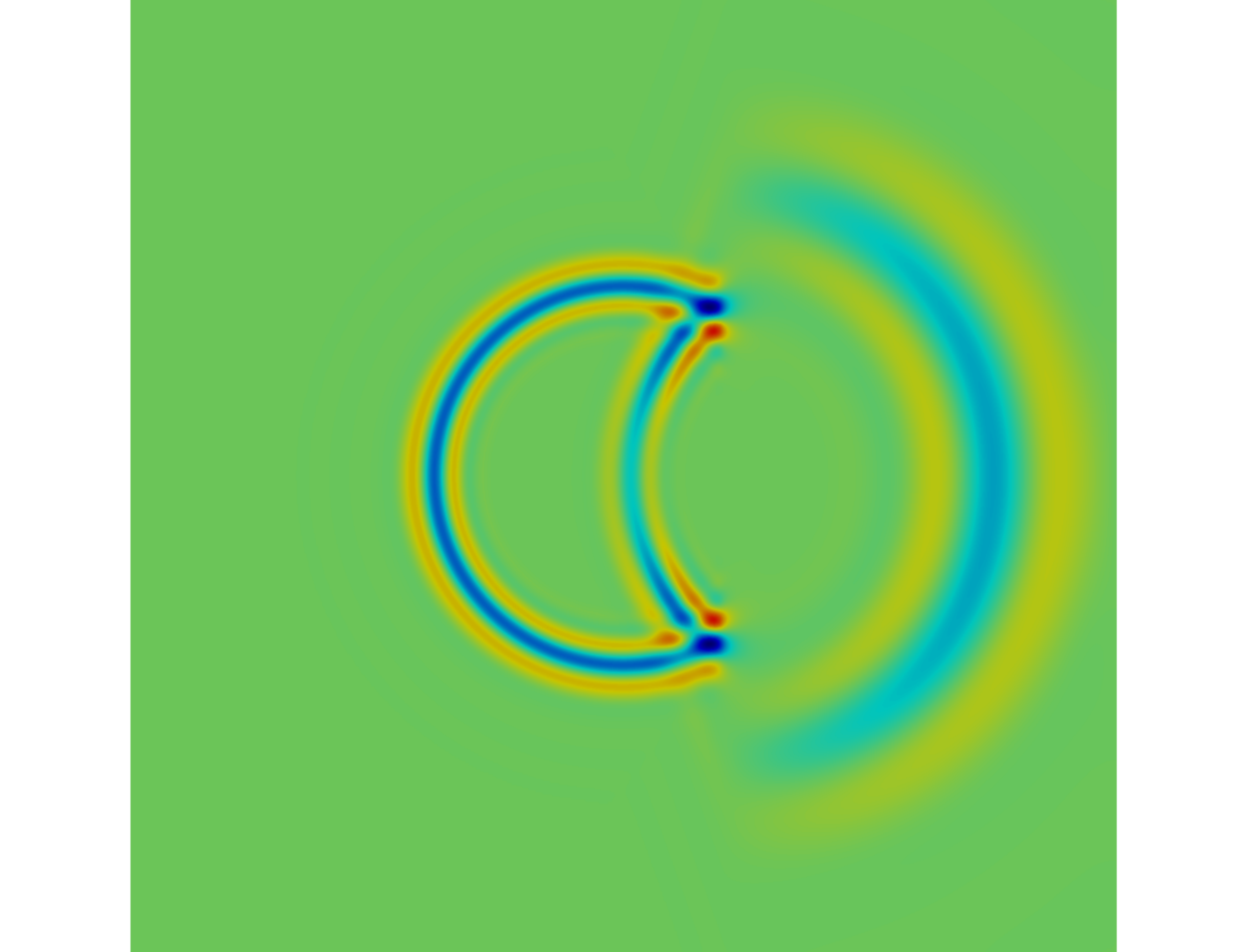}
 \subcaption{\small{$t=0.4$}}
\end{minipage}
 \caption{\small {Wave propagation with the full-tensor DG scheme through a heterogeneous medium, the pressure $v$ is shown at different times. 
 The experimental set-up is described in \emph{Test 3}, \S \ref{ss:test_1.3}.}}
 \label{fig:test13_snapshots}
\end{figure}

\subsection{Test 4: singular solution, heterogeneous medium}
\label{ss:test_1.4}

Consider the space--time domain $Q = (0,2)^2 \times (0,0.3)$. The wavespeed in $\Omega = (0,2)^2$ is the piecewise constant function
\[
    c(\bx) = \begin{cases}
              c_1 = 3,& x_1 >    1.2,\ x_2 >    1,\\
              c_2 = 1,& x_1 \leq 1.2,\ x_2 >    1,\\
              c_3 = 3,& x_1 \leq 1.2,\ x_2 \leq 1,\\
              c_4 = 1,& x_1 >    1.2,\ x_2 \leq 1.
             \end{cases}
\]
As the initial condition, we take a Gaussian wave given by
\[
    u_{0} = \exp(-\N{\bx - \bx_0}^2/\lambda^2), \;\;
    v_{0} = 0, \;\;
    \bsigma_{0} = - \nabla_{\bx} u,
\]
where $\bx_0 = (1,1.125)^{\top}$ and $\lambda = 0.01$. 
We consider homogeneous Dirichlet boundary conditions. 
The computations are performed with polynomial degrees $p_x^{v} = p_{x}^{\bsigma} = 2$ and $p_t^{v} = p_t^{\bsigma} = 1$, and 
the uniform temporal mesh has step size $h_t = 0.3\cdot2^{-4}$.

In this experiment, we use the same quasi-uniform spatial mesh as in Test 3, Figure~\ref{fig:test13_merged}. 
Additionally, we use the locally refined spatial mesh shown in Figure \ref{fig:test14}.
We have approximately $9.5$ million and $16.5$ million degrees of freedom in space--time for the quasi-uniform and locally refined meshes, respectively.
The size of each linear system is approximately $6\times10^5$ for the 
quasi-uniform case and $10^6$ for the locally refined case.
Snapshots of the solution at various times are shown in \autoref{fig:test14_snapshots}.
\begin{figure}[htb]
\centering
\includegraphics[width = 0.30\textwidth, clip, trim=0 0 0 0]{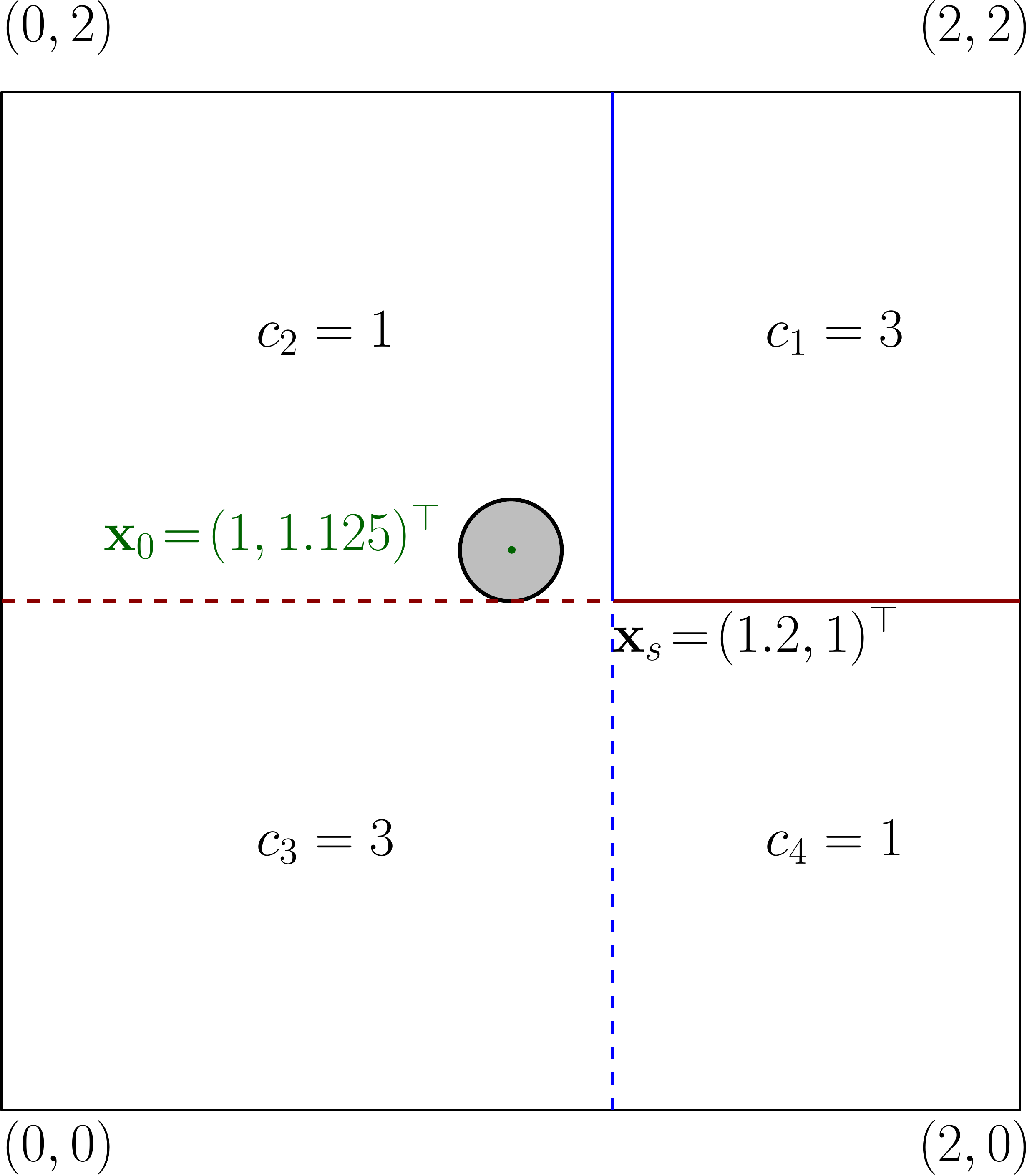}\qquad
\includegraphics[width = 0.38\textwidth, clip, trim=0 70 70 100]{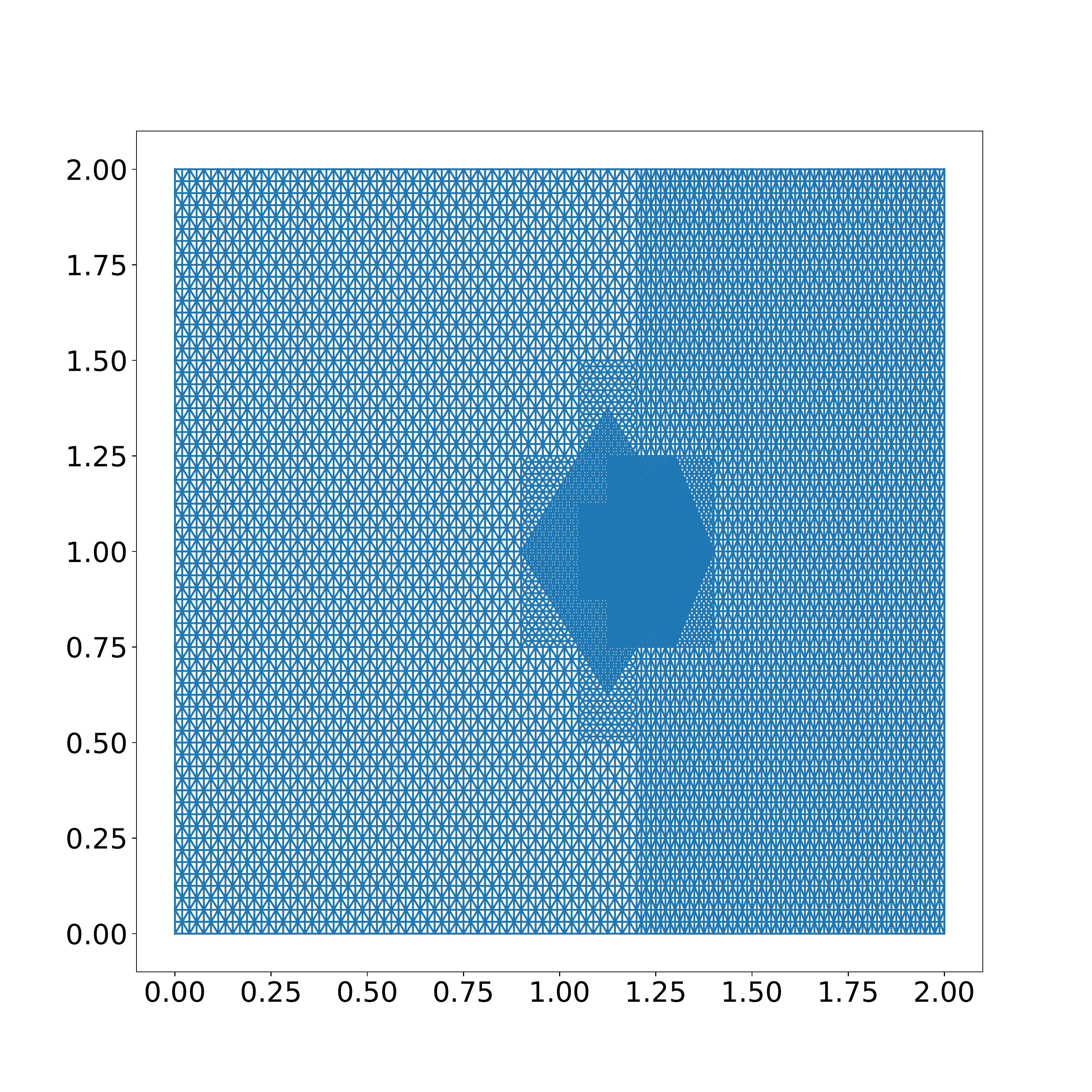}
\caption{\small {[left] Set-up of \emph{Test 4}, \S \ref{ss:test_1.4}:
the grey region indicates the initial Gaussian wave with centre $\bx_0 = (1,1.125)^{\top}$,
the material interfaces are indicated by the blue and brown lines.
[right] The locally refined bisection-tree spatial mesh conforming to the material interfaces, generated using Algorithm \ref{alg:bisection_meshGen} only for the material interface corner $\bx_s = (1.2,1)^{\top}$, with polynomial degree $p_{\bsigma}^{x} = 2$, cut-off radius $R_{c} = 0.392$, refinement weight $\delta = 0.4$ and mesh parameter $h_{\bx}^{\prime} = 0.0625$, which yields the number of refinements $J = 19$. The resultant mesh has maximum meshwidth $h_{\bx} \approx 0.0365$.}}
\label{fig:test14}
\end{figure}
\begin{figure}[htpb]
\begin{center}
\begin{minipage}{0.30\textwidth} \centering
 \includegraphics[width = \textwidth, clip, trim = 80 0 80 0]{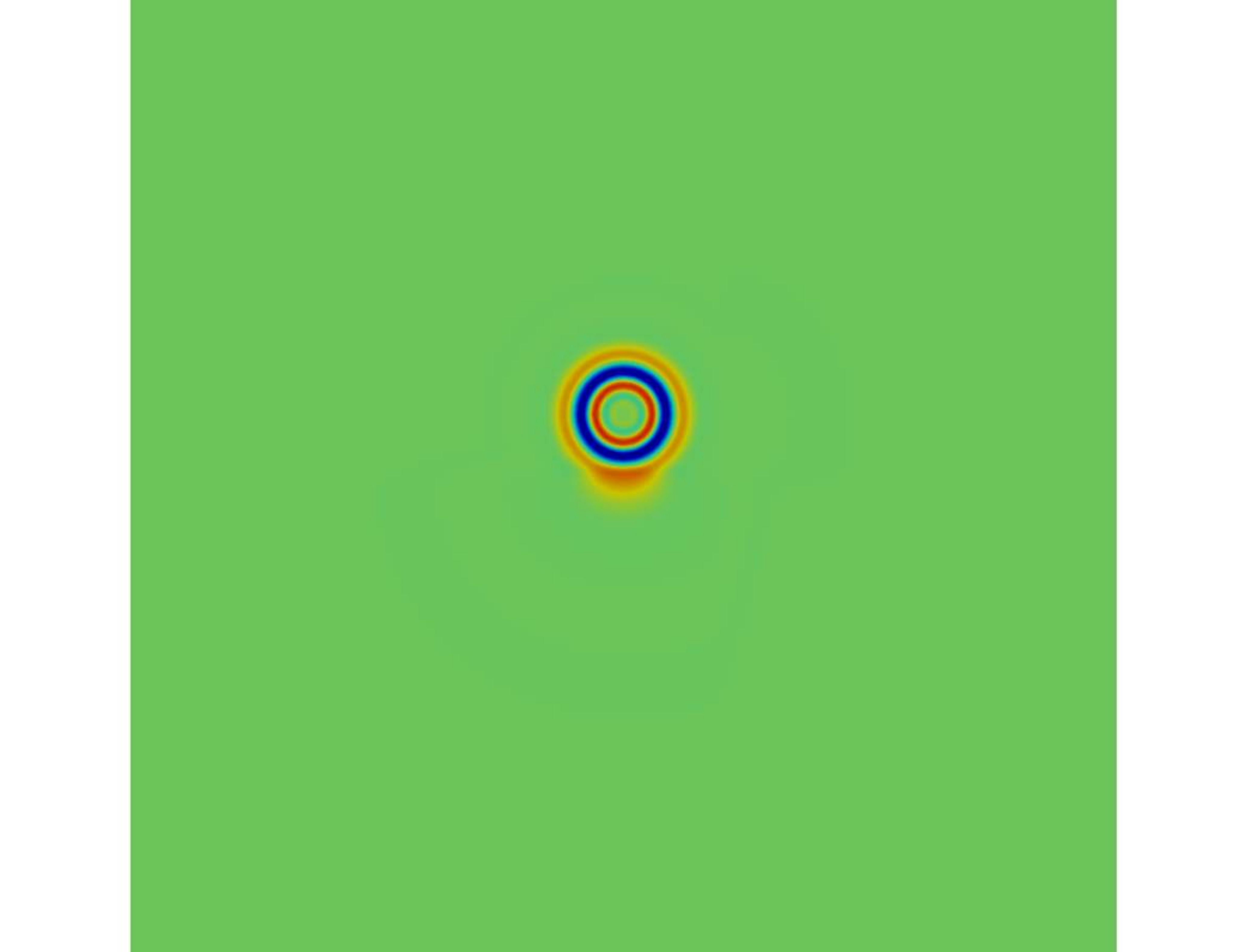}
 \subcaption{$t=0.1$}
\end{minipage}
\begin{minipage}{0.30\textwidth} \centering
 \includegraphics[width = \textwidth, trim = 80 0 80 0]{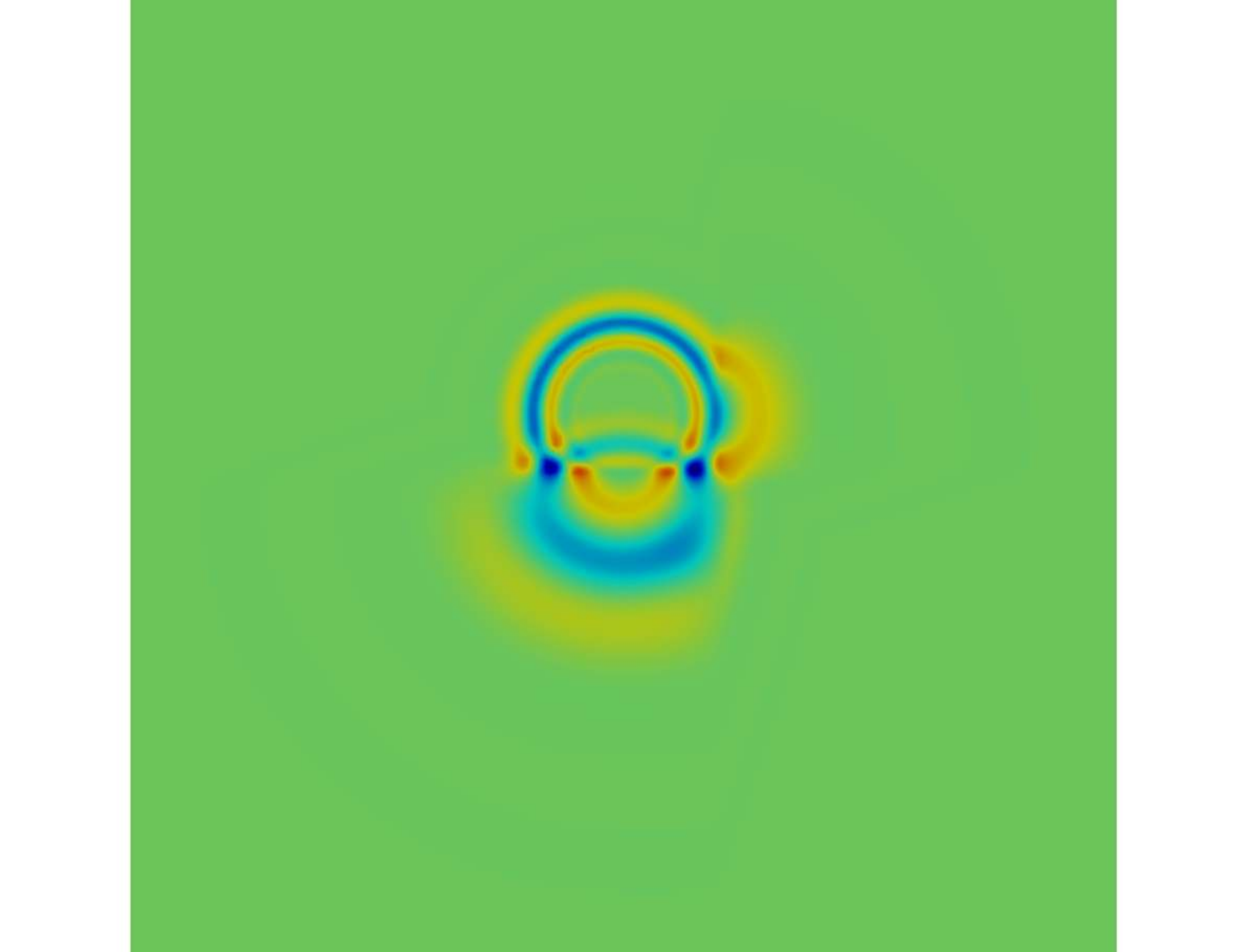}
 \subcaption{$t=0.2$}
\end{minipage}
\begin{minipage}{0.30\textwidth} \centering
 \includegraphics[width = \textwidth, trim = 80 0 80 0]{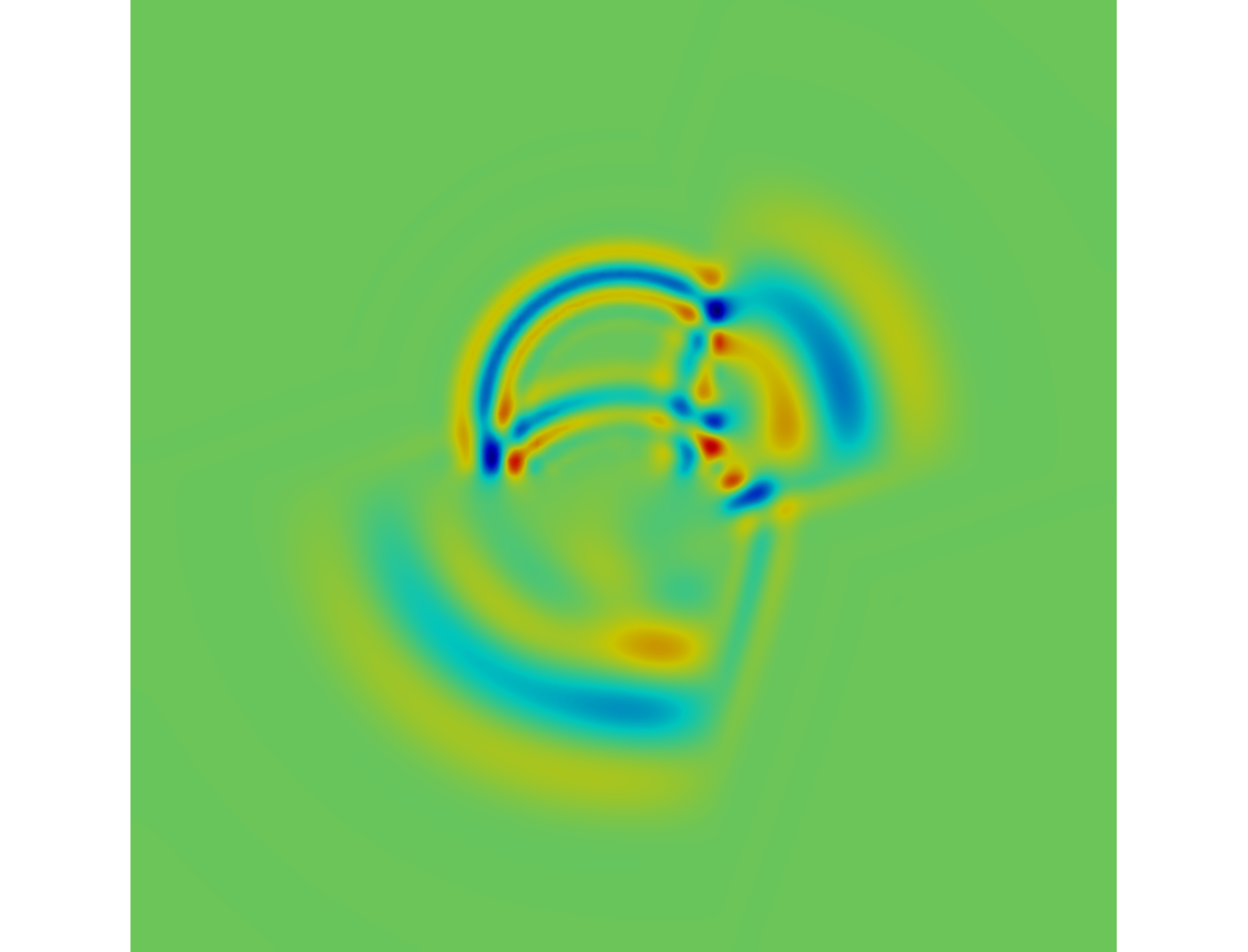}
 \subcaption{$t=0.3$}
\end{minipage}
\caption{\small {Wave propagation with the full-tensor DG scheme through a heterogeneous medium, the pressure $v$ is shown at different times. 
 The experimental set-up is described in \emph{Test 4}, \S\ref{ss:test_1.4}.
 The snapshots are computed with the locally refined mesh shown in Figure \ref{fig:test14}.}}
\label{fig:test14_snapshots}
\end{center}
\end{figure}

\section{Sparse space--time discretization}
\label{sec:SprseXT}
Based on the work of Bungartz, Griebel et al., 
see \cite{sparseBG} and the references therein, 
we develop now a \emph{sparse--tensor space--time} (``sparse-$xt$'' for short) 
approximation based on the above DG formulation and 
on \emph{anisotropic} tensorized finite element spaces in the space--time cylinder $Q$.
The sparse-$xt$ DG solution will be computed by (independent) numerical
solution of several \emph{anisotropic space--time} 
DG discretizations with different combinations of space and time steps,
given by the so-called \emph{combination formula}, equation~\eqref{eq:xtCombForm}.
Subsequently, 
these anisotropic, full--tensor space--time numerical approximations
will be combined into the sparse-$xt$ approximation.
As we shall show, this approximation will afford the 
same asymptotic order of convergence as the full-$xt$ approximation,
with the total number of degrees of freedom in the entire space--time
domain proportional to that of one time step on the finest spatial grid. 
\subsection{Combination formula}
\label{sec:CombForm}
To describe the sparse-$xt$ discretization, 
we shall require the following {notation}.
We denote the discretization ``levels'' (corresponding, roughly, to 
generations in bisection-tree refinements in the spatial domain 
$\Omega$ and the time-interval $I$)
by $l_x$ and $l_t \in \IN_0$, respectively.
We tag the spatial and temporal discretization levels in the index pairs
\begin{equation*}
\bll := (l_x, l_t)^{\top}\ \in \IN_0^2.
\end{equation*}
We introduce the total order $\abs{\bll} = l_x + l_t$.
The indices $l_x, l_t$ represent the resolution level in space and in time, 
respectively.
We let $L_x, L_t, L_{0,x}, L_{0,t}\in\IN_0$ be such that 
\begin{equation*}
 L_x - L_{0,x} = L_t - L_{0,t},
\end{equation*}
and define the vectors
\begin{align*}
  \Uu{L} = (L_x, L_t)^{\top}\ \in \IN_0^2, \quad 
  \Uu{L}_0 = (L_{0,x}, L_{0,t})^{\top}\ \in \IN_0^2,
 \end{align*}
and the set of admissible indices
\begin{equation}
\Xi_{(\Uu{L}, \Uu{L}_0)} 
:= 
\big\{\bll \in \IN_0^2 : l_x \geq L_{0,x},\ l_t \geq L_{0,t},\
\abs{\bll} \in \{L_x + L_{0,t},\ L_x + L_{0,t} - 1\}\big\}.
\label{eq:combForm_indexSet}
\end{equation}
The indices $L_x, L_t$ and $L_{0,x}, L_{0,t}$ 
represent maximum and minimum resolution levels in space and in time, respectively.

\begin{figure}[htb]\centering
\begin{tikzpicture}[scale=.7]
\draw[thick,->](0,0)--(7,0); \draw(7.5,0)node{$l_x$};
\draw[thick,->](0,0)--(0,7); \draw(-.5,7)node{$l_t$};
\draw[thick](0,5)--(5,0);
\draw[thick](0,6)--(6,0)--(6,-.2);
\draw(6,-.5)node{$L_{x}+L_{0,t}$};
\draw[thick,fill](3,2)circle(.15);\draw[thick,fill](4,2)circle(.15);
\draw[thick,fill](2,3)circle(.15);\draw[thick,fill](3,3)circle(.15);
\draw[thick,fill](1,4)circle(.15);\draw[thick,fill](2,4)circle(.15);\draw[thick,fill](1,5)circle(.15);
\draw[thick,dashed](1,-.2)--(1,6);\draw(1,-.5)node{$L_{0,x}$};
\draw[thick,dashed](-.2,2)--(6,2);\draw(-.6,2)node{$L_{0,t}$};
\draw[thick,dashed](4,-.2)--(4,2);\draw(4,-.5)node{$L_x$};
\draw[thick,dashed](-.2,5)--(1,5);\draw(-.6,5)node{$L_t$};
\end{tikzpicture}
\caption{The indices in the set $\Xi_{(\bL,\bL_0)}$ as in \eqref{eq:combForm_indexSet}, with $L_{0,x}=1$, $L_{0,t}=2$, $L_x=4$, $L_t=5$.}
\end{figure}
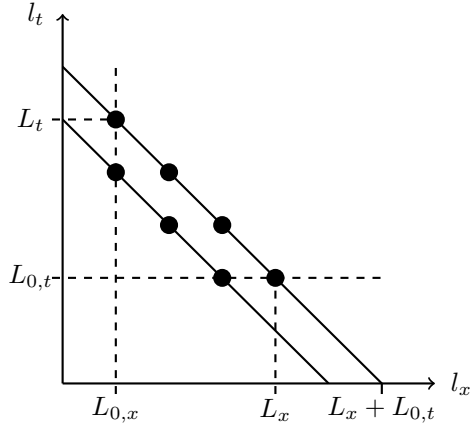

Let $h_{l_x,x}$, $h_{l_t,t} \in \IR^{+}$ and $h_{0,x}, h_{0,t} \in \IR^{+}$ 
be such that
\begin{align*}
 h_{l_x,x} =& 2^{-l_x} h_{0,x},\;\; h_{l_t,t} = 2^{-l_t} h_{0,t}.
\end{align*}
A mesh $\calT_{h}^{\bll}$ has 
meshwidth bound $h_{l_x,x}$ in the spatial variable 
and $h_{l_t,t}$ with respect to the time variable.
Assume that $(\calT_{h}^{\bll})_{l_x=L_{0,x}}^{L_x}$ for a fixed $h_{l_t,t}$, and 
vice-versa $(\calT_{h}^{\bll})_{l_t=L_{0,t}}^{L_t}$ for a fixed $h_{l_x,x}$, form a nested hierarchy of space--time meshes.
For 
$(v,\bsigma)\in L^2\big(I; H^{2,2}_\udelta\OO\big)\times L^2\big(I; H^{1,1}_\udelta\OO^2\big)$
and its discrete approximation 
$(v_h^{\bll},\bsigma_h^{\bll}) \in \bV_\bp(\calT_h^{\bll})$, 
define 
\begin{align*}
 \Uu{w} :=&\ (v,\bsigma^{\top})^{\top}, \quad 
 \Uu{w}_{\bll} := \ (v_h^{\bll},{\bsigma_h^{\bll}}^{\top})^{\top}.
\end{align*}

Using the notation and definitions above, we can write a sparse-$xt$ 
approximation $\hat{\Uu{w}}_{(\Uu{L}, \Uu{L}_0)}$ of $\Uu{w}$ 
using the combination formula \cite{sparseBG} in space and time dimensions as
\begin{equation}
\begin{aligned}
 \hat{\Uu{w}}_{(\Uu{L}, \Uu{L}_0)} 
 &= \sum_{\substack{\bll \in \Xi_{(\Uu{L}, \Uu{L}_0)}}} c_{\bll} \Uu{w}_{\bll}
 = (\hat{v}_{(\Uu{L}, \Uu{L}_0)}, \hat{\bsigma}_{(\Uu{L}, \Uu{L}_0)}^{\top})^{\top},
 \mathrm{where}\quad c_{\bll} = \begin{cases}
                                  +1,\ \mathrm{for } \abs{\bll} = L_x + L_{0,t}\\
                                  -1,\ \mathrm{for } \abs{\bll} = L_x + L_{0,t} - 1,
                                 \end{cases}
\end{aligned}
 \label{eq:xtCombForm}
\end{equation}
are the so-called ``combination coefficients''. 
We refer to $\Uu{w}_{\Uu{L}}$ as \emph{full-$xt$ DG numerical approximation}.
Then, a \emph{sparse-$xt$ DG numerical approximation} can be computed as follows:
\begin{enumerate}[label=(\alph*)]
\itemsep0em
 \item Given $\Uu{L}, \Uu{L}_0 \in \IN^2_0$, compute (in parallel) the space--time DG solution $\Uu{w}_{\bll}$
 for all $\bll \in \Xi_{(\Uu{L}, \Uu{L}_0)}$.
 \item Build the sparse-$xt$ numerical solution using the combination formula \eqref{eq:xtCombForm}.
\end{enumerate}

\subsection{Error vs.\ complexity analysis}\label{ss:errVsWork}
Given a maximal dyadic refinement level $L \in \IN$, 
fix the indices $L_{0,x} = L_{0,t} = 0$ and $L_x = L_t = L$.
We omit the $\Uu{L}_0$ vector in the sparse-$xt$ approximation for ease of notation.
Given $h_{0,x} = h_{0,t} = 0.5$ and $\bll \in \IN_0^2$, 
consider a uniform time mesh with $N_t = O((h_{l_t,t})^{-1})$ intervals
of uniform (time) step size $h_{l_t,t} = T/N_t$. 
We shall further assume that $h_{l_t,t} = O(2^{-l_t})$.
In the spatial domain $\Omega$,
we consider a regular partition of $\Omega$ into triangles 
of meshwidth $h_{l_x,x} = O(2^{-l_x})$. 
Then,
$h_{l,x} \simeq h_{l,t}$ holds for $l_x = l_t = l$, 
and in that case denote by $h_l$ the common spatial and temporal meshwidth.
Fix the polynomial degrees 
$p_{x,K}^{v} = p_{x,K}^{\bsigma} = p_x$ and $p_{t,K}^{v} = p_{t,K}^{\bsigma} = p_t$
for all $K \in \calT_h^{\bll}$.
Then,
$$
 M_{h}^{\bll} := \dim\big(\bV_\bp(\calT_h^{\bll})\big)
              = 6 (p_t+1) \frac{(p_x+1)(p_x+2)}{2} 2^{2l_x + l_t}
              = C(p_x, p_t)\cdot 2^{2l_x + l_t},
$$
is the number of degrees of freedom of the finite-dimensional space $\bV_\bp(\calT_h^{\bll})$.
Let $M_{\Uu{L}}^f$ and $M_{\Uu{L}}^s$ denote the number of degrees of freedom for 
full-$xt$ and sparse-$xt$ approximations, respectively. 

For a \emph{full-$xt$ approximation}, 
with equal resolution levels $L = l_x = l_t$ in space and time,
we have, as $L\to \infty$,
\begin{align*}
 M_{\Uu{L}}^f &:= M_{h}^{\Uu{L}} = C(p_x, p_t) \cdot 2^{3L} = O(h_L^{-3}).
\end{align*}
The overall number of degrees of freedom in a \emph{sparse-$xt$ approximation} 
equals asymptotically, as $L\to \infty$,
\begin{align*}
M_{\Uu{L}}^s = \sum_{\bll \in \Xi_L} M_h^{\bll}
&=C(p_x, p_t)\bigg(\sum_{l_x=0}^L 2^{2l_x+L-l_x}+\sum_{l_x=0}^{L-1} 2^{2l_x+L-l_x-1} \bigg)  \\
&= C(p_x, p_t) \left( 2^{L}\cdot (2^{L+1}-1) + 2^{L-1} \cdot (2^L-1) \right)\\
&= C(p_x, p_t)\cdot \frac{(5 - 3 \cdot 2^{-L})}{2}\cdot 2^{2L} 
< \frac{5}{2}\cdot C(p_x, p_t)\cdot 2^{2L} = O(h_L^{-2}).
\end{align*}
We consider fixed, and equal spatial and temporal discretization orders 
$p = \min{(p_x, p_t)} \geq 1$. 
Then, under the solution regularity assumptions of Proposition \ref{prop:errorBound}, 
we have for the full-tensor $xt$-DG scheme, as $L\to \infty$ (i.e., $h_L\to 0$)
\begin{align*}
 &\frac12 \N{c^{-1}(v-v\hp^{\Uu{L}})}_{L^2(\Omega\times\{t_n\})} 
    +\frac12\N{\bsigma-\bsigma\hp^{\Uu{L}}}_{L^2(\Omega\times\{t_n\})^2}\\
 & \qquad\qquad \qquad\qquad 
 \le \mathcal{E}^{ft}:=\abs{\vs-\vsh}\DGQn
   = O(h_L^{p+\frac12}) = 
   O([M_{\Uu{L}}^f]^{-\frac{p+1/2}{3}}).
\end{align*}
On the other hand, if we consider the error bound for the sparse-tensor $xt$-DG scheme 
based on the combination formula \eqref{eq:xtCombForm}, 
we expect the error to be bounded according to
\begin{align*}
   &\frac12 \N{c^{-1}(v-v\hp^{\Uu{L}})}_{L^2(\Omega\times\{t_n\})} 
    +\frac12\N{\bsigma-\bsigma\hp^{\Uu{L}}}_{L^2(\Omega\times\{t_n\})^2}\\
 & \qquad\qquad \qquad\qquad \le \mathcal{E}^{st} 
:=\abs{\vs-\vsh}\DGQn
= O(h_L^{p+\frac12}) = O([M_{\Uu{L}}^s]^{-\frac{p+1/2}{2}}).
\end{align*}
These asymptotic relations are valid in both cases, for the quasi-uniform and the corner-refined meshes.
This is demonstrated in the numerical experiments in \S \ref{ss:numexp_sparseXT}.

The sparse-tensor discretization requires considerably \emph{fewer degrees of freedom} 
than the standard, full-tensor scheme for comparable 
consistency order ($M^s_\bL=O(h_L^{-2})$ vs.\ $M^f_\bL=O(h_L^{-3})$),
under realistic solution regularity (allowing, in particular, for conical singularities
    in the spatial variable).

A further advantage of the combination formula is provided 
by the \emph{distribution of the sizes of the linear systems} to be solved, in parallel.
We recall that the numbers of degrees of freedom $M^f_\bL$ and $M^s_\bL$ calculated above 
are the sums of the dimensions of all systems to be solved.
The full-tensor $xt$-DG scheme requires the numerical solution of 
$O(h_L^{-1})=O(2^L)$ linear systems of size $O(h_L^{-2})=O(2^{2L})$ 
(i.e., one linear system solve per each time-step).
On the other hand, the sparse $xt$-DG approximation requires the solution of 
$\sum_{\bll\in\Xi_L}O(2^{l_t})=O(2^{L+1})$ many 
\emph{linear systems of different sizes} given by $O(2^{2l})$
for $l=0,\ldots, L$.
That is, only $O(1)$ systems of size $O(2^{2L})$ ($l_x=L$, $l_t=0$, 
i.e.\ finest mesh in space and coarsest in time),
twice as many systems of a quarter size $O(2^{2(L-1)})$ ($l_x=L-1$, $l_t=1$), \ldots, 
$O(2^L)$ systems of size $O(1)$ ($l_x=0$, $l_t=L$, i.e.\ coarsest mesh in space and finest in time).
Since $O(2^{2L})$ is the size of the linear system for the DG discretization of a two-dimensional elliptic problem on the finest spatial mesh, the computational effort of the sparse $xt$-DG scheme amounts to that of this elliptic problem.

\subsection{Numerical experiments}
\label{ss:numexp_sparseXT}

Our implementation to compute the sparse space--time solution is serial. 
However, it is important to note that the space--time combination formula \eqref{eq:xtCombForm} is 
\emph{inherently parallel}. 
All the solutions $\Uu{w}_{\bll}$, for $\bll \in \Xi_{(\Uu{L}, \Uu{L}_0)}$, 
can be computed simultaneously using multiple processors. 
Computing a linear combination of these anisotropically discretized (so-called ``detail'') solutions 
can be implemented as a reduction operation across processors.

We re-run the numerical experiments described in Test 1, \S\ref{ss:test_1.1}, and Test 2, \S\ref{ss:test_1.2}, 
for the sparse-$xt$ DG scheme. A comparison of the full-$xt$ and the sparse-$xt$ DG schemes, 
in terms of their accuracy versus the total number of degrees of freedom used, 
is shown in Figures \ref{fig:test21_FGvsSG12} and \ref{fig:test22_FGvsSG12}.

In particular, for $p = p_{x}^{v} = p_{x}^{\bsigma} = p_{t}^{v} = p_{t}^{\bsigma}\in\{1,2\}$ and $\alpha=\beta=1$ we observe in the first rows of Figure~\ref{fig:test21_FGvsSG12} (smooth solution, quasi-uniform mesh) and Figure~\ref{fig:test22_FGvsSG12} (singular solution, locally refined mesh) that the errors decay as $O([M_\bL^f]^{{-}\frac{p+1}3})$ for the full-tensor product space, and as $O([M_\bL^s]^{{-}\frac{p+1}2})$ for the sparse tensor product space (both are slightly better than what is expected from the theory).
Further experiments show that the errors for the two fields $v$ and $\bsigma$ decay with the same rate (cf.\ the first two rows of Figures 7--10 in the first arXiv version of the present  preprint). 
If $p_x^\bsigma$ is lowered to $p-1$ and the flux parameters are scaled as $\alpha^{-1}=\beta=h_{F_\bx}$, then the errors for $\bsigma$ reduce to 
$O([M_\bL^f]^{{-}\frac{p}3})$ and $O([M_\bL^s]^{{-}\frac{p}2})$, while those for $v$ behave more erratically;
error plots for this choice of the parameters are shown in the second rows of Figure~\ref{fig:test21_FGvsSG12} (smooth solution, quasi-uniform mesh) and Figure~\ref{fig:test22_FGvsSG12} (singular solution, locally refined mesh).

\begin{figure}[htb]
\begin{center}
 $p = p_{t}^{v} = p_{t}^{\bsigma} = p_{x}^{v} = p_{x}^{\bsigma}$, \ $\alpha = \beta = 1$\\
\begin{minipage}{0.5\textwidth}
\centering $p=1$
 \includegraphics[width = \textwidth, clip, trim=40 0 40 100]{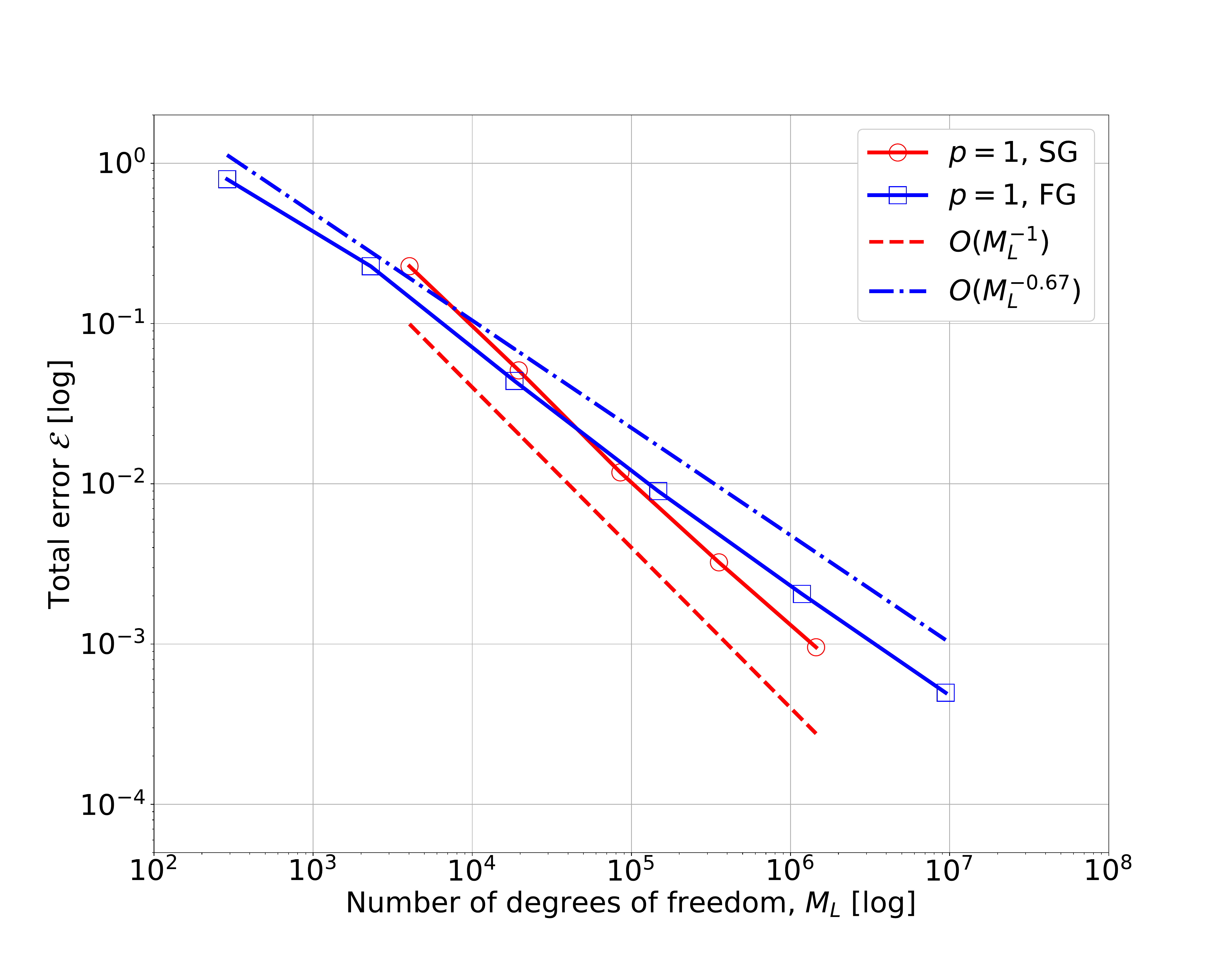}
\end{minipage}\begin{minipage}{0.5\textwidth}
\centering $p=2$
 \includegraphics[width = \textwidth, clip, trim=40 0 40 100]{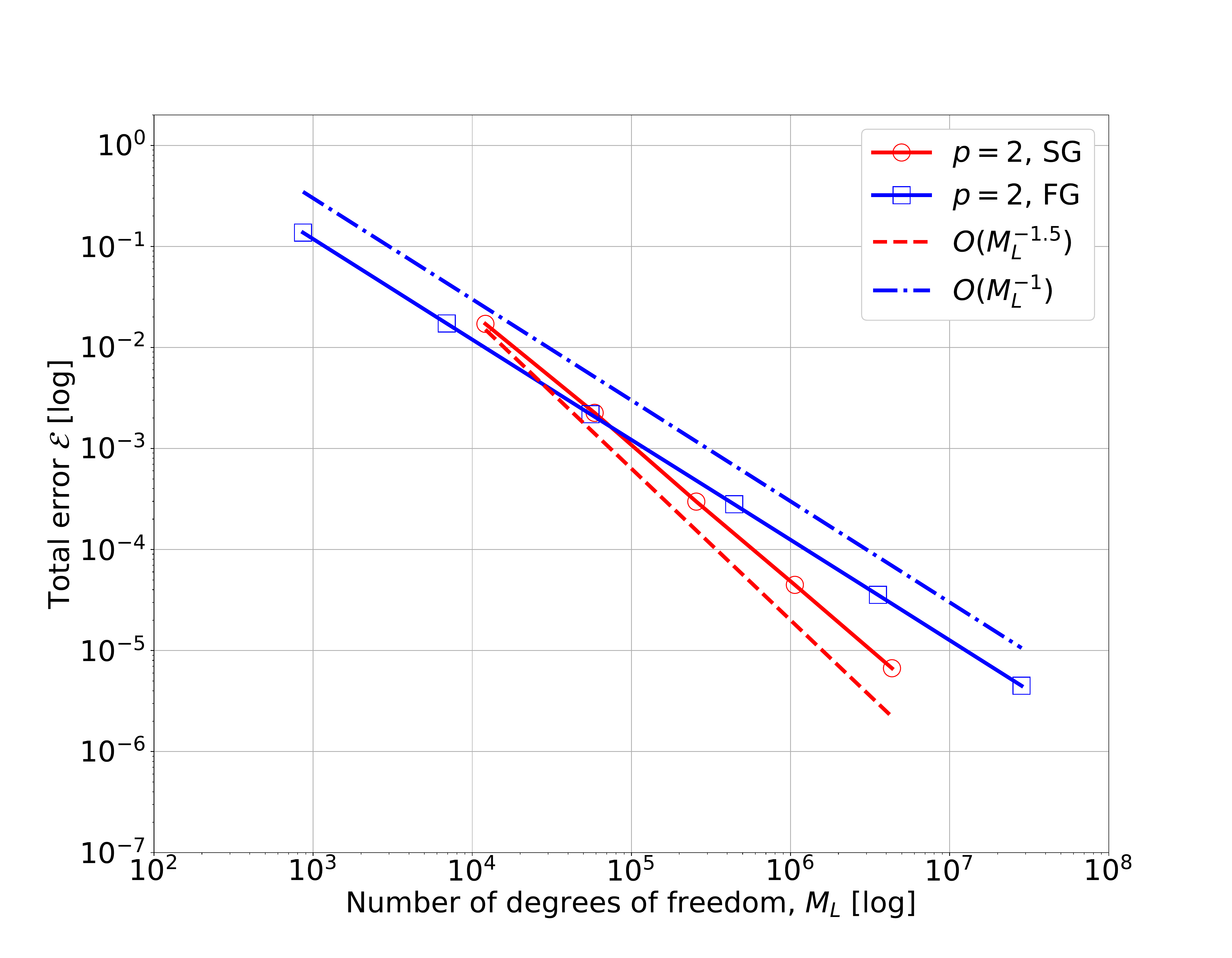}
\end{minipage}
$p = p_{t}^{v} = p_{t}^{\bsigma} = p_{x}^{v} = p_{x}^{\bsigma} + 1$, \ $\alpha^{-1} = \beta = h_{F_{\bx}}$\\
\begin{minipage}{0.5\textwidth}
\centering
 \includegraphics[width = \textwidth, clip, trim=40 40 40 100]{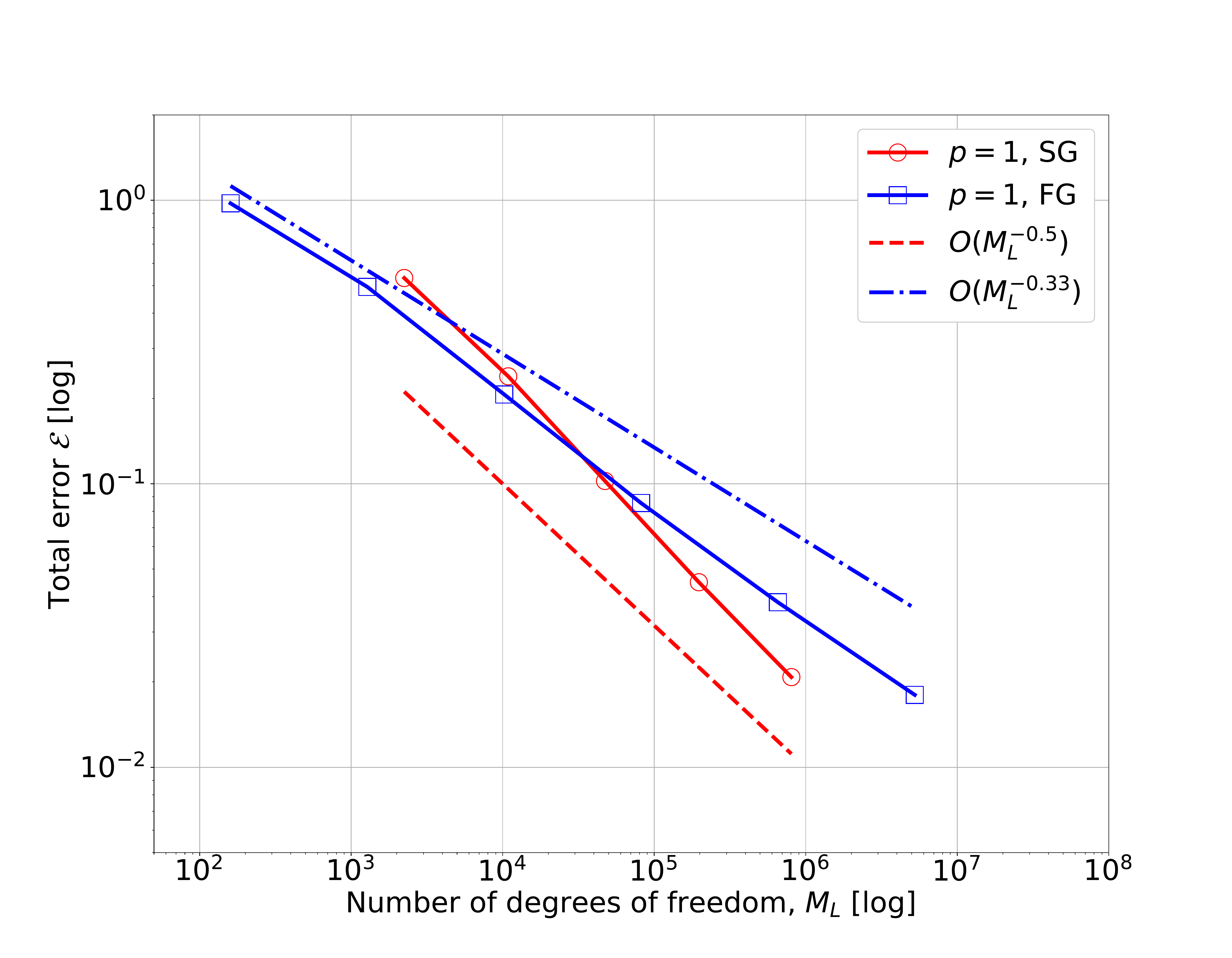}
\end{minipage}\begin{minipage}{0.5\textwidth}
\centering
 \includegraphics[width = \textwidth, clip, trim=40 40 40 100]{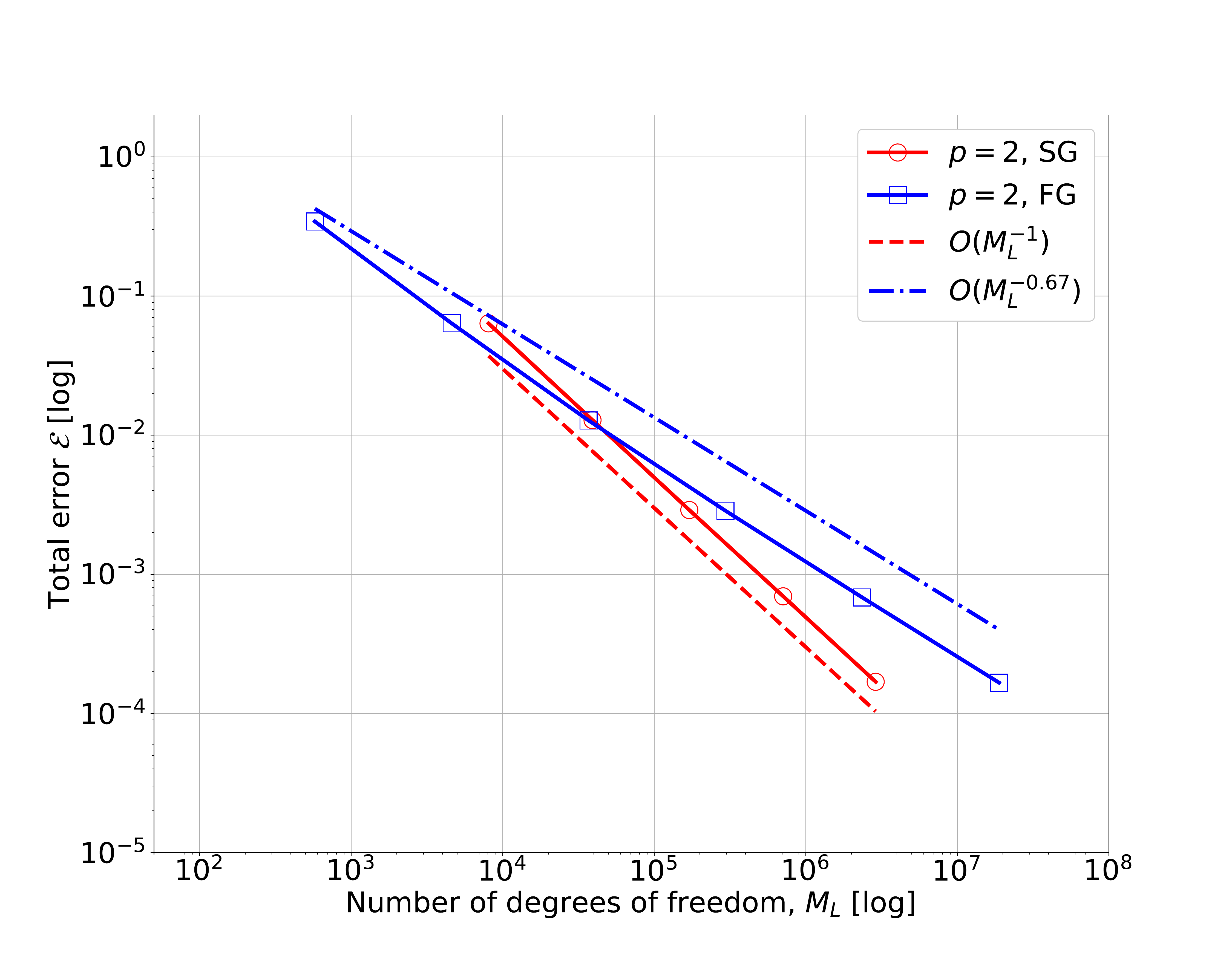}
\end{minipage}
\caption{\small {
Numerical results for \emph{Test 1}, as described in \S \ref{ss:test_1.1};
FG: \textbf{full-tensor} $xt$-DG, SG: \textbf{sparse-tensor} $xt$-DG.
The polynomial degrees and the stabilization parameters are chosen as $p = p_{t}^{v} = p_{t}^{\bsigma} = p_{x}^{v} = p_{x}^{\bsigma}$, $\alpha = \beta = 1$ (first row) and $p = p_{t}^{v} = p_{t}^{\bsigma} = p_{x}^{v} = p_{x}^{\bsigma} + 1$,  $\alpha^{-1} = \beta = h_{F_{\bx}}$ (second row), 
with $p=1$ (left column) and $p=2$ (right column).
The meshwidth for full-tensor $xt$-DG scheme is $h_{\bx} \approx h_t=2^{-l}$, for $l = 1,\ldots,6$.
The sparse-tensor $xt$-DG parameters are chosen as $h_{0,x} \approx h_{0,t} = 0.5$, $L_{0,x} = 0$, $L_{0,t} = 1$ and $L_x = L_t - 1 = L$, for $L = 1,\ldots,5$.
}}
\label{fig:test21_FGvsSG12}
\end{center}
\end{figure}

\begin{figure}[htb]
\begin{center}
 $p = p_{t}^{v} = p_{t}^{\bsigma} = p_{x}^{v} = p_{x}^{\bsigma}$, \  $\alpha = \beta = 1$\\
\begin{minipage}{0.5\textwidth}
\centering $p=1$
 \includegraphics[width = \textwidth, clip, trim=40 0 40 100]{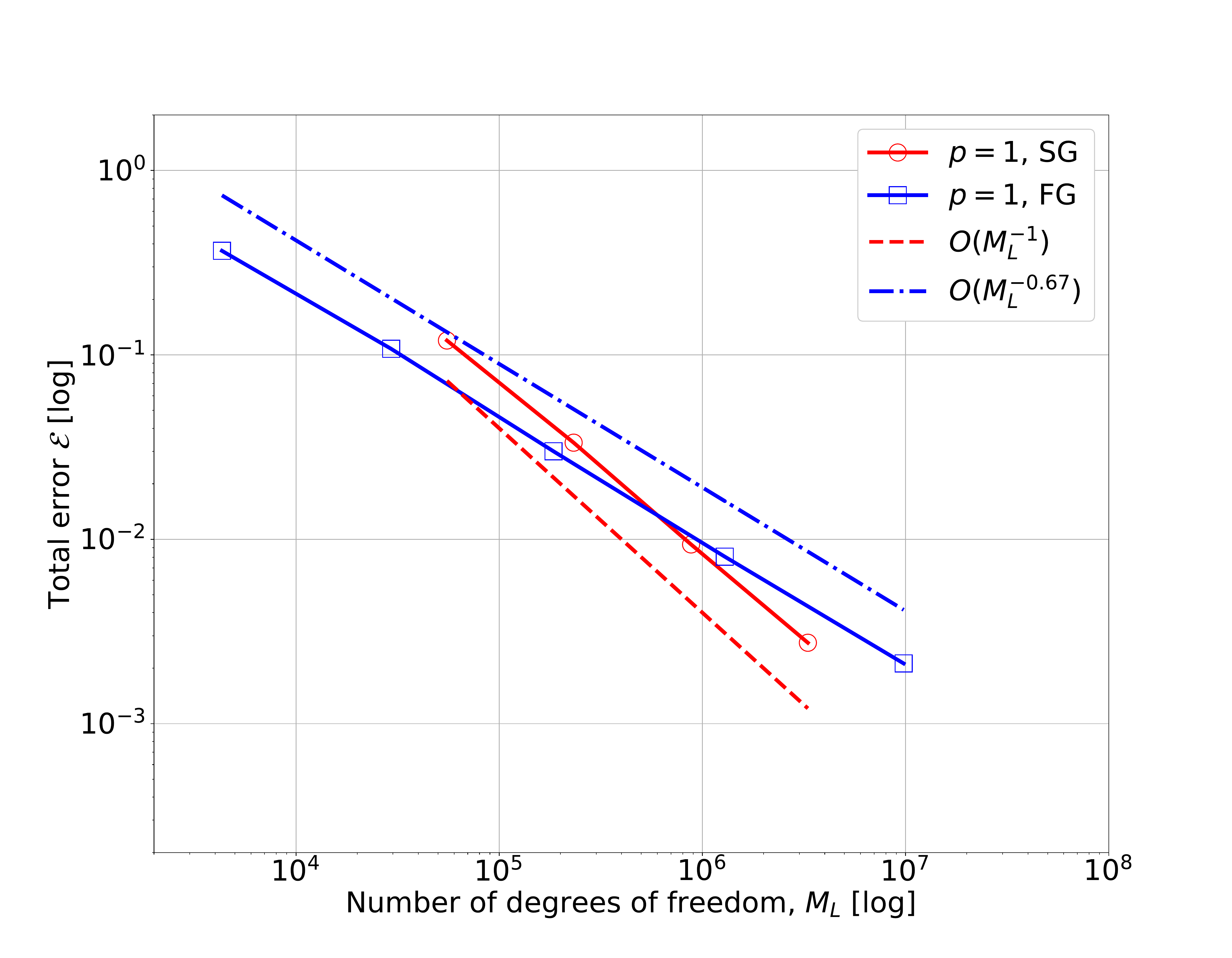}
\end{minipage}\begin{minipage}{0.5\textwidth}
\centering $p=2$
 \includegraphics[width = \textwidth, clip, trim=40 0 40 100]{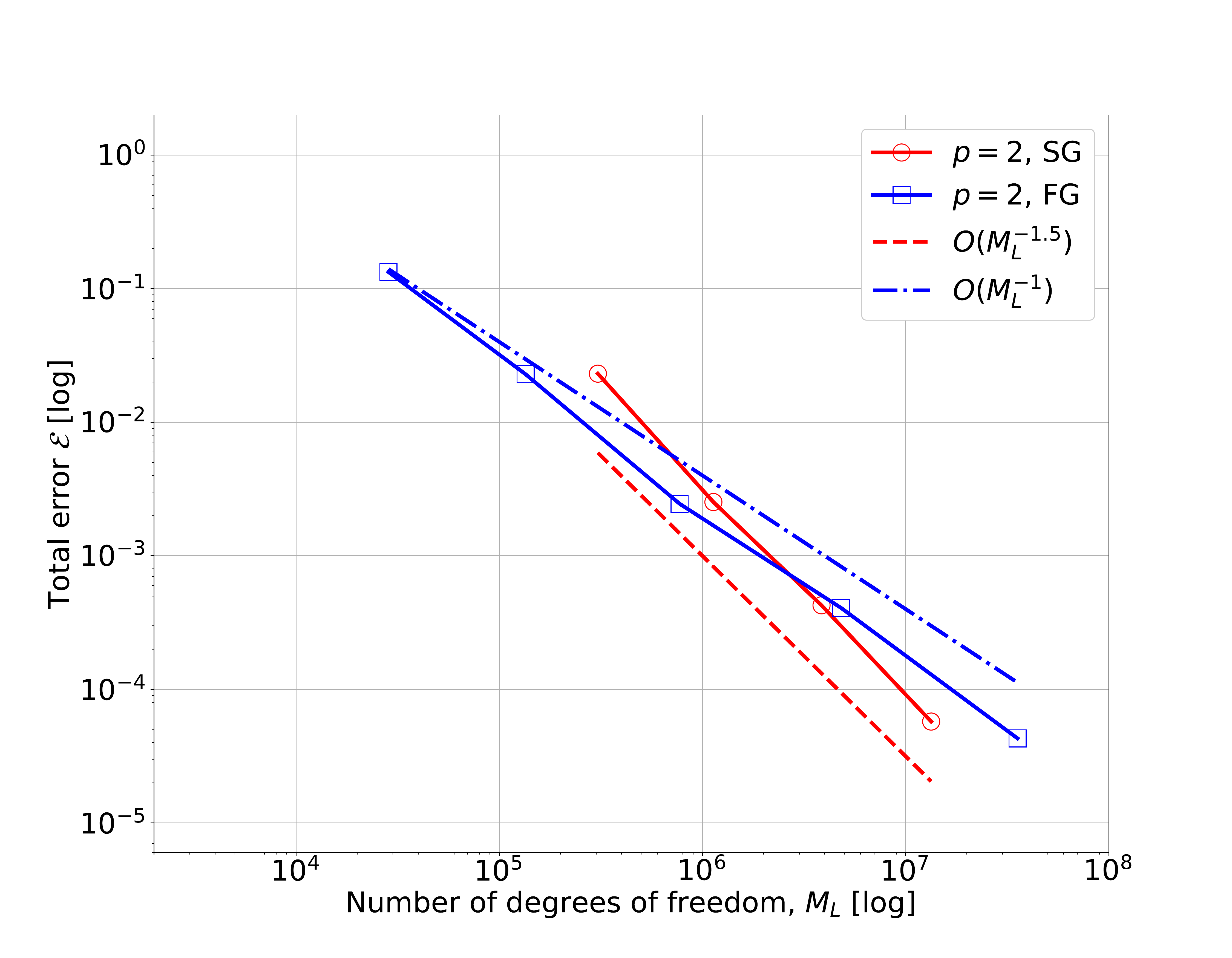}
\end{minipage}
$p = p_{t}^{v} = p_{t}^{\bsigma} = p_{x}^{v} = p_{x}^{\bsigma} + 1$, \ $\alpha^{-1} = \beta = h_{F_{\bx}}$\\
\begin{minipage}{0.5\textwidth}
\centering
 \includegraphics[width = \textwidth, clip, trim=40 40 40 100]{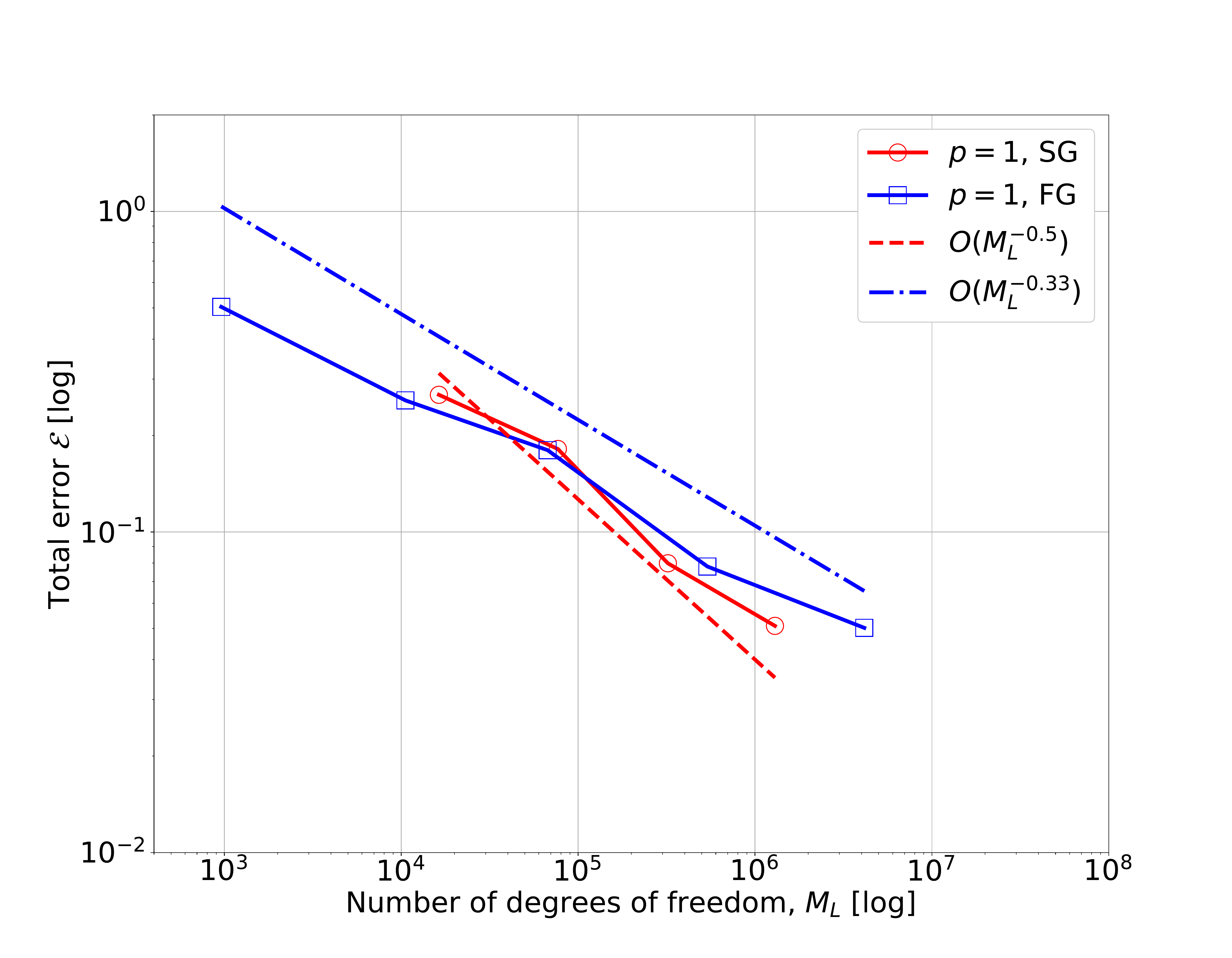}
\end{minipage}\begin{minipage}{0.5\textwidth}
\centering
 \includegraphics[width = \textwidth, clip, trim=40 40 40 100]{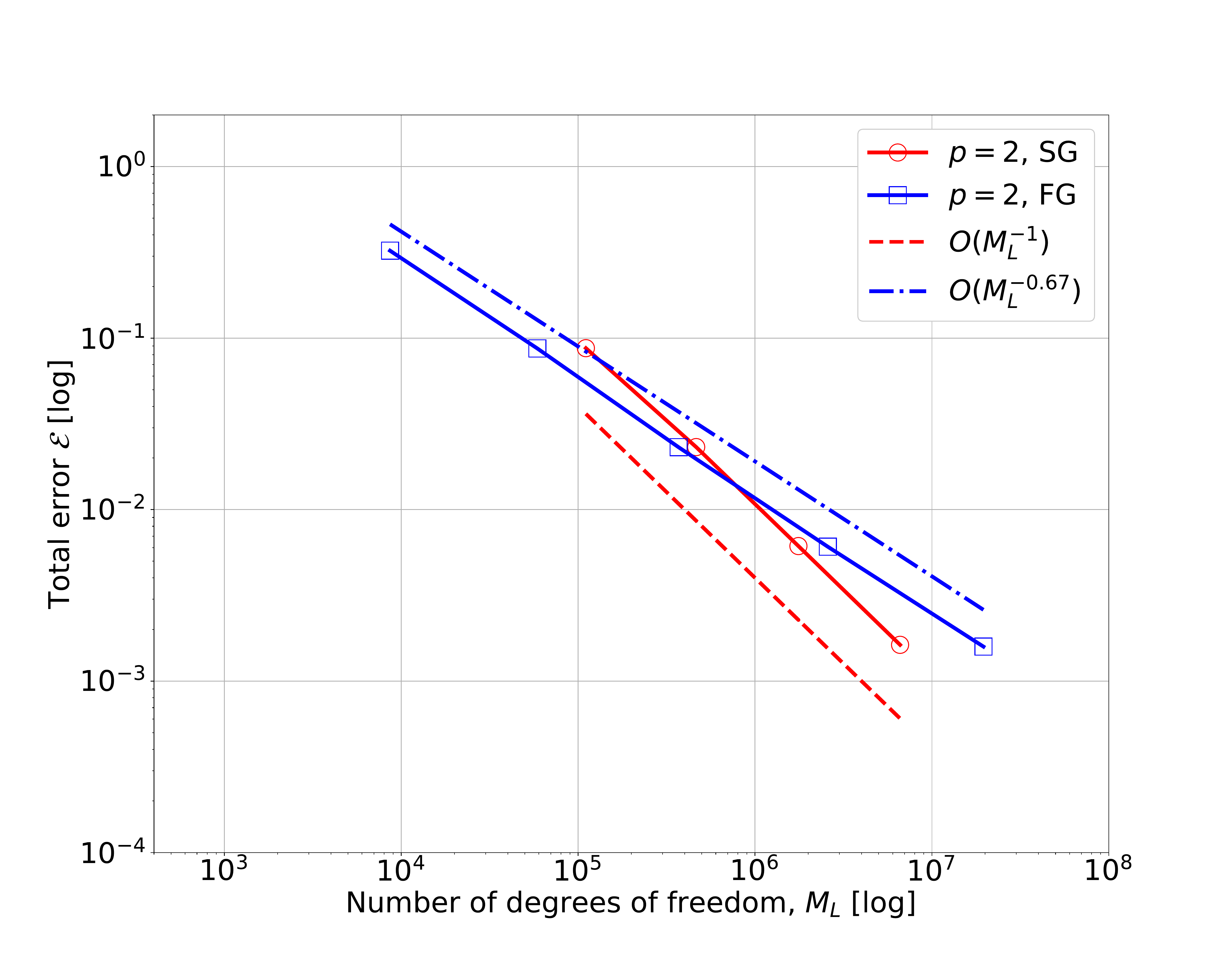}
\end{minipage}
\caption{\small {
Numerical results for \emph{Test 2}, as described in \S \ref{ss:test_1.2};
FG: \textbf{full-tensor} $xt$-DG, SG: \textbf{sparse-tensor} $xt$-DG.
The polynomial degrees and the stabilization parameters are chosen as $p = p_{t}^{v} = p_{t}^{\bsigma} = p_{x}^{v} = p_{x}^{\bsigma}$, $\alpha = \beta = 1$ (first row) and $p = p_{t}^{v} = p_{t}^{\bsigma} = p_{x}^{v} = p_{x}^{\bsigma} + 1$,  $\alpha^{-1} = \beta = h_{F_{\bx}}$ (second row), with $p=1$ (left column) and $p=2$ (right column).
The meshwidth for full-tensor $xt$-DG scheme is $h_{\bx} \approx h_t=2^{-l}$, for $l = 2,\ldots,6$.
 The sparse-tensor $xt$-DG parameters are chosen as $h_{0,x} \approx h_{0,t} = 0.25$, $L_{0,x} = 0$, $L_{0,t} = 1$ and $L_x = L_t - 1 = L$, for $L = 1,\ldots,4$.
}}
 \label{fig:test22_FGvsSG12}
\end{center}
\end{figure}

\section{Conclusions}
\label{s:Concl}
We analyzed a space--time discontinuous Galerkin discretization 
based on the variational formulation
from \cite{MoPe18} for the linear, acoustic wave equation.
We admitted, in particular, spatial mesh refinement
in polygonal domains and for multi-material configurations where
transmission conditions are imposed at material interfaces.
The proposed DG discretization is unconditionally stable with respect to space 
and time-step sizes which arise in such space--time partitions.
Using recent regularity results for the time-domain acoustic wave equation,
we proved that spatial mesh grading towards corners and multi-material
interface points can restore the maximal possible convergence rate.
Numerical experiments for piecewise-constant coefficient in polygonal domains with 
manufactured exact solutions, which exhibit conical diffraction singularities,
confirmed the convergence rate analysis.
For piecewise-homogeneous materials occupying
polygonal subdomains that are meshed exactly, the polynomial shape functions in the DG scheme give rise to cell contributions that can be evaluated exactly. 
For more general, (piecewise) smoothly varying coefficient functions, 
numerical integration has to be used in the setup of the proposed DG scheme. 
We expect that the presently developed stability bounds 
and the consistency analysis for the space--time DG scheme extend 
to variable coefficients subject to a quadrature consistency error analysis.

We also presented a novel, \emph{sparse space--time DG discretization} 
admitting local mesh refinement towards point singuarities in the spatial domain.
It is based on the \emph{combination} of several (possibly parallel) 
numerical solutions of the space--time DG scheme for a wide range of 
space and time-steps, 
including the mentioned spatial local mesh refinement towards conical singular points.
The unconditional stability of the full tensor space--time DG scheme is essential
here, as the combination technique will always access CFL-violating combinations
of space and time step sizes.
As we showed in numerical experiments, this strategy furnishes the maximal asymptotic convergence order afforded by the discretization, even in the presence of conical singularities, and with overall work scaling as the solution of one elliptic spatial solve (i.e., one implicit time step) on the finest spatial grid.

In numerical experiments (\S\ref{ss:numexp_sparseXT}) for model singular problems,
the sparse space--time DG scheme (with polynomial degree $p=1,2$)
actually outperformed the corresponding full--tensor
space--time DG scheme in terms of (relative) 
error versus total number of degrees of freedom, albeit only in the
error regime below one percent. 
A consistency error analysis of the sparse space--time DG approximation is under
investigation.

As mentioned, the unconditional stability of the considered space--time DG method is
important in view of widely varying spatial step-sizes due to mesh refinement near point singularities in the spatial domain.
This is also key in facilitating the sparse--tensor space--time DG algorithm: 
the terms which enter the combination formula
involve arbitrary combinations of spatial and temporal step sizes,
in particular, therefore, CFL-violating pairs. 
Furthermore, the global space--time nature of the discretization 
suggests that the presently proposed space--time DG method
is very robust with respect to the control of \emph{numerical dispersion}, 
which usually is significant in explicit time-marching schemes.
A detailed account of this will be the subject of future investigations.

Recently, there has been considerable interest in time-explicit schemes for acoustic and seismic wave propagation. 
Strong corner mesh refinement in the spatial domain as considered here will, however, 
generally entail rather \emph{anisotropic space--time slabs} 
which require prohibitively small explicit time-steps in elements
situated close to corners.
This could be addressed either by \emph{implicit local time-stepping} or by 
multi-level, explicit local time-stepping, as recently advocated e.g. in \cite{DiazGrote2015}; 
an error and stability analysis for explicit, local time-stepping of the presently proposed,
graded and bisection-tree meshes is yet to be developed.

The presently proposed space--time DG discretizations will extend also to other linear, second-order wave equations, 
in particular to elastic and electromagnetic waves, and to three-dimensional domains. 
These issues will be addressed in detail elsewhere.

\section*{Acknowledgments}\addcontentsline{toc}{section}{Acknowledgments}

The Authors would like to acknowledge the kind hospitality of the
Erwin Schr\"odinger International Institute for Mathematics and
Physics (ESI), where part of this research was developed under the
frame of the Thematic Programme {\it Numerical Analysis of Complex PDE
Models in the Sciences}.

P. Bansal acknowledges support from the project \emph{ModCompShock}, 
funded by the European Union's Horizon 2020 research and innovation programme
under the Marie Sklodowska-Curie grant agreement number 642768.
A.\ Moiola acknowledges support from GNCS-INDAM, from PRIN project
``NA\_FROM-PDEs'' and from MIUR through the ``Dipartimenti di
Eccellenza'' Programme (2018-2022)--Dept.\ of Mathematics, University
of Pavia.
I.\ Perugia has been funded by the Austrian Science Fund (FWF) through the projects F 65 and P 29197-N32.

\end{document}